\documentclass[a4paper, 11pt]{article}
\usepackage{anysize}
\marginsize{3,5cm}{2,5cm}{2,5cm}{2,5cm}
\usepackage{amssymb}
\usepackage{amsmath,amsbsy,amssymb,amscd}
\usepackage{t1enc}
\usepackage[cp1250]{inputenc}
\usepackage[all]{xy}
\usepackage{color}
\usepackage{amsfonts}
\usepackage{latexsym}
\usepackage{amsthm}
\usepackage{mathrsfs}
\usepackage{hyperref}
\usepackage{graphicx}
\usepackage{comment}

\usepackage{color}
\usepackage{caption}
\usepackage{subcaption}
\usepackage{enumerate}
\usepackage{xcolor}
\usepackage[draft]{todonotes}
\usepackage{float}
\usepackage{amsmath}
\usepackage{amsfonts}
\usepackage{amssymb}
\usepackage{amsthm}
\usepackage{mathtools}
\usepackage{graphicx}
\usepackage{fancyhdr}
\usepackage{geometry}
\usepackage{hyperref}
\usepackage{mathrsfs}
\usepackage{longtable}
\usepackage[nottoc,numbib]{tocbibind}
\usepackage{centernot}
\usepackage{color}
\usepackage{verbatim}
\usetikzlibrary{arrows, decorations.markings, positioning}
\usepackage{lipsum}
\usepackage[normalem]{ulem}

\allowdisplaybreaks

\makeatletter
\def\blfootnote{\xdef\@thefnmark{}\@footnotetext}
\makeatother

\makeatletter
\DeclareRobustCommand{\format@sec@number}[2]{{\normalfont\upshape#1}#2}

\makeatother

\def\e{\varepsilon}

\def\a{\alpha}
\def\b{\beta}

\def\R{\mathbb R}
\def\N{{\mathbb N}}

\def\Z{\mathbb Z}
\def\T{\mathbb T}

\def\({\biggl(}
\def\){\biggr)}
\def\<{\mathbf{\langle}}
\def\>{\mathbf{\rangle}}

\def\ent{\operatorname{ent}}
\def\uent{\overline{\operatorname{ent}}}
\def\lent{\underline{\operatorname{ent}}}
\def\top{\operatorname{top}}
\def\pol{\operatorname{pol}}

\newcommand{\Meng}[2]{\left\{#1\mathrel{}\middle|\mathrel{}#2\right\}}
\newcommand{\abs}[1]{\left\lvert#1\right\rvert}
\newcommand{\norm}[1]{\left\lVert#1\right\rVert}

\DeclareMathOperator*{\argmin}{arg\,min}

\def\Diff{\text{Diff}^{\infty}}

\newcommand{\freq}{\operatorname{freq}}
\newcommand{\sh}{\operatorname{sh}}

\newcommand{\vertiii}[1]{{\left\vert\kern-0.25ex\left\vert\kern-0.25ex\left\vert #1
    \right\vert\kern-0.25ex\right\vert\kern-0.25ex\right\vert}}

\numberwithin{equation}{section}

\theoremstyle{definition}
\newtheorem{definition}{Definition}[section]
\newtheorem{theorem}[definition]{Theorem}

\newtheorem{proposition}[definition]{Proposition}

\newtheorem{lemma}[definition]{Lemma}

\newtheorem{remark}[definition]{Remark}

\newtheorem{claim}[definition]{Claim}

\newtheorem{maintheorem}{Theorem}

\title{Slow entropy for some Anosov-Katok diffeomorphisms\blfootnote{Keywords: }
\blfootnote{2020 Mathematics Subject Classification: Primary; Secondary}}
\author{Shilpak Banerjee\footnote{Department of Mathematics, Indraprastha Institute of Information Technology Delhi (IIIT-Delhi), Okhla Industrial Estate, Phase-III, New Delhi, 110020, India, E-mail: banerjee.shilpak@gmail.com} \and Philipp Kunde\footnote{University of Hamburg, Department of Mathematics, Bundesstrasse 55, 20146 Hamburg, Germany, E-mail: pkunde.math@gmail.com P.K. acknowledges financial support from a DFG Forschungsstipendium under Grant No. 405305501.} \and Daren Wei\footnote{Einstein Institute of Mathematics, The Hebrew University of Jerusalem, Givat Ram. Jerusalem, 9190401, Israel, E-mail: Daren.Wei@mail.huji.ac.il  D.W. was supported ERC-2018-ADG project HomDyn.}}
\date{\today}

\begin{document}

\maketitle

\begin{abstract}
	The Anosov-Katok method is one of the most powerful tools of constructing smooth volume-preserving diffeomorphisms of entropy zero with prescribed ergodic or topological properties. To measure the complexity of systems with entropy zero, invariants like slow entropy have been introduced. In this article we develop several mechanisms facilitating computation of topological and measure-theoretic slow entropy of Anosov-Katok diffeomorphisms.
\end{abstract}


\section{Introduction}

Measure-theoretical and topological entropy serve as crucial tools in the study of complexity in dynamical systems in both measurable and topological category.  A dynamical system has positive entropy if and only if its orbit structure has exponential growth rates. However the celebrated results of Rokhlin \cite{Rokhlin} showed that zero-entropy systems are generic among automorphisms of a Lebesgue probability space equipped with a natural Polish topology. As a result, we know that low complexity systems are not rare and can exhibit rich dynamical phenomena. In order to study the dynamical systems with entropy zero, several different invariants have been introduced and studied: sequence entropy (\cite{Kushnirenko}, \cite{Goodman}), slow entropy (\cite{KT}), measure-theoretic and topological complexity (\cite{FerencziComplexity}, \cite{BHM}), entropy dimension (\cite{Carvalho}, \cite{DouHuangPark}), entropy convergence rate (\cite{Blume1}), scaled entropy(\cite{Vershik}), amorphic complexity (\cite{FGJ}). In particular, slow entropy enables us to measure precise complexities for both homogeneous systems (\cite{KatokKatokRodriguezHertz}, \cite{KanigowskiVinhageWei}, \cite{KanigowskiKundeVinhageWei}) and non-homogeneous systems (\cite{Kanigowski}, \cite{Adams}, \cite{BKW1}), and hence it has become a popular tool to estimate the complexity for zero-entropy systems. For a more detailed summary of the history, background and further references, we refer to the survey article \cite{KanigowskiKatokWei}.

In \cite{BKW1}, we studied the measure-theoretic slow entropy for some combinatorial constructions and addressed three problems stated in \cite{KanigowskiKatokWei}. More precisely, several results in \cite{BKW1} rely on an abstract \emph{Approximation by Conjugation method} to construct rigid automorphisms as well as transformations with good cyclic approximation for any prescribed value of lower or upper measure-theoretic slow entropy with respect to the polynomial scaling function $a^{\pol}_n(t)=n^t$. The Approximation by Conjugation method (also known as AbC or Anosov-Katok method) was introduced by D. Anosov and A. Katok in the highly influential paper \cite{AK} and has become a very powerful tool to construct smooth volume-preserving systems of entropy zero with prescribed ergodic or topological properties. We present the scheme of the AbC method in Section \ref{subsec:Anosov-Katok} and we refer to survey articles \cite{FK} as well as \cite{Ksurvey} for more detailed expositions of the AbC method and its wide range of applications.

However, the AbC transformations constructed in \cite{BKW1} are \emph{not} necessarily smooth. To guarantee smoothness further growth conditions have to be posed on the parameter sequences within the AbC constructions. In this paper we present several methods describing how to compute the measure-theoretic and topological slow entropy of \emph{smooth} AbC diffeomorphisms. Hereby, we provide another answer to Problem 6.3.1 in \cite{KanigowskiKatokWei}.

When working with slow entropy an important first step is to determine the scale which describes the growth rates and also distinguishes different systems. For instance, for smooth flows on surfaces the growth rates are $n^t$ and $n(\log n)^t$, depending on whether the singularities of the flow are degenerated or not, respectively \cite{Kanigowski}. For abelian unipotent $\R^k$ actions the correct family of scales to choose is the polynomial scale $n^t$ (\cite{KanigowskiVinhageWei}, \cite{KanigowskiKundeVinhageWei}). In our investigation it turns out that the choice of an appropriate scale depends on the regularity of the AbC diffeomorphisms. While for any fixed $k \in \N$ we construct $C^k$ diffeomorphisms with positive upper slow entropy with respect to the polynomial scale $a^{\pol}_n(t)=n^t$, we have to use intermediate scales between logarithmic $a^{\ln}_n(t)=(\ln(n))^t$ and polynomial to see non-trivial upper slow entropy for $C^{\infty}$ AbC diffeomorphisms.  We introduce these intermediate scaling functions in Section \ref{subsec:scaling}.  For a given dynamical system, we say its topological (measure-theoretic) slow entropy is in polynomial (or logarithmic, intermediate) scale if its polynomial (or logarithmic, intermediate) topological (measure-theoretic) slow entropy is a finite positive real number. 

One of the key ingredients in our method are precise norm estimates on the conjugation maps in our AbC constructions. In the literature, such constructions are called \emph{quantitative version} of the AbC method (see e.g. \cite[section 2]{Ksurvey}) which was initiated by B. Fayad and M. Saprykina in \cite{FS}. In case of the disc $\mathbb{D}^2$ or the annulus $\mathbb{S}^1 \times [0,1]$ the constructions in that paper provide for each Liouville number\footnote{We recall that an irrational number $\alpha$ is called \emph{Liouville} if for every $C>0$ and every positive integer $n$ there are infinitely many pairs of integers $(p,q)$ with $q>1$ such that $0<\abs{\alpha - \frac{p}{q}}<\frac{C}{q^n}$. If an irrational number is not Liouville, it is called \emph{Diophantine}} $\alpha$ a weakly mixing area-preserving diffeomorphism whose restriction to the boundary is the rotation by $\alpha$. Combining this result with ``Herman's last geometric theorem'' \cite{HermansLast} one obtains the striking dichotomy that an irrational number $\beta$ is Diophantine if and only if there is no ergodic $C^{\infty}$-diffeomorphism of the disc whose restriction to the boundary has rotation number $\beta$. Inspired by the constructions in \cite{FS}, we compute measure-theoretic as well as topological slow entropy for some AbC constructions of weakly mixing diffeomorphisms:
\begin{maintheorem}[Theorem \ref{thm:weakMixingTopInter}, \ref{thm:weakMixingTopIn}, \ref{thm:weakMixingTopPol}, \ref{thm:weakMixingMeaInter}] Let $M$ be the disc $\mathbb{D}^2$, or the annulus $\mathbb{A}=\mathbb{S}^1 \times [0,1]$ or the torus $\mathbb{T}^2$. Then,
\begin{itemize}
    \item there exist weakly mixing $C^{\infty}$ AbC diffeomorphisms on $M$ such that their upper topological slow entropy are in intermediate scale\footnote{ Intermediate scale refers to a scale we introduce in section \ref{subsec:scaling} based on the inverse gamma functions. It has speed between the logarithmic and polynomial scales}.
    \item there exist weakly mixing $C^{\infty}$ AbC diffeomorphisms on $M$ such that their upper topological slow entropy are in logarithmic scale.
    \item for every $k\in \N$ there also exists a weakly mixing $C^k$ AbC diffeomorphism such that its upper topological slow entropy is in polynomial scale.
    \item there exists a weakly mixing $C^{\infty}$ AbC diffeomorphism such that its upper measure-theoretic slow entropy is in intermediate scale.
    \item there exists a weakly mixing $C^{\infty}$ AbC diffeomorphism such that its upper measure-theoretic slow entropy is in logarithmic scale.
\end{itemize}
\end{maintheorem}

The quantitative version of the AbC method is also used to find non-standard smooth realizations of Liouville circle rotations. While the non-standard realizations in \cite[section 6]{AK} did not allow control over the rotation number, it is shown in \cite{FSW} that for every Liouville number $\alpha$ there exists an ergodic $T\in \Diff(M,\mu)$ measure-theoretically isomorphic to the circle rotation by $\alpha$. In case of the torus $\T^d$, $d\geq 2$, the result can be strengthened to obtain a uniquely ergodic $C^{\infty}$-diffeomorphism $T$ \cite[Theorem 2]{FSW}. In both general and uniquely ergodic setting, we compute the topological slow entropy for such  AbC diffeomorphisms measure-theoretically isomorphic to a circle rotation:

\begin{maintheorem}[Theorem \ref{thm:untwistedTopInter}, \ref{thm:untwistedTopIn}, \ref{thm:untwistedTopPol}]\label{thm:mainUniquelyErgodic}
Let $M$ be $\mathbb{D}^2$, $\mathbb{A}$ or $\mathbb{T}^2$. Then,
\begin{itemize}
    \item there exist $C^{\infty}$ AbC diffeomorphisms  on $M$ isomorphic to an irrational translation of the circle such that their upper topological slow entropy are in intermediate scale.
    \item there exist $C^{\infty}$ AbC diffeomorphisms  on $M$ isomorphic to an irrational translation of the circle such that their upper topological slow entropy are in logarithmic scale.
    \item for every $k\in \N$ there also exists a $C^k$ AbC diffeomorphism  on $M$ isomorphic to an irrational translation of the circle such that its upper topological slow entropy is in polynomial scale.
\end{itemize}
\end{maintheorem}

\begin{maintheorem}[Theorem \ref{thm:uniquelyErgodicTopInter}, \ref{thm:uniquelyErgodicTopIn}, \ref{thm:uniquelyErgodicTopPol}] \label{theo:uniquelyErgodicShort}
Let $M$ be $\mathbb{T}^2$. Then,
\begin{itemize}
    \item there exist uniquely ergodic $C^{\infty}$ AbC diffeomorphisms  on $M$ isomorphic to an irrational translation of the circle such that their upper topological slow entropy are in intermediate scale.
    \item there exist uniquely ergodic $C^{\infty}$ AbC diffeomorphisms on $M$ isomorphic to an irrational translation of the circle such that their upper topological slow entropy are in logarithmic scale.
    \item for every $k\in \N$ there also exists an uniquely ergodic  $C^k$ AbC diffeomorphism on $M$ isomorphic to an irrational translation of the circle such that its upper topological slow entropy is in polynomial scale.
\end{itemize}
\end{maintheorem}

Since transformations isomorphic to a translation on a compact group have measure-theoretic slow entropy $0$ with respect to every scaling function \cite[Proposition 3]{FerencziComplexity}, these diffeomorphisms provide examples for the failure of a variational principle for slow entropy. More precisely, since the uniquely ergodic diffeomorphism $T$ from Theorem \ref{theo:uniquelyErgodicShort} is measure-theoretically isomorphic to a circle rotation, it satisfies $\ent^{\mu}_{a_n(t)}(T)=0$ with respect to its unique invariant probability measure $\mu$ and every scaling function $\{a_n(t)\}_{n\in \N,t>0}$. On the other hand, we found explicit scaling functions with $\uent^{\top}_{a_n(t)}(T)>0$. Hence, the diffeomorphism $T$ is an example that the variational principle does not hold for slow entropy. This has already been observed in \cite[Appendix A.2.]{KanigowskiVinhageWei}. While in \cite{KanigowskiVinhageWei} just the existence of a scaling function with failure of variational principle could be shown, we provide counterexamples for specific scaling functions. In fact, we obtain the following corollary from Theorem \ref{theo:uniquelyErgodicShort}:
\begin{maintheorem}
For every $k\in\mathbb{N}\cup\{+\infty\}$, there exist $C^k$ diffeomorphisms such that slow entropy variational principle fails at scale $a_m^{\operatorname{int1,5}}(t)$ and logarithmic scale $a^{\ln}_m(t)$, respectively. Moreover, for every $k\in \N$, there exist $C^k$ diffeomorphisms such that slow entropy variational principle fails at polynomial scale $a_m^{\pol}(t)$.
\end{maintheorem}

Our last result is a general upper bound for the complexity of $C^{\infty}$ AbC diffeomorphisms:
\begin{maintheorem}\label{thm:vanishSuperSmooth}
For any $C^{\infty}$ AbC diffeomorphism $T$, its upper measure-theoretic slow entropy is always zero at polynomial scale.
\end{maintheorem}

While we have already shown that there cannot be a $C^{\infty}$ AbC diffeomorphism with positive measure-theoretic polynomial entropy, it is an interesting question if there are $C^{\infty}$ AbC diffeomorphisms with positive topological polynomial entropy.

\paragraph{Plan of the paper:} In Section \ref{sec:prelim}, we provide basic definitions and properties of measure-theoretic slow entropy, topological slow entropy, AbC constructions, properties of scaling functions and some simple estimates of complexity of AbC diffeomorphisms. In particular, we show in Subsection \ref{subsec:ArbSlow} that for any given scaling function $\{a_n(t)\}_{n\in \N,t>0}$ one can construct a $C^{\infty}$ AbC diffeomorphism $T$ with $\ent^{\lambda}_{a_n(t)}(T)=0$, where $\lambda$ denotes Lebesgue measure. In Section \ref{sec:untwistedAbC}, we first provide a specific AbC construction of a diffeomorphism measure-theoretically isomorphic to an irrational circle rotation. This construction is a slight modification of the one from \cite{FSW} in order to alleviate our calculations. Then we prove several estimates for cardinality of maximal separated sets and cardinality of minimal covering sets, which give estimates of upper topological slow entropy of the constructed AbC diffeomorphisms at several different scales. In Section \ref{sec:uniquelyErgodicAbC}, we further modify the AbC construction to obtain a uniquely ergodic diffeomorphism on $\T^2$. Similar to before we deduce estimates for cardinality of maximal separated sets and cardinality of minimal covering sets which yield the value of upper topological slow entropy of these AbC diffeomorphisms. We start Section \ref{sec:weakMixingAbC} by presenting a construction of weakly mixing AbC diffeomorphisms inspired by the examples in \cite{FS}. As before, this is followed by estimates of cardinality of maximal separated sets and cardinality of minimal covering sets in order to compute the topological slow entropy. By careful combinatorial estimates we also obtain the upper measure-theoretic slow entropy of the given weakly mixing AbC diffeomorphisms in Subsection \ref{subsec:UpperMeasure}.

\section{Preliminaries}\label{sec:prelim}

\subsection{Topological and measure-theoretic slow entropy}
In this section, we introduce the topological and measure-theoretic slow entropy for homeomorphisms and invertible measure-preserving transformations, respectively.
\subsubsection{Topological slow entropy}
Suppose $T$ is a homeomorphism from a locally compact metric space $(X,d)$ to itself and $K\subset X$ is a compact subset.
Then the cardinality of maximal separated sets and minimal covering set can be defined as follows:
\begin{definition}[Cardinality of maximal separate sets and minimal covering sets]
Let $n\in\mathbb{Z}^+$ and $x,y\in X$, we define the Bowen metric as
$$d_n^T(x,y)=\max_{0\leq i\leq n-1}\{d(T^ix,T^iy)\}.$$
Then we define $N_{d_n^T,K}(\epsilon)$ as the minimal number of $(\e,n)$-Bowen balls required to cover $K$ and $S_{d_n^T,K}(\epsilon)$ as maximal number of possible disjoint $(\e,n)$-Bowen balls with centers in $K$.
\end{definition}

With the help of cardinality of maximal separated sets and minimal covering sets, the topological slow entropy can be introduced as follows:
\begin{definition}[Upper topological slow entropy]
Suppose $a_n(t)$ is a family of positive sequences increasing in $n$ and monotonically increasing in $t$, then the upper topological slow entropy of $T$ with respect to $a_n^t$ is defined as
\begin{equation}
\uent^{\top}_{a_n(t)}(T)=\sup_K\lim_{\epsilon\to0}N(\epsilon,K)=\sup_K\lim_{\epsilon\to0}S(\epsilon,K),
\end{equation}
where $N(\epsilon,K)=\left\{
                       \begin{array}{ll}
                         \sup N_1(\epsilon,K), & \hbox{if $N_1(\epsilon,K)\neq\emptyset$;} \\
                         0, & \hbox{if $N_1(\epsilon,K)=\emptyset$.}
                       \end{array}
                     \right.
$ with
\begin{equation}\label{eq:minimalCovering}
N_1(\epsilon,K)=\{t>0:\limsup_{n\to\infty}\frac{N_{d_n^T,K}(\epsilon)}{a_n(t)}>0\},
\end{equation}
and $S(\epsilon,K)=\left\{
                     \begin{array}{ll}
                       \sup S_1(\epsilon,K), & \hbox{if $S_1(\epsilon,K)\neq\emptyset$;} \\
                       0, & \hbox{if $S_1(\epsilon,K)=\emptyset$.}
                     \end{array}
                   \right.
$ with
\begin{equation}\label{eq:maximalSeparated}
S_1(\epsilon,K)=\{t>0:\limsup_{n\to\infty}\frac{S_{d_n^T,K}(\epsilon)}{a_n(t)}>0\}.
\end{equation}
By replacing $\limsup$ in \eqref{eq:minimalCovering} and \eqref{eq:maximalSeparated} by $\liminf$, we can define the lower topological slow entropy $\lent_{a_n(t)}^{\top}(T)$ for $T$. If $\uent^{\top}_{a_n(t)}(T)=\lent^{\top}_{a_n(t)}(T)$, we define this value as the topological slow entropy of $T$ with respect to $a_n(t)$ and denote it as $\ent^{\top}_{a_n(t)}(T)$.
\end{definition}

We provide following characterization of the vanishing of topological slow entropy at all scales  in the setting of minimality:
\begin{proposition}[Proposition A.2, \cite{KanigowskiVinhageWei}]
Suppose $T:(X,d)\to(X,d)$ is a minimal homeomorphism and $(X,d)$ is a compact metric space. Then $T$ is topologically conjugate to a translation on a compact abelian group if and only if $\uent^{\top}_{a_n(t)}(T)=0$ for every family of scales $a_n(t)$.
\end{proposition}

\subsubsection{Measure-theoretic slow entropy}
In this section, we assume that $T$ is an invertible measure-preserving transformation on a standard Borel probability space $(X,\mathcal{B},\mu)$ and $\mathcal{P}=\{P_1,\ldots,P_m\}$ is a finite measurable partition of $X$. Denote $\Omega_{m,n}=\{w=(w_k)_{k=0}^{n-1}:w_k\in\{1,\ldots,m\}\}$ with $n\in\mathbb{Z}^+$ and the coding map $\phi_{\mathcal{P},n}(x):X\to\Omega_{m,n}$ of $T$ and $\mathcal{P}$ as
$$
\phi_{\mathcal{P},n}(x)=w(x), \text{ where }T^k(x)\in P_{w_k(x)}.
$$
Moreover, for any $w,w'\in\Omega_{m,n}$, we define Hamming metric between $w$ and $w'$ as:
$$
d_n^H(w,w')=\frac{1}{n}\sum_{i=0}^{n-1}(1-\delta_{w_i,w'_i}),
$$
where $\delta_{p,q}=\left\{
                      \begin{array}{ll}
                        1, & \hbox{if $p=q$;} \\
                        0, & \hbox{if $p\neq q$.}
                      \end{array}
                    \right.
$. With the help of coding map $\phi_{\mathcal{P},x}(x)$ and Hamming metric, we can define $(\epsilon,n)$-Hamming balls for any $x\in X$:
$$B_{\mathcal{P},n}(x,\epsilon)=\{y\in X:d_n^H(w(x),w(y))<\epsilon\}.$$
Then let $\alpha_n(\epsilon,\mathcal{P})$ be a family of $(\epsilon,n)$-Hamming balls with $\mu(\cup\alpha_n(\epsilon,\mathcal{P}))>1-\epsilon$, and call it an $(\epsilon,\mathcal{P},n)$-covering of $X$. In this setting, we denote the minimal cardinality of an $(\epsilon,\mathcal{P},n)$-covering by
\begin{equation}
S_{\mathcal{P}}^H(T,n,\epsilon)=\min\{\operatorname{Card}(\alpha_n(\epsilon,\mathcal{P}))\}.
\end{equation}


\begin{definition}[Upper measure-theoretic slow entropy]
Let $a_n(t)$ be a family of positive sequences increasing to infinity and monotonically increasing in $t$, we define the upper measure-theoretic slow entropy of $T$ with respect to a finite measurable partition $\mathcal{P}$ by
$$\uent_a^{\mu}(T,\mathcal{P})=\lim_{\epsilon\to0}A(T,\epsilon,\mathcal{P}),$$
where $A(T,\epsilon,\mathcal{P})=\left\{
                                 \begin{array}{ll}
                                   \sup B(T, \epsilon,\mathcal{P}), & \hbox{if $B(T,\epsilon,\mathcal{P})\neq\emptyset$;} \\
                                   0, & \hbox{if $B(T,\epsilon,\mathcal{P})=\emptyset$.}
                                 \end{array}
                               \right.
$ with
\begin{equation}\label{eq:upperSlow}
B(T,\epsilon,\mathcal{P})=\{t>0:\limsup_{n\to\infty}\frac{S_{\mathcal{P}}^H(T,n,\epsilon)}{a_n(t)}>0\}.
\end{equation}
The upper measure-theoretic slow entropy of $T$ is defined as $\uent^{\mu}_a(T)$ by taking the supremum over all finite measurable partitions:
$$\uent_a^{\mu}(T)=\sup_{\mathcal{P}}\uent_a^{\mu}(T,\mathcal{P}).$$
By replacing $\limsup$ in \eqref{eq:upperSlow} by $\liminf$, we can define the lower measure-theoretic slow entropy $\lent_a^{\mu}(T)$ for $T$. If $\uent_a^{\mu}(T)=\lent_a^{\mu}(T)$, we define this value as the measure-theoretic slow entropy of $T$ and denote it as $\ent^{\mu}_a(T)$.
\end{definition}

Here we document the generating sequence property for measure-theoretic slow entropy, which is a very important feature of measure-theoretic slow entropy:
\begin{proposition}[Proposition 1 in \cite{KT}]\label{prop:generator}
Let $T$ be a measure-preserving transformation on standard Borel probability space $(X,\mathcal{B},\mu)$ and $\mathcal{P}_1\leq\mathcal{P}_2\leq\ldots$ be a family of increasing finite measurable partitions of $X$ with $\vee_{n=1}^{+\infty}\mathcal{P}_n$ generates the $\sigma$-algebra $\mathcal{B}$. Then for any scale family $a_n(t)$, we have
$$\uent^{\mu}_a(T)=\lim_{n\to+\infty}\uent_a^{\mu}(T,\mathcal{P}_n),$$
$$\lent^{\mu}_a(T)=\lim_{n\to+\infty}\lent_a^{\mu}(T,\mathcal{P}_n).$$
\end{proposition}

In fact, with the help of measure-theoretic slow entropy, we can detect all complexity among zero entropy systems. On the other hand, when measure-theoretic entropy vanishes at all scales, our system will look like a translation:
\begin{proposition}[Proposition 3, \cite{FerencziComplexity}]
Suppose $T$ is a measure-preserving transformation on standard Borel probability space $(X,\mathcal{B},\mu)$. Then $T$ is measure-theoretic isomorphic to a translation on a compact group if and only if
$$\uent_a^{\mu}(T)=0$$
with respect to every family of scales $a_n(t)$, or if and only if
$$\lent^{\mu}_a(T)=0$$
with respect to every family of scales $a_n(t)$.
\end{proposition}

Similar to entropy's situation, slow entropy also obeys Goodwyn's inequality. However we need the following definition at first:
\begin{definition}
A metric space $X$ is well-partitionable if it is $\sigma$-compact and for any Borel probability measure $\mu$, compact set $K\subset X$ and $\epsilon,\delta>0$, there exist $\kappa>0$ and a finite measurable partition $\mathcal{P}$ of $K$ with atoms' diameters belonging to $(\frac{\epsilon}{2},\epsilon)$ such that $\mu\left(\cup_{\xi\in\mathcal{P}}\partial_{\kappa}\xi\right)<\delta$, where
$$\partial_{\kappa}\xi=\{y\in X:B(y,\kappa)\cap\xi\neq\emptyset\text{ but }B(y,\kappa)\nsubseteq\xi\}.$$
\end{definition}
It is worth to point out that any smooth manifold is well-partitionable. Now slow Goodwyn's inequality can be formulated as:
\begin{theorem}[Proposition 2, \cite{KT}]
Suppose $X$ is a well-partitionable metric space and $T:X\to X$ is a homeomorphism preserving a non-atomic Borel probability measure $\mu$. Then for any family of scale functions $a_n(t)$, we have
$$\uent^{\mu}_a(T)\leq\uent^{\top}_a(T),$$
$$\lent^{\mu}_a(T)\leq\lent^{\top}_a(T).$$
\end{theorem}

\subsection{Space of smooth diffeomorphisms}

Let $M$ be the disc $\mathbb{D}^2$, or the annulus $\mathbb{A}=\mathbb{S}^1\times [0,1]$ or the torus $\mathbb{T}^2$. For any $k\in\N$, the set of all measure-preserving $C^k$ diffeomorphisms, denoted by $\text{Diff }^k(M,\mu)$, has the structure of a polish group. For $k=\infty$, the coarsest topology refining all the $C^k$-topologies makes $\text{Diff}^\infty(M,\mu)$ into a polish group.

Since we use the `quantitative' version of the AbC method, it is necessary to have an explicit formulation of the topology.  We borrow the description from Section 2.3 of \cite{FS} and skip details to give a terse presentation of the definitions and results relevant to our paper. We discuss topologies on the space of smooth diffeomorphisms on $\mathbb{A} = \mathbb{S}^1 \times \left[0,1\right]$. It is straightforward to adapt these definitions to the other manifolds.

For a diffeomorphism $f = \left(f_1,f_2\right): \mathbb{S}^1 \times \left[0,1\right] \rightarrow \mathbb{S}^1 \times \left[0,1\right]$, where $f_1,f_2$ are the coordinate functions, let $\tilde{f}=(\tilde{f}_1,\tilde{f}_2):\mathbb{R}\times \left[0,1\right] \rightarrow \mathbb{R} \times \left[0,1\right]$ be a lift of $f$ to the universal cover. Then for $m\in \mathbb{Z}$, $\tilde{f}_1\left(\theta +m,  r\right) - \tilde{f}_i \left(\theta, r\right) \in \mathbb{Z}$, and $\tilde{f}_2\left(\theta +m, r\right) = \tilde{f}_2 \left(\theta, r\right)$.

To define explicit metrics on $\text{Diff}^k\left(\mathbb{S}^1 \times \left[0,1\right]\right)$ the subsequent notations will be useful:
\begin{definition}
	\begin{enumerate}
		\item For a sufficiently differentiable function $f: \mathbb{R}^2 \rightarrow \mathbb{R}$ and a multi-index $\vec{a} = \left(a_1,a_2\right) \in \mathbb{N}^2_0$
		\begin{equation*}
			D_{\vec{a}}f := \frac{\partial^{\left|\vec{a}\right|}}{\partial x_1^{a_1}\partial x_2^{a_2}} f,
		\end{equation*}
		where $\left|\vec{a}\right| = a_1+a_2$ is the order of $\vec{a}$.
		\item For a continuous function $F: [0,1]^2 \rightarrow \mathbb{R}$
		\begin{equation*}
			\left\|F\right\|_0 := \sup_{z \in [0,1]^2} \left|F\left(z\right)\right|.
		\end{equation*}
	\end{enumerate}
\end{definition}
A diffeomorphism $f\in\text{Diff}^k\left(\mathbb{S}^1 \times \left[0,1\right]\right)$ can be regarded as a map from $\left[0,1\right]^2$ to $\mathbb{R}^2$ by taking a lift of $f$ to the universal cover and then restricting the domain to $\left[0,1\right]^2$. In this way the expressions $\left\|f_i\right\|_0,$ as well as  $\left\|D_{\vec{a}}f_i\right\|_0$ for any multiindex $\vec{a}$ with $\left| \vec{a}\right|\leq k,$ can be understood for $f=\left(f_1,f_2\right) \in\text{Diff}^k\left(\mathbb{S}^1 \times \left[0,1\right]\right)$. (Here $\left\|f_i\right\|_0$ is taken to be the minimum value of $\left\|F_i\right\|_0$, over all choices of lifts of $f$, where $F_i$ is the $i$th coordinate function of the lift.) Thus such a diffeomorphism can be regarded as a continuous map on the compact set $\left[0,1\right]^2$,  and every partial derivative of order at most $k$ can be extended continuously to the boundary.  Therefore the maxima that occur in the definition below are finite.
\begin{definition}
	\begin{enumerate}
		\item For $f,g \in\text{Diff}^k\left(\mathbb{S}^1 \times \left[0,1\right]\right)$ with coordinate functions $f_i$ and $g_i$, respectively, we define
		\begin{equation*}
			\tilde{d}_0\left(f,g\right) = \max_{i=1,2} \left\{ \inf_{p \in \mathbb{Z}} \left\| \left(f - g\right)_i + p\right\|_0\right\}
		\end{equation*}
		as well as
		\begin{equation*}
			\tilde{d}_k\left(f,g\right) = \max \left\{ \tilde{d}_0\left(f,g\right), \left\|D_{\vec{a}}\left(f-g\right)_i\right\|_0 \ : \ i=1,2 \ , \ 1\leq \left|\vec{a}\right| \leq k \right\}.
		\end{equation*}
		\item Using the definitions from part 1. we define for $f,g \in\text{Diff}^k\left(\mathbb{S}^1 \times \left[0,1\right]\right)$:
		\begin{equation*}
			d_k\left(f,g\right) = \max \left\{ \tilde{d}_k\left(f,g\right) \ , \ \tilde{d}_k\left(f^{-1},g^{-1}\right)\right\}.
		\end{equation*}
	\end{enumerate}
\end{definition}

Obviously $d_k$ describes a metric on $\text{Diff}^k\left(\mathbb{S}^1 \times \left[0,1\right]\right)$ measuring the distance between the diffeomorphisms as well as their inverses.

\begin{definition}
	\begin{enumerate}
		\item A sequence of diffeomorphisms in $\text{Diff}^{\infty}\left(\mathbb{S}^1 \times \left[0,1\right]\right)$ is called convergent in $\text{Diff}^{\infty}\left(\mathbb{S}^1 \times \left[0,1\right]\right)$ if  for every $k \in \mathbb{N}$ it converges in $\text{Diff}^k\left(\mathbb{S}^1 \times \left[0,1\right]\right)$.
		\item On $\text{Diff}^{\infty}\left(\mathbb{S}^1 \times \left[0,1\right]\right)$ we declare the following metric
		\begin{equation*}
			d_{\infty}\left(f,g\right) = \sum^{\infty}_{k=1} \frac{d_k\left(f,g\right)}{2^k \cdot \left(1 + d_k\left(f,g\right)\right)}.
		\end{equation*}
	\end{enumerate}
\end{definition}

It is a general fact that $\text{Diff}^{\infty}\left(\mathbb{S}^1 \times \left[0,1\right]\right)$ is a complete metric space with respect to this metric $d_{\infty}$.

Again considering diffeomorphisms on $\mathbb{S}^1 \times \left[0,1\right]$ as maps from $\left[0,1\right]^2$ to $\mathbb{R}^2$ we add the next notation:
\begin{definition}\label{def:norm}
	Let $f \in\text{Diff}^k\left(\mathbb{S}^1 \times \left[0,1\right]\right)$ with coordinate functions $f_i$ be given. Then
	\begin{equation*}
		\left\| Df \right\|_0 := \max_{i,j \in \left\{1,2\right\}} \left\| D_j f_i \right\|_0
	\end{equation*}
	and
	\begin{equation*}
		\vertiii{f}_k := \max \left\{ \left\|D_{\vec{a}} f_i \right\|_0 , \left\|D_{\vec{a}} \left(f^{-1}_{i}\right)\right\|_0 \ : \ i = 1,2, \ \vec{a} \in \mathbb{N}^2_0,\ 0\leq \left| \vec{a}\right| \leq k \right\}.
	\end{equation*}
\end{definition}

The following lemma, which is lemma 5.6 in \cite{FS}, is going to be used in later estimates.

\begin{lemma} \label{lem:proximity}
Let $k\in\N$. For all $h\in\text{Diff}^k(M,\mu)$ and all $\a,\b\in\R$,
$$d_k(h\circ R_\a\circ h^{-1},h\circ R_\b\circ h^{-1})\leq C \cdot \max\{\vertiii{h}^{k+1}_{k+1}, \vertiii{h^{-1}}^{k+1}_{k+1}\} \cdot |\a-\b|,$$
where $C$ depends only on $k$.
\end{lemma}

We also list another result regarding sub-multiplicity of the $\vertiii{\cdot}_k$ defined above, which is a direct consequence of Fa\`{a} di Bruno's formula.
\begin{lemma}\label{lem:submultiplicative}
For any two smooth functions $f:\R^2\to\R^2$ and $g:\R^2\to\R^2$, and any $k>0$, if the composition $f\circ g$ is defined on some open set $U$, then
$$\vertiii{f\circ g}_k\leq C\vertiii{f}_k\cdot \vertiii{g}_k^k.$$
\end{lemma}

\subsection{AbC diffeomorphisms} \label{subsec:Anosov-Katok}

Let $M$ be the disc $\mathbb{D}^2$, or the annulus $\mathbb{A}$ or the torus $\mathbb{T}^2$ equipped with the action of the circle inducing by the flow $$R_t(x,y)=(x+t,y).$$ The coordinates when $M=\T^2$ or $\mathbb{A}$ are regular Cartesian coordinates inherited from $\R^2$, while for the disc we use polar coordinates. Define
\begin{equation}
\begin{aligned}
    & \Delta_n := \{(x,y)\in M:0\leq x<\frac{1}{n}\},\\
    & \Delta_{n}^i := R_\frac{i}{n}(\Delta_n),\\
    & \Delta_{n,m} := \{(x,y)\in\T^2:0\leq x  < \frac{1}{n}, 0\leq y < \frac{1}{m}\},\\
    & \Delta_{n,m}^{i,j} := \{(x,y)\in\T^2:(x,y)=\big(x'+\frac{i}{n},y'+\frac{j}{m}\big)\text{ for some }(x',y')\in \Delta_{n,m}\}.
\end{aligned}
\end{equation}
We collect the above sets to form the following two  partitions
\begin{equation}\label{eq:eta}
\begin{aligned}
    & \eta_{n} =\{\Delta_{n}^i: 0\leq i< n\},\\
    & \mathcal{A}_{n, m} =\{\Delta_{n,m}^{i,j}: 0\leq i< n,\; 0\leq j< m\}.
\end{aligned}
\end{equation}
We now outline the AbC method. One can refer to \cite{Ka03} or \cite{AK} for further details. Our exposition here is general but we will put additional restrictions to suit our needs later. Given any summable sequence of positive real numbers $\{\e_n\}_{n \in \N}$ and a non-decreasing sequence of positive integers $\{m_n\}_{n \in \N}$, the construction proceeds inductively. Assume that we have chosen sequences of integers $\{k_i: i=1,2,3,\ldots, n-1\}$, $\{l_i: i=1,2,3,\ldots, n-1\}$, $\{p_i: i=1,2,3,\ldots, n\}, \{q_i: i=1,2,3,\ldots, n\}$, a sequence of rationals $\{\a_i: i=1,2,3,\ldots, n\}$, sequences of diffeomorphisms $\{h_i: i=1,2,3,\ldots, n-1\}, \{H_i: i=1,2,3,\ldots, n-1\},\{T_i: i=1,2,3,\ldots, n-1\}$ and two sequences of partitions  $\{\eta_{q_{i}}: i=1,2,3,\ldots, n-1\}$(see \eqref{eq:eta}), $\{\xi_{i}: i=1,2,3,\ldots, n-1\}$ such that the following properties are satisfied for any $i<n$:
\begin{equation}\label{eq:abcparameters}
\begin{aligned}
    & \a_i=\frac{p_i}{q_i},\qquad \a_{i+1}=\a_i+\b_i,\qquad \b_i=\frac{1}{k_il_iq_i^2},\\
    & p_{i+1}=k_il_ip_i+1,\qquad q_{i+1}=k_il_iq_i^2,\\
    & H_{i}=h_{1}\circ h_2\circ \ldots \circ h_{i},\\
    & T_{i}=H_i\circ R_{\a_{i+1}}\circ H_i^{-1},\\
    & \xi_{i}=H_{i}\eta_{q_{i}},\\
    & d_{m_i}(T_{i-1},T_i)<\epsilon_i.
\end{aligned}
\end{equation}
After completing $n-1$ steps, at the $n$-th step we choose parameter $k_n$ first and construct the smooth diffeomorphism $h_{n}$ such that $h_n\circ R_{\a_n}=R_{\a_n}\circ h_n$. We choose the parameter $l_n$ last in the $n$-th step to guarantee sufficient proximity of the transformations $T_n$  and $T_{n-1}$ in the $C^{m_n}$ topology.

\subsection{Arbitrarily slow AbC diffeomorphisms} \label{subsec:ArbSlow}
One of the interesting dynamical features of AbC diffeomorphisms is that there is no uniform lower bound for their topological slow entropies. This has already been observed by Kanigowski, Vinhage and Wei (see Theorem 4.10.1 in \cite{KanigowskiKatokWei}). Since the complete proof has never been published, we provide the complete proof of this result here with their kind permission.
\begin{theorem}[Theorem 4.10.1 in \cite{KanigowskiKatokWei}]\label{thm:ABCentropy}
Assuming that $\epsilon_{n_0}> \sum_{n=n_0+1}^\infty\epsilon_n$ for all $n_0$ sufficiently large, then for any scale $a_n(t)$, there exists an AbC diffeomorphism $T$ such that its topological slow entropy at this scale is zero, i.e.
$$\uent_{a_n(t)}^{\top}(T)=0.$$
\end{theorem}
The proof of this theorem relies on the following two lemmas, which establish some estimate for the minimal cardinality of covering balls for $T_n$ and the relation between cardinalities of minimal covering sets for $T$ and $T_n$ when both maps are sufficiently close to each other. Recall that $\vertiii{\cdot}_k$ is the norm defined as Definition \ref{def:norm} on $\mathbb{T}^2$.
\begin{lemma} \label{lem:estimationOfBowenBalls}
Given any $\epsilon>0$ and $n\in\mathbb{Z}^+$, the minimal number of Bowen balls required to cover $\mathbb{T}^2$ has an upper bound:
\begin{equation}
    N_{d_m^{T_n}}(\epsilon)\leq\frac{4C^4\vertiii{H_n}_1^4}{\epsilon^2},\qquad\forall m\in\mathbb{N}.
\end{equation}
\end{lemma}

\begin{proof}
Let $x,y\in\mathbb{T}^2$, then by Lemma \ref{lem:submultiplicative}, we have
\begin{equation}
\begin{aligned}
     d(T_n^mx,T_n^my) & \leq \vertiii{T_n^m}_1d(x,y)= \vertiii{H_n^{-1}\circ R_{\alpha_{n+1}}^m\circ H_n}_1d(x,y)\\
     & \leq C^2\vertiii{H_n^{-1}}_1\vertiii{R_{\alpha_{n+1}}^m}_1\vertiii{H_n}_1d(x,y)\leq C_RC^2\vertiii{H_n}_1^2d(x,y),\\
\end{aligned}
\end{equation}
where $C_R$ is the constant to bound the $\vertiii{\cdot}_1$ norms of any rotations.

So we obtain,
\begin{equation}
     d(x,y)  \leq \frac{\epsilon}{C^2\vertiii{H_n}_1^2}
    \implies  d(T_n^mx,T_n^my) \leq \epsilon.
\end{equation}
As a result, every $(\epsilon, m)$-Bowen ball contains a regular ball of radius $\frac{\epsilon}{C^2\vertiii{H_n}_1^2}$. Since we are dealing with regular balls on $\mathbb{T}^2$, we have the desired results.
\end{proof}

\begin{lemma}\label{lem:BowenTnT}
Let $m\in\mathbb{Z}^+$ and $\epsilon>0$. If $n$ is large enough to satisfy
$$\vertiii{T-T_n}_{0}\leq\frac{\epsilon}{2\sum_{k=1}^m\vertiii{T^{k-1}}_{1}},$$
then we have
\begin{equation} \label{eq:BowenTnT}
\begin{aligned}
N_{d_m^{T_n}}(4\epsilon)\leq N_{d_m^{T}}(2\epsilon)\leq N_{d_m^{T_n}}(\epsilon).
\end{aligned}
\end{equation}
\end{lemma}
\begin{proof}
Recall that $x,y$ belong to the same $(\epsilon, m)$-Bowen ball of $T$ if and only if
$$ d(T^ix,T^iy)\leq \epsilon \text{ for } 0\leq i <m.$$
By using triangle inequality, we obtain that
\begin{equation}\label{eq:triangle1}
d(T^ix,T^iy)\leq d(T^ix,T^i_nx)+d(T^i_nx,T^i_ny)+d(T^i_ny,T^iy).
\end{equation}
The bounds of the first and the third term in \eqref{eq:triangle1} follow from the following inequality:
\begin{equation}\label{eq:telescope1}
\begin{aligned}
    \vertiii{T^i-T_n^i}_0 & \leq \vertiii{T^i-T^{i-1}T_n}_0+\vertiii{T^{i-1}T_n-T^i_n}_0 \\
    &\leq \vertiii{T^{i-1}}_1\vertiii{T-T_n}_0 + \vertiii{T^{i-1}T_n-T^i_n}_0\\
    & \leq \vertiii{T^{i-1}}_1\vertiii{T-T_n}_0 + \vertiii{T^{i-1}-T^{i-1}_n}_0
     \leq \sum_{k=1}^i\vertiii{T^{k-1}}_1\vertiii{T-T_n}_0\\
    &\leq\vertiii{T-T_n}_0\sum_{k=1}^m\vertiii{T^{k-1}}_1\leq\frac{\epsilon}{2}.
\end{aligned}
\end{equation}

Combining \eqref{eq:triangle1} and \eqref{eq:telescope1}, we obtain that
\begin{equation}\label{eq:telescopeControl}
    d(T^ix,T^iy)\leq d(T_n^ix,T_n^iy) + \epsilon.
\end{equation}
Inequality \eqref{eq:telescopeControl} implies that if $x$ and $y$ both belong to some $(\epsilon, m)$-Bowen ball for $T_n$ centered around say $z$, then $x$ and $y$ both belong to the $(2\epsilon, m)$-Bowen ball for $T$ centered around $z$. In other words, the $(\epsilon, m)$-Bowen ball for $T_n$ centered around $z$ is contained inside the $(2\epsilon, m)$-Bowen ball for $T$ centered around $z$. As a result, we have
$$ N_{d_m^{T}}(2\epsilon)\leq N_{d_m^{T_n}}(\epsilon).$$
The other side of the inequality \eqref{eq:BowenTnT} can be obtained in similar way by replacing \eqref{eq:triangle1} by
$$ d(T^ix,T^iy)\geq -d(T^ix,T^i_nx)+d(T^i_nx,T^i_ny)-d(T^i_ny,T^iy),$$
which indeed implies that
\begin{equation}\label{eq:telescopecontrol1}
    d(T^i_nx,T^i_ny)\leq d(T^ix,T^iy) + \epsilon.
\end{equation}
\end{proof}

Now we proceed with the proof of theorem \ref{thm:ABCentropy}.

\begin{proof}[Proof of Theorem \ref{thm:ABCentropy}]
Fix $\epsilon>0$ and $m\in\mathbb{N}$. Pick $n_0$ large enough to satisfy the hypothesis of lemma \ref{lem:estimationOfBowenBalls}. So, we have the estimate
\begin{equation}\label{eq:TelescopeControl}
 \vertiii{T-T_{n_0}}_0\leq\frac{\epsilon}{2\sum_{n=1}^m\vertiii{T^{k-1}}_1}.
\end{equation}

By AbC construction, we have:
\begin{equation}\label{eq:estimateOfTT0}
    \vertiii{T - T_{n_0}}_1 \leq \sum_{n=n_0}^\infty\epsilon_n<2\epsilon_{n_0}.
\end{equation}

Recall that $T_n=H_nR_{\alpha_{n+1}}H_n^{-1}$ and thus
$$\vertiii{H^{-1}_n}_0\vertiii{T_n}_0\geq\vertiii{R_{\alpha_{n+1}}H^{-1}_n}_0=\vertiii{H^{-1}_n}_0,$$
which yields $\vertiii{T_n}_0\geq1$.  Since $\vertiii{T-T_n}_{1}\to0$ as $n\to\infty$, we have $\vertiii{T}_{1}\geq1$.
Combining above estimate with \eqref{eq:TelescopeControl}, we get
$$\vertiii{T-T_{n_0}}_{1}\leq\frac{\epsilon}{2m},$$
which is equivalent to
\begin{equation}\label{eq:mCondition}
m\leq\frac{\epsilon}{2\vertiii{T-T_{n_0}}_{1}}.
\end{equation}

By \eqref{eq:estimateOfTT0}, $m=\frac{\epsilon}{4\epsilon_{n_0}}$ satisfies \eqref{eq:mCondition}. Recall that the number of $(2\epsilon, m)$-Bowen balls is bounded by $\frac{4C^4\vertiii{H_n}_1^4}{\epsilon^2}$ due to Lemma \ref{lem:estimationOfBowenBalls} and Lemma \ref{lem:BowenTnT}. Notice that the choice of $\epsilon_{n_0}$ is independent of $\epsilon$ and thus by choosing $\epsilon_{n_0}$ small enough, $m$ can be large enough to get zero topological slow entropy at any prescribed given scale.
\end{proof}

\subsection{Additional estimates associated with AbC constructions}

In this section we present a few results that comes in handy during the computation of slow entropy and related estimates. In particular, we impose some restrictions on the growth of parameters, that are stronger than necessary for convergence, to ensure that bounds obtained for size of minimal covering and maximal separated  sets can be upgraded from $T_n$ to $T$. The general strategy we adopt to compute the topological slow entropy for some of the well known AbC diffeomorphism is to get upper bounds for the maximal size of separated sets and lower bounds for the minimal number of Bowen balls required to cover $\T^2$ for the periodic diffeomorphism $T_n$. In order to upgrade these results to estimates for the limit diffeomorphism $T$, we will need to introduce several requirements and modification to the AbC method itself and prove a sequence of lemmas to enable estimates for $T$. Without the modifications presented below, the upgrade process will be challenging.

As a recurring theme in all the constructions that appear later in this article, we choose parameter $l_n$ in the following way: First we choose a parameter $l_n'$ to be a positive integer satisfying
\begin{equation}\label{eq:equation ln' convergence}
l_n'\geq \max\{\vertiii{H_{n}}^{m_n+1}_{m_n+1}, \vertiii{H_{n}^{-1}}^{m_n+1}_{m_n+1}\}.
\end{equation}
We also require the sequence $\{l_n'\}_{n \in \N}$ to satisfy
\begin{equation}\label{eq:equation ln' are summable}
\sum_{n=1}^\infty \frac{1}{l_n'}<1
\end{equation}
Then we choose the parameter $l_n$ to be a positive integer such that the following condition is satisfied,
\begin{equation}\label{eq:equation ln condition}
l_n\geq\lceil\vertiii{H_{n}}_1\rceil\cdot l_n'
\end{equation}

\begin{lemma}\label{lem:convergenceOfAbC}
Suppose $T_n$ are $C^{\infty}$ AbC diffeomorphisms with parameters $l_n'$ and $l_n$ satisfying inequalities \eqref{eq:equation ln' convergence} and \eqref{eq:equation ln condition}, respectively. If the sequence $m_n$ increases to $K$, where $K$ is any positive integer or $\infty$, then $T_n$ converges to a $C^K$ diffeomorphism $T$.
\end{lemma}

\begin{proof}
From Lemma \ref{lem:proximity} we get with inequalities \eqref{eq:equation ln' convergence} and \eqref{eq:equation ln condition} that
\begin{equation}
\begin{aligned}
&d_{m_n}(H_{n-1}\circ R_{\a_{n}}\circ H_{n-1}^{-1},H_n\circ R_{\a_{n+1}}\circ H_n^{-1})\\ = &d_{m_n}(H_n\circ R_{\a_{n}}\circ H_n^{-1},H_n\circ R_{\a_{n+1}}\circ H_n^{-1})\\
\leq &C\;\max\{\vertiii{H_{n}}^{m_n+1}_{m_n+1}, \vertiii{H_{n}^{-1}}^{m_n+1}_{m_n+1}\}\;\frac{1}{k_nl_nq_n^2}\\
\leq &C\; \frac{1}{k_nq_n^2}
\end{aligned}
\end{equation}
In conclusion, $T_n$ is Cauchy and hence converges to a diffeomorphism $T$ in $C^{m_n}$. Since $m_n\to K$ as $n\to \infty$, we complete the proof.
\end{proof}

We remark at this point that for the purpose of convergence as described in the lemma above, it is sufficient to assume $l_n=l_n'$. However for the purpose of making the computation of slow entropy easier, we instead chose $l_n$ as in \eqref{eq:equation ln condition}. It is also worth to point out that further restrictions, in particular to the decay rate of ${\e_n}$ may be imposed for individual cases of construction. But for now, we proceed without any further restrictions.

\begin{lemma}\label{lem:convergenceCriterion2}
The parameters in the AbC construction can be chosen such that for $n\in\N$ sufficiently large we have for $0\leq i\leq l_n'q_n$ and $j\geq 0$
$$d_0(T_{n-1}^i, \; T_{n+j}^i)\leq \frac{2}{k_nq_n}, \ \ \forall x\in\T^2.$$
\end{lemma}
\begin{proof}
The desired estimate follows from mean value inequality and inequality \eqref{eq:equation ln condition}:
\begin{equation}
\begin{aligned}
d_0(T_{n-1}^i, T_{n+j}^i) \leq &\sum_{m=0}^{j}d_0(T_{n-1+m}^i, T_{n+m}^i)\leq\sum_{m=0}^{j}\vertiii{H_{n+m}}_1\frac{i}{k_{n+m}l_{n+m}q_{n+m}^2}\\
\leq &\frac{1}{k_{n}q_{n}} +\sum_{m=1}^j\frac{1}{k_{n+m}l_{n+m}'q_{n+m}}< \frac{1}{k_nq_n} +\frac{1}{k_nq_n} \sum_{m=n+1}^\infty\frac{1}{l_m'}\\
\overset{\eqref{eq:equation ln' are summable}}{<} &\frac{2}{k_nq_n}.
\end{aligned}
\end{equation}
\end{proof}

\begin{lemma}\label{lem:convergenceCriterion3}
The parameters in the AbC construction can be chosen such that for any $n\in\N$ and for $0\leq i\leq l_n'q_n$, we have:
$$d_0(T_{n-1}^i,\; T^i)\leq \frac{3}{k_nq_n}\qquad\forall x\in\T^2.$$
\end{lemma}
\begin{proof}
This lemma follows from
\begin{equation}
\begin{aligned}
d_0(T^i_{n-1},\; T^i) & \leq d_0(T^i_{n-1},\; T^i_{n+j}) + d_0(T^i_{n+j},\; T^i) \leq d_0(T^i_{n-1},\; T^i_{n+j}) + \frac{1}{k_nq_n} \\
& \leq \frac{3}{k_nq_n},
\end{aligned}
\end{equation}
where the second inequality follows by choosing $j$ large enough and the last inequality follows from Lemma \ref{lem:convergenceCriterion2}.
\end{proof}

\begin{proposition}\label{prop:upgradeTntoTForSeparatedSets}
Suppose $T$ is a diffeomorphism obtained as the limit of  AbC diffeomorphisms $T_n$ satisfying all the requirements described above. For any fixed $0\leq m \leq l_{n+1}'q_{n+1}$ and $\e>0$, we have
$$S_{d^T_m}(\e)\geq S_{d^{T_n}_m}(2\e).$$
\end{proposition}
\begin{proof}
Suppose $x,y$ are two points such that there exists $0\leq i\leq m\leq l_{n+1}'q_{n+1}$ with $d(T_n^ix,T_n^iy) \geq 2\e$. Combining this with triangle inequality and Lemma \ref{lem:convergenceCriterion3}, we have the following estimates for $n$ sufficiently large:
\begin{equation}
\begin{aligned}
d(T^ix,T^iy) & \geq -d(T^ix,T_n^ix)+d(T_n^ix,T_n^iy)-d(T_n^iy,T^iy)\\
& \geq -\frac{3}{k_{n+1}q_{n+1}} + 2\e - \frac{3}{k_{n+1}q_{n+1}}\\
& \geq \e,
\end{aligned}
\end{equation}
which finishes the proof by the definition of maximal separated sets.
\end{proof}

\begin{proposition}\label{prop:upgradeTntoTForSpanningSets}
Suppose $T$ is a diffeomorphism obtained as the limit of  AbC diffeomorphisms $T_n$ satisfying all the requirements described above. For any fixed $0\leq m \leq l_{n+1}'q_{n+1}$ and $\e>0$, we have
$$N_{d^T_m}(\e)\leq N_{d^{T_n}_m}(\frac{\e}{2}).$$
\end{proposition}
\begin{proof}
Suppose $x,y$ are two points such that there exists $0\leq i\leq m\leq l_{n+1}'q_{n+1}$ satisfied  $d(T_n^ix,T_n^iy) \leq \frac{\e}{2}$. Then by triangle inequality and Lemma \ref{lem:convergenceCriterion3}, we have the following estimates if $n$ is large enough:
\begin{equation}
\begin{aligned}
d(T^ix,T^iy) & \leq d(T^ix,T_n^ix)+d(T_n^ix,T_n^iy)+d(T_n^iy,T^iy)\\
& \leq \frac{3}{k_{n+1}q_{n+1}} + \frac{\e}{2} + \frac{3}{k_{n+1}q_{n+1}}\\
& \leq \e,
\end{aligned}
\end{equation}
which finishes the proof by the definition of minimal covering sets.
\end{proof}

\subsubsection{Quasi-rotations}

We end this section with a discussion of quasi-rotations which forms the foundation of the measure-preserving diffeomorphisms we construct later.
\begin{lemma}[From proof of lemma 2, \cite{FSW}]\label{lem:quasiRotation}
Given any $\e>0$, there exists an area-preserving $C^\infty$ diffeomorphism
$$\varphi_\e:[0,1]^2\to[0,1]^2$$
satisfying the following properties:
\begin{enumerate}
    \item When restricted to $[2\e,1-2\e]^2$, $\varphi_\e$ acts as a pure rotation by $\frac{\pi}{2}$.
    \item $\varphi_\e={id}$ when restricted to $[0,1]^2\setminus [\e,1-\e]^2$.
\end{enumerate}
\end{lemma}

From now on, the coordinate functions corresponding to the quasi-rotation $\varphi_\e$ is denoted as $([\varphi_\e]_1, [\varphi_\e]_2)$. Moreover, for any $q\in\mathbb{Z}^+$ we define $\phi_{q,\e}: [0,\frac{1}{q}]\times [0,1]\to [0,\frac{1}{q}]\times [0,1]$ as
\begin{equation}
\phi_{q,\e}(x,y)=\left(\frac{1}{q}[\varphi_{\e}]_1(qx, y), [\varphi_{\e}]_2(qx,y)\right).
\end{equation}
We end this section by recalling a lemma about norm estimates for quasi-rotations.
\begin{lemma}[Lemma 3, \cite{FSW}]\label{lem:derivativeNormEstimateForQuasiRotations}
The quasi rotations defined above satisfy the following norm estimates:
$$\max\{\vertiii{\phi_{q,\e}}_k, \vertiii{\phi_{q,\e}^{-1}}_k\}\leq C q^{k},$$
where the constant $C$ depends on $\e$ and $k$ but not on $q$.
\end{lemma}

\noindent\textbf{Additional requirements:} In all the constructions for the rest of the paper we will require that $\{\e_n\}_{n\in \N}$ is a summable sequence of positive numbers  and in addition
\begin{itemize}
    \item $\{\e_n\}_{n\in \N}$ is monotonically decreasing.
    \item For every $n \in \N$ we have
    \begin{equation}\label{eq:en4}
        \e_n\leq \frac{1}{n^4}.
    \end{equation}
    \item At stage $n$ of the construction $q_n$ is large enough to guarantee
    \begin{equation}\label{eq:qEpsilon}
        q_n>\frac{1}{\e_n}.
    \end{equation}
\end{itemize}

\subsection{Scaling functions} \label{subsec:scaling}

We primarily use three types of scaling functions. Two of them are the polynomial scaling functions $a_m^{\pol}(t)=m^t$ and logarithmic scaling functions $a_m^{\log}(t)=(\ln m)^t$, respectively. The third one is a scaling function with speed of growth faster than any log type scaling function, but slower than any polynomial type scaling function. We will use a ``gamma-like'' function to explicitly construct this \emph{intermediate} speed scaling function.

Let $\Gamma$ denote the usual Gamma function,
$$\Gamma(x)=\int_0^\infty t^{x-1}e^{-t}dt.$$
Recall that the gamma function gives the usual factorial at positive integer values, $\Gamma(n+1)=n!$. For any integer $r\geq 4$, we define a function $\Gamma_r:\R^+\to\R^+$ using the gamma function as
$$\Gamma_r(x)=\big[\Gamma(x^\frac{1}{r}+1)\big]^r.$$
In particular, notice that $\Gamma_r$ is a well defined increasing function, which takes the following values on any integer of the form $n^r$:
$$\Gamma_r(n^r)=[\Gamma(n+1)]^r=\prod_{m=1}^n(m^r).$$
Since $\Gamma_r$ is an increasing function, thus it is invertible, which implies that we can define two families of functions as follows:
\begin{equation}\label{equation scaling function}
\begin{aligned}
    & a_m^{\operatorname{int1,r}}(t)=m^\frac{t}{\Gamma_r^{-1}({\frac{\ln m}{\ln q_1})}},\\
    & a_m^{\operatorname{int2,r}}(t)=m^\frac{t}{\big[{\Gamma_r^{-1}({\frac{\ln m}{\ln q_1})}}\big]^\frac{r-2}{r}}.
\end{aligned}
\end{equation}
It is clear that these functions are increasing with respect to $t$ and $m$ and hence qualify as scaling functions. The first family of functions gives interesting results in Section \ref{sec:untwistedAbC} and \ref{sec:uniquelyErgodicAbC} for upper topological slow entropy and also in Section \ref{sec:weakMixingAbC} for upper measure-theoretic entropy, while the second is an appropriate choice in Section \ref{sec:weakMixingAbC} for upper topological slow entropy.

In fact, some estimates for the values of these scaling functions at specific points are very helpful for our AbC estimates. It will turn out that the parameter sequence $\{q_n\}_{n\in \N}$ defined in any of the \textit{smooth} AbC methods described later will satisfy a sequence of inequalities for any given $n$:
\begin{align}\label{equation order inequality}
q_n<q_n^{n^2}<l_n'q_n<q_n^{n^r}= q_{n+1},
\end{align}
where the last equality inductively implies
\begin{align}
q_{n+1}=q_1^{\prod_{m=1}^n(m^r)}=q_1^{\Gamma_r(n^r)}.
\end{align}
Then with the help of above equation, we can compute the exact value of our intermediate scaling functions at $q_{n+1}$:
\begin{equation}\label{eq:intermediateScaleAtQ(n+1)}
\begin{aligned}
    a^{\operatorname{int1,r}}_{q_{n+1}}(t)   &= q_{n+1}^\frac{t}{\Gamma_r^{-1}({\frac{\ln q_{n+1}}{\ln q_1})}} = \big[q_n^{n^r}\big]^{\frac{t}{\Gamma_r^{-1}({\frac{\ln  q_1^{\prod_{m=1}^n(m^r)}}{\ln q_1})}}}  = \big[q_n^{n^r}\big]^{\frac{t}{\Gamma_r^{-1}( {\prod_{m=1}^n(m^r))}}} = \big[q_n^{n^r}\big]^{\frac{t}{n^r}} = q_n^t,\\
    a^{\operatorname{int2,r}}_{q_{n+1}}(t)  & = q_{n+1}^\frac{t}{\big[{\Gamma_r^{-1}({\frac{\ln q_{n+1}}{\ln q_1})}}\big]^\frac{r-2}{r}}   = \big[q_n^{n^r}\big]^{\frac{t}{\big[{\Gamma_r^{-1}( {\prod_{m=1}^n(m^r))}}\big]^\frac{r-2}{r}}} = \big[q_n^{n^r}\big]^{\frac{t}{n^{r-2}}} = q_n^{n^2t}.
\end{aligned}
\end{equation}

As we have already indicated, the speed of intermediate scaling function is between logarithmic scaling function and polynomial scaling function. The following proposition provides a rigorous proof of this phenomenon:
\begin{proposition}\label{prop:differentScales}
For any two positive real number $s$ and $t$ and any integer $r\geq 4$, we have
\begin{equation}\label{eq:pol > int2}
\lim_{m\to \infty}\frac{a^{\operatorname{int2,r}}_m(t)}{a^{\pol}_m(s)}=0,
\end{equation}
\begin{equation}\label{eq:int2 > int1}
\lim_{m\to \infty}\frac{a^{\operatorname{int1,r}}_m(t)}{a^{\operatorname{int2,r}}_m(s)}=0,
\end{equation}
\begin{equation}\label{eq:int1 > log}
\lim_{m\to \infty}\frac{a^{\log}_m(t)}{a^{\operatorname{int1,r}}_m(s)}=0.
\end{equation}
In the `little-o' notation, the above limits translate to $a^{\operatorname{int2,r}}_m(t)=o(a^{\pol}_m(s))$, $a^{\operatorname{int1,r}}_m(t)=o(a^{\operatorname{int2,r}}_m(s))$ and $a^{\log}_m(t)=o(a^{\operatorname{int1,r}}_m(s))$.
\end{proposition}

\begin{proof}
First we compare the speed of $a^{\pol}$ and $a^{\operatorname{int2,r}}$. For any two positive real numbers $s$ and $t$, we note that for large enough values of $m$ we have
\begin{equation}
\begin{aligned}
\frac{a_m^{\operatorname{int2,r}}(t)}{a_m^{\pol}(s)} &=  \Big[m\Big]^{\Big[\frac{t}{\big[{\Gamma_r^{-1}({\frac{\ln m}{\ln q_1})}}\big]^\frac{r-2}{r}} - s\Big]} = \Big[m^{\frac{t}{\big[{\Gamma_r^{-1}({\frac{\ln m}{\ln q_1})}}\big]^\frac{r-2}{r}}}\Big]^{\Big[1 - \frac{s}{t}\big[{\Gamma_r^{-1}({\frac{\ln m}{\ln q_1})}}\big]^\frac{r-2}{r}\Big]}\\&\leq \Big[m^{\frac{t}{\big[{\Gamma_r^{-1}({\frac{\ln m}{\ln q_1})}}\big]^\frac{r-2}{r}}}\Big]^{[-1]}\searrow 0,
\end{aligned}
\end{equation}
which gives us \eqref{eq:pol > int2}.

Next to compare $a^{\operatorname{int1,r}}$ with $a^{\operatorname{int2,r}}$, we proceed in a similar fashion,
\begin{equation}
\begin{aligned}
\frac{a_m^{\operatorname{int1,r}}(t)}{a_m^{\operatorname{int2,r}}(s)}& =  \Big[m\Big]^{\frac{t}{{\Gamma_r^{-1}({\frac{\ln m}{\ln q_1})}}} - \frac{s}{\big[{\Gamma_r^{-1}({\frac{\ln m}{\ln q_1})}}\big]^\frac{r-2}{r}}} = \Big[m^{\frac{t}{{\Gamma_r^{-1}({\frac{\ln m}{\ln q_1})}}}}\Big]^{\Big[1 -  \frac{s}{t}\big[{\Gamma_r^{-1}({\frac{\ln m}{\ln q_1})}}\big]^\frac{2}{r}\Big]}\\&\leq \Big[m^{\frac{t}{{\Gamma_r^{-1}({\frac{\ln m}{\ln q_1})}}}}\Big]^{[-1]}\searrow 0.
\end{aligned}
\end{equation}
So we have \eqref{eq:int2 > int1}.

Finally, to compare $a^{\operatorname{int1,r}}$ with $a^{\operatorname{log}}$, we use the value of the respective functions at the intermediate point $q_n^{n^2}$. Observe that
$$ n^2\prod_{m=1}^{n-1}m^r \leq \prod_{m=1}^{n}m^r \Rightarrow  \Gamma_r^{-1}(\frac{\ln q_n^{n^2}}{\ln q_1})=\Gamma_r^{-1}(n^2\prod_{m=1}^{n-1}m^r)\leq n^r\Rightarrow a^{\text{int1,r}}_{q_n^{n^2}}(t)\geq \big[q_n\big]^{\frac{t}{n^{r-2}}},$$
which in turn produces the following bounds for the ratio of our scaling functions:
\begin{equation}
\begin{aligned}
\frac{a^{\log}_m(s)}{a^{\operatorname{int1,r}}_m(t)}\leq \begin{cases} \frac{(n^{2}(n-1)^r\ln(q_{n-1}))^s}{q_{n-1}^t}&\text{if } q_n<m\leq q_n^{n^2}\\ & \\ \frac{(n^r\ln(q_{n}))^s}{\big[q_n\big]^{\frac{t}{n^{r-2}}}}=\frac{(n!^r\cdot\ln(q_{1}))^s}{\big[q_1\big]^{\frac{t\cdot (n-1)!^r}{n^{r-2}}}} &\text{if } q_n^{n^2}<m\leq q_{n+1}\end{cases}
\end{aligned}
\end{equation}
This gives us \eqref{eq:int1 > log}.
\end{proof}

The last estimates in this section are for the value of our intermediate scaling functions at $l_n'q_n$, which is important for the computation of slow entropy. Recall that so far we have put some restrictions on $l_n'$ but have not pinned down a value. To proceed we choose:
\begin{equation}\label{eq:ln'}
    l_n'=q_n^{(n-1)^r-1}
\end{equation}
which makes $l_n'q_n=q_n^{(n-1)^r}$ and the conditions required by \eqref{eq:equation ln' are summable} and \eqref{eq:equation ln condition} are also satisfied. We will verify for each individual construction that \eqref{eq:equation ln' convergence} will also be satisfied with this choice. Recall the estimates for $\Gamma_r$,
$$  \Gamma_r\big((n-1)^r\big)\leq (1-\frac{1}{n})^r\Gamma_r(n^r)<\Gamma_r(n^r),$$
which results in estimates for our scaling function as follows
\begin{equation}\label{eq:intermediateScaleAtQ(n)}
\begin{aligned}
    \big[q_n\big]^{\frac{(n-1)^rt}{n^r}}&\leq  a^{\text{int1,r}}_{q_{n}^{(n-1)^r}}(t) =  \Big[q_{n}^{(n-1)^r}\Big]^\frac{t}{\Gamma_r^{-1}({\frac{\ln q_n^{(n-1)^r}}{\ln q_1})}} = \Big[q_{n}^{(n-1)^r}\Big]^\frac{t}{\Gamma_r^{-1}({(1-\frac{1}{n})^r\Gamma_r(n^r)})} \\&\leq \big[q_n\big]^{\frac{(n-1)^rt}{(n-1)^r}}=q_n^t,\\
    \big[q_n\big]^{\frac{(n-1)^rt}{n^{r-2}}}&\leq  a^{\text{int2,r}}_{q_{n}^{(n-1)^r}}(t) =  \Big[q_{n}^{(n-1)^r}\Big]^\frac{t}{\big[{\Gamma_r^{-1}({\frac{\ln q_n^{(n-1)^r}}{\ln q_1})}}\big]^\frac{r-2}{r}} = \Big[q_{n}^{(n-1)^r}\Big]^\frac{t}{\big[{\Gamma_r^{-1}({(1-\frac{1}{n})^r\Gamma_r(n^r)})}\big]^\frac{r-2}{r}}\\& \leq \big[q_n\big]^{\frac{(n-1)^rt}{(n-1)^{r-2}}}=q_n^{(n-1)^2t}.
\end{aligned}
\end{equation}

\section{Topological slow entropy for an untwisted AbC diffeomorphism} \label{sec:untwistedAbC}
This section is based on the quantitative AbC construction in \cite{FSW} producing a smooth diffeomorphism measure-theoretically conjugated to a rotation of the circle by a prescribed Liouville number. However, we do not compute the topological slow entropy for the ergodic version of the construction as presented in the original paper but have to make certain modifications to their construction. These allow us to obtain an untwisted example of an ergodic diffeomorphism on the torus isomorphic to a rotation and to compute its topological upper slow entropy. First we present our construction, then we obtain upper bounds for the cardinality of minimal covering sets, lower bounds for the cardinality of maximal separated sets and proceed with the calculation of the upper topological slow entropy.



\subsection{The AbC construction}

At the $n$-th stage of the construction define the conjugation map $h_n$ at $n$-th stage as follows:
\begin{equation}\label{eq:hnUntwisted}
h_{n}(x,y)|_{\Delta_{q_n}^0}=\begin{cases}\big(\frac{1}{q_n}[\varphi_{\e_n}]_1(q_nx, y), [\varphi_{\e_n}]_2(q_nx,y)\big) & x\in [0,\frac{1}{q_n}-\frac{1}{q_n^2}],\\
\big(\tau_{n,1}(x), [\varphi_{\e_n}]_2(q_n^2(x-\frac{1}{q_n}+\frac{1}{q_n^2}),y)\big) & x\in [\frac{1}{q_n}-\frac{1}{q_n^2},\frac{1}{q_n}],\end{cases}
\end{equation}
where $\tau_{n,1}(x)=\frac{1}{q_n^2}[\varphi_{\e_n}]_1(q_n^2(x-\frac{1}{q_n}+\frac{1}{q_n^2}), y) + \frac{1}{q_n}-\frac{1}{q_n^2}$ and $\varphi_{\e_n}=([\varphi_{\e_n}]_1,[\varphi_{\e_n}]_2)$ is the quasi-rotation as described in Lemma \ref{lem:quasiRotation}. Recall that \eqref{eq:qEpsilon} guarantees that $q_n$ grow faster than $\frac{1}{\e_n}$ and thus the above construction defines a smooth diffeomorphism of $\Delta^0_{q_n}$. Finally notice that these diffeomorphisms can be extended to the whole torus in an equivariant way.

Combinatorially, when restricted to $\Delta^0_{q_n}$, $h_n$ acts as two consecutive rotations. The first component is a bigger rotation that rotates the bulk of the measure of the interior of the rectangle $[0,\frac{1}{q_n}-\frac{1}{q_n^2}]\times [0,1]$ by $90$ degrees and the second rotation also rotates the bulk of the measure of the narrower rectangle $[\frac{1}{q_n}-\frac{1}{q_n^2},\frac{1}{q_n}]\times [0,1]$ by $90$ degrees.
With the above conjugating diffeomorphisms, we can define the Anosov-Katok conjugacies
\begin{equation}\label{eq:anosov-katokUntwisted}
    T_n=H_n\circ R_{\a_{n+1}}\circ H_n^{-1},\text{ where } H_n=h_1\circ\ldots\circ h_{n}.
\end{equation}
Similar to \cite{FSW} we obtain the following theorem using Lemma \ref{lem:convergenceOfAbC}.
\begin{theorem}\label{thm:UntwistedRegularity}
The sequence of diffeomorphisms $\{T_n\}_{n\in \N}$ described in \eqref{eq:anosov-katokUntwisted}, with parameter $m_n$ increasing to $k\leq \infty$, and with other parameters chosen according to the specification provided in \eqref{eq:abcparameters} and \eqref{eq:equation ln condition}, converges to an ergodic $C^k$ diffeomorphisms $T$ of the torus that is measure-theoretically isomorphic to an irrational rotation of the circle.
\end{theorem}

Before proceeding further we introduce the notion of \emph{the central index}: the index $i_c\in[0,q_{n}]$ is the integer such that the rectangle $\Delta^{i_c}_{q_n}$ is closest to the center of $\Delta^0_{q_m}$ for any $m<n$, i.e.
\begin{equation}\label{eq:centralIndex}
i_c=\argmin_{0\leq i<q_n}\Big|\frac{i}{q_n}-\sum_{m=1}^{n-1}\frac{1}{2q_m}\Big|.
\end{equation}

\subsubsection{Norm estimates and parameter growth}
In this section, we obtain norm estimates for the conjugating diffeomorphisms, which allow us to control the parameters growth rates in the AbC constructions.
\begin{lemma}\label{lem:upperBoundOnDerivative}
The conjugating diffeomorphisms $h_n$ satisfy the following norm estimates:
\begin{align*}
\max\{\vertiii{h_n}_k, \vertiii{h_n^{-1}}_k\}\leq C q_n^{2k},
\end{align*}
where the constant $C$ is dependent on $k$ and $\e_n$ but not on $q_n$.
\end{lemma}

\begin{proof}
Follows from equation \eqref{eq:hnUntwisted} and Lemma \ref{lem:derivativeNormEstimateForQuasiRotations}.
\end{proof}

So far we have not specified any specific value for the parameter $l_n$. However we will provide a series of estimates for the derivatives of the conjugating diffeomorphisms that will enable us to simultaneously come up with exact values for both $q_{n+1}$ and $l_n$.
\begin{lemma}\label{lem:estimateQ(n+1)vsQn}
The expression $\lceil\vertiii{H_{n}}_1\rceil\max\{\vertiii{H_{n}}^{m_n+1}_{m_n+1}, \vertiii{H_{n}^{-1}}^{m_n+1}_{m_n+1}\}$ is bounded above by $q_n^{(m_n+1)^4-2}$ and hence any choice of $l_n\geq q_{n}^{(m_n+1)^4-2}$ satisfies the requirement imposed by \eqref{eq:equation ln condition}. Additionally if $k_n=1$, we get $$q_{n+1}\geq q_n^{(m_n+1)^4}.$$ In particular we can choose
\begin{itemize}
    \item $m_n=n-1$ and $l_n=q_n^{n^4-2}$ yielding $q_{n+1}= q_n^{n^4}=q_1^{\prod_{m=1}^n(m^4)},$ or,
    \item $m_n=n-1$ and $l_n=q_n^{q_n-2}$ yielding $q_{n+1}= q_n^{q_n},$ or,
    \item $m_n=K-1$ for some integer $K>1$ and $l_n=q_n^{K^4-2}$ yielding $q_{n+1}= q_n^{K^4}=q_1^{K^{4n}}.$
\end{itemize}
Moreover, if $m_n\geq 1$, then we have the freedom for the choice of a sequence $l_n'$ to satisfy \eqref{eq:equation ln' convergence} and \eqref{eq:equation ln' are summable}.
\end{lemma}
\begin{proof}
Denote $\gamma_n=m_n+1$, then the proof follows from the following estimates by Lemma \ref{lem:submultiplicative} and Lemma \ref{lem:upperBoundOnDerivative}:
\begin{equation}
\begin{aligned}
& \vertiii{H_{n}}_1 \max\{\vertiii{H_{n}}^{\gamma_n}_{\gamma_n}, \vertiii{H_{n}^{-1}}^{\gamma_n}_{\gamma_n}\} \\
\leq &\Big[\vertiii{H_{n-1}^{-1}}_1 \max\{\vertiii{H_{n-1}}^{\gamma_n}_{\gamma_n}, \vertiii{H_{n-1}^{-1}}^{\gamma_n^2}_{\gamma_n}\}\Big] \cdot \vertiii{h_{n}}_1\cdot \max\{\vertiii{h_{n}}^{\gamma_n^2}_{\gamma_n}, \vertiii{h_{n}^{-1}}^{\gamma_n}_{\gamma_n}\} \\
\leq &\Big[\vertiii{H_{n-1}^{-1}}_1 \max\{\vertiii{H_{n-1}}^{\gamma_n}_{\gamma_n}, \vertiii{H_{n-1}^{-1}}^{\gamma_n^2}_{\gamma_n}\}\Big]\cdot  q_n^{2\gamma_n^3+2}\\
< &q_n^{\gamma_n^4-2},
\end{aligned}
\end{equation}
where the last inequality follows from the fact that $q_n>>q_{n-1}$. As a result of this estimate, a choice of $l_n\geq q_n^{\gamma_n^4-2}$ will satisfy \eqref{eq:equation ln condition}. This completes the lemma.
\end{proof}

\subsection{Lower bounds for cardinality of maximal separated sets}

\begin{figure}[H]
\begin{center}
\includegraphics[trim=0cm 0cm 0cm 0cm, clip, width=4.4cm,height=7.5cm]{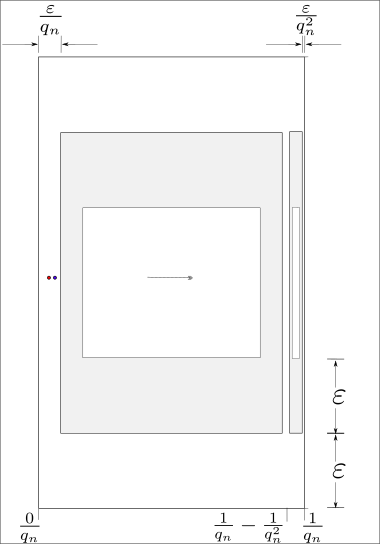}\hspace{.2cm}
\includegraphics[trim=0cm 0cm 0cm 0cm, clip, width=4.4cm,height=7.5cm]{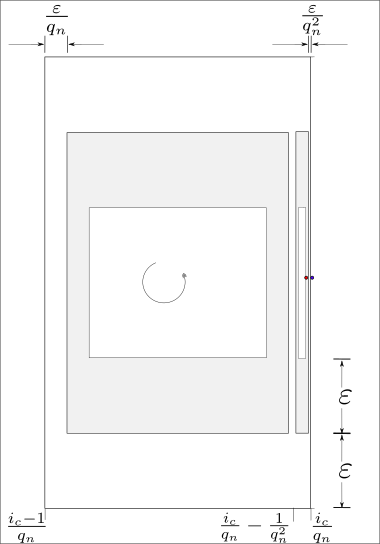}\hspace{.2cm}
\includegraphics[trim=0cm 0cm 0cm 0cm, clip, width=4.4cm,height=7.5cm]{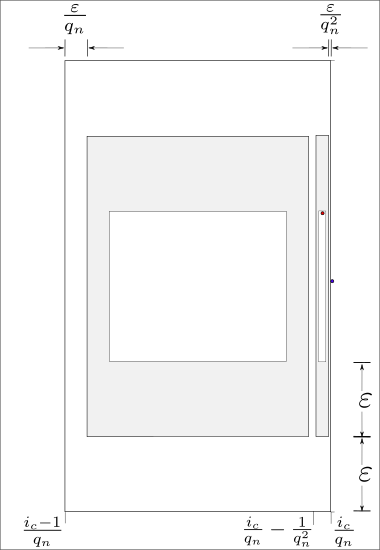}
\caption{Demonstration of the separation mechanism: Under the action of $R^t_{\alpha_{n+1}}$ the red and the blue point move right to the `separation region', then the subsequent action of the conjugation map $h_{n}$ rotates one point to the top (since that point lies in the rotation kernel), while the other remains at place (since that point lies in the identity region).}
\label{figure separation untwisted}
\end{center}
\end{figure}

\begin{lemma} \label{lem:FSWSeperationAtN-thStage}
For any given $\e>0$ and any $n\in\N$, we have
$$S_{d^{T_n}_{q_{n+1}}}(\e)\geq C\frac{q_n}{\e},$$
where $C>0$ is a constant that is independent of $n$ and $\e$.
\end{lemma}
\begin{proof}
Let $0\leq i_0<q_n$ be fixed. Define
\begin{equation}\label{eq:xi}
 \Xi_n(y)=\Big\{\big(\frac{i_0}{q_n}+\frac{2j\e_n}{q_n^2},y\big) : j= 0,1,2, \ldots, \frac{q_n}{2}-1\}\text{ and } \Xi_n= \bigcup_{k=0}^{\lfloor \frac{1}{4\e}\rfloor}\Xi_n\big(\frac{3}{8}+k\e\big).
\end{equation}
Notice that
$$\#(\Xi_n)=C\cdot\frac{1}{\e}\cdot q_n$$
for some positive constant $C$ independent of $n$ and $\e$.

We will give two different strategies to exhibit how points on the same horizontal level in $\Xi$ separate.

\smallskip

\noindent\textit{Strategy I:} Let $P_1$ and $P_2$ be two points in $\Xi_n(y)\subset \Xi_n$. Then assuming without loss of generality $\pi_1(P_1)<\pi_1(P_2)$, where $\pi_1:\mathbb{T}^2\to\mathbb{T}$ is defined as $\pi_1(x,y)=x$. Notice that there exists $0\leq t< q_{n+1}$ such that
$$0<\frac{i_c}{q_n}- \frac{2\e_n}{q_n^2}-\frac{\e_{n-1}}{q_n^2} < \pi_1\left(R^t_{\a_{n+1}}(P_1)\right)\leq \frac{i_c}{q_n}- \frac{2\e_n}{q_n^2},$$
and
$$0<\frac{i_c}{q_n}- \frac{\e_n}{q_n^2} < \pi_1\left(R^t_{\a_{n+1}}(P_2)\right)< \frac{i_c}{q_n}+ \frac{\e_n}{q_n}.$$

Note that the horizontal separation for two points belonging to $\Xi(y)$ is bounded below by $\frac{2\e_n}{q_n^2}$ and above by $\frac{\e_n}{q_n}$, and hence the above set of conditions can always be guaranteed.   So $R^t_{\a_{n+1}}(P_2)$ belongs to a zone where $h_n$ acts as the identity transformation. Hence $h_n(R^t_{\a_{n+1}}(P_2))=R^t_{\a_{n+1}}(P_2)$, while $h_n(R^t_{\a_{n+1}}(P_1))$ rotates by $90$ degrees to the top and into the identity zone of $h_m$ for any $m<n$ because of the monotonicity of the sequence $\{\e_m\}_{m\in \N}$ (see figure \ref{figure separation untwisted}). Hence $H_n(R^t_{\a_{n+1}}(P_1))=h_n(R^t_{\a_{n+1}}(P_1))$, and as a consequence $|H_n(R^t_{\a_{n+1}}(P_1))-h_n(R^t_{\a_{n+1}}(P_2))|\geq \frac{1}{4}$.

On the other hand, since $R^t_{\a_{n+1}}(P_2)$ is at most an horizontal distance $\frac{1}{q_n}$ away from the center $c$ of $\Delta^{j_m}_{q_m}$ for some $j_m$ for any $m<n$, we note that $P_2$ remains in $K_m$ for any $m<n$, where $K_m$ is the region that $h_m$ acts as rotations on. Further we show that since $R^t_{\a_{n+1}}(P_2)$ is close to the center for all the $\Delta^{j_m}_{q_m}$, it does not move much upon the application of $h_m$ for any $m<n$. Precisely, the horizontal separation of $H_{n-1}(R^t_{\a_{n+1}}(P_2))$ from $c$ is bounded from above by
$$\big(\frac{q_{n-1}}{q_n}\big)\big(\frac{q_{n-3}}{q_{n-2}}\big)\ldots\big(\frac{q_{1}}{q_2}\big)\quad\text{if }n\text{ is even} \qquad \text{or}\qquad\big(\frac{q_{n-1}}{q_n}\big)\big(\frac{q_{n-3}}{q_{n-2}}\big)\ldots\big(\frac{q_{2}}{q_3}\big)\big(\frac{1}{q_1}\big)\quad\text{if }n\text{ is odd.}$$
In either case we observe that the horizontal separation of $H_{n-1}(R^t_{\a_{n+1}}(P_2))$ from $c$ is at most $\frac{q_{n-1}}{q_n}<\frac{1}{q_{n-1}}$ away from $c$ which does not move upon application of $H_{n-1}$. A similar argument shows that the vertical separation of $H_{n-1}(R^t_{\a_{n+1}}(P_2))$ from $c$ is bounded above by $\frac{1}{8q_{n-1}}$. Hence, $H_{n}(R^t_{\a_{n+1}}(P_1))$ and $H_{n}(R^t_{\a_{n+1}}(P_2))$ are at least $\e$ apart.

\smallskip

\noindent\textit{Strategy II:} Let $P_1$ and $P_2$ be two points in $\Xi_n(y)\subset \Xi_n$. Then assuming without loss of generality $\pi_1(P_1)<\pi_1(P_2)$, where $\pi_1:\mathbb{T}^2\to\mathbb{T}$ is defined as $\pi_1(x,y)=x$. Notice that there exists $0\leq t< q_{n+1}$ such that both the conditions below are simultaneously satisfied
\begin{align*}
    (1) \qquad & 0<1- \frac{2\e_n}{q_n^2}-\frac{\e_{n-1}}{q_n^2} < \pi_1\left(R^t_{\a_{n+1}}(P_1)\right)\leq 1- \frac{2\e_n}{q_n^2},\\
    (2) \qquad & 0<1- \frac{\e_n}{q_n^2} < \pi_1\left(R^t_{\a_{n+1}}(P_2)\right)\quad\text{or}\quad\pi_1\left(R^m_{\a_{n+1}}(P_2)\right)< \frac{\e_n}{q_n}.
\end{align*}

Note that the horizontal separation for two points belonging to $\Xi(y)$ is bounded below by $\frac{2\e_n}{q_n^2}$ and above by $\frac{\e_n}{q_n}$, and hence the above set of conditions can always be guaranteed.   So $R^t_{\a_{n+1}}(P_2)$ belongs to a zone where $h_m$ acts as the identity transformation for any $m\leq n$. Hence $H_n(R^t_{\a_{n+1}}(P_2))=R^t_{\a_{n+1}}(P_2)$, while $h_n(R^t_{\a_{n+1}}(P_1))$ rotates by $90$ degrees to the top and into the identity zone of $h_m$ for any $m<n$ because of the monotonicity of the sequence $\{\e_m\}_{m\in \N}$ (see figure \ref{figure separation untwisted}). Hence $H_n(R^t_{\a_{n+1}}(P_1))=h_n(R^t_{\a_{n+1}}(P_1))$, and as a consequence $|H_n(R^t_{\a_{n+1}}(P_1))-H_n(R^t_{\a_{n+1}}(P_2))|\geq \frac{1}{4}$.

\smallskip

\noindent\textit{Strategy for horizontal separation:} In the situation, where $P_1\in\Xi(y_1)$ and $P_2\in\Xi(y_2)$ with $y_1\neq y_2$, separation comes from the difference in $y$'s coordinates in $\Xi$'s construction. Indeed, we can choose $0\leq t\leq q_{n+1}$ such that both $P_1$ and $P_2$ lie in an area where $h_m$ acts as the identity transformation for any $m\leq n$.

In conclusion, we note that $H_n(\Xi_n)$ forms an $(q_{n+1},\e)$-separated set for $T_n$ of cardinality $C\cdot \frac{1}{\e}\cdot q_n$.
\end{proof}

\begin{lemma}\label{lem:FSWSeperation}
For any given $\e>0$ and any $n\in\N$, we have
$$S_{d^T_{q_{n+1}}}(\e)\geq C\frac{q_n}{\e},$$
when $C>0$ is some constant independent of $n$ and $\e$.
\end{lemma}

\begin{proof}
Follows immediately from Lemma \ref{lem:FSWSeperationAtN-thStage} and Proposition \ref{prop:upgradeTntoTForSeparatedSets}.
\end{proof}

\subsection{Upper bounds for cardinality of minimal covering sets}

\begin{lemma} \label{lem:FSWSpanningAtn-thStage}
For any given $\e>0$ and any $n\in\N$, we have for any integer $m\geq 0$ that
$$N_{d^{T_n}_m}(\e)\leq C_n\frac{q_n}{\e^2},$$
where $C_n>0$ is a constant that dependents on $\e_n$ and $H_{n-1}$ but independent of $m$ and $\e$.
\end{lemma}

\begin{proof}
Note that from Lemma \ref{lem:estimateQ(n+1)vsQn}, $\prod_{i=1}^{n-1}q_i^4<<q_n$. Let $d_n$ be the largest integer such that
	\begin{equation} \label{eq:dH0}
	\norm{DH_{n-1}}_0 \cdot \frac{d_n}{q_n} \leq \frac{\e}{4}.
	\end{equation}
	Since $d_n$ is the largest possible integer satisfying this condition, we also get
	\begin{equation}\label{eq:d0}
	d_n \geq \frac{\e}{4 \cdot \norm{DH_{n-1}}_0} \cdot q_n -1 \geq \frac{\e}{8 \cdot \norm{DH_{n-1}}_0} \cdot q_n
	\end{equation}
	for $n$ sufficiently large. Moreover, we define
	\begin{equation}\label{eq:rnxrny}
r^{(x)}_n \coloneqq \frac{\varepsilon}{8q^2_n \cdot \norm{D\varphi_{\e_n}}_0 \cdot \norm{DH_{n-1}}_0} \ \text{ and } \  r^{(y)}_n \coloneqq \frac{\varepsilon}{8\cdot \norm{D\varphi_{\e_n}}_0  \cdot \norm{DH_{n-1}}_0}.
\end{equation}
	Using these numbers we define the following sets
	$$B_{i_1,i_2,i_3} \coloneqq \bigcup^{d_n-1}_{j=0} \Big[ \frac{i_1d_n +j}{q_n}+ i_2 \cdot r^{(x)}_n, \frac{i_1d_n +j}{q_n}+ (i_2+1) \cdot r^{(x)}_n \Big) \times \Big[ i_3 \cdot r^{(y)}_n , (i_3 +1) \cdot r^{(y)}_n \Big)$$
	for $0 \leq i_1 < \lfloor \frac{q_n}{d_n} \rfloor$, $0\leq i_2 < \lfloor \frac{1}{q_n \cdot r^{(x)}_n} \rfloor$, and $0\leq i_3 < \lfloor \frac{1}{r^{(y)}_n} \rfloor$.
	Notice that \eqref{eq:hnUntwisted} and \eqref{eq:rnxrny} imply that
	$$\operatorname{diam}\left( h_n \left( [i \cdot r^{(x)}_n , (i+1) \cdot r^{(x)}_n) \times [j \cdot r^{(y)}_n, (j+1) \cdot r^{(y)}_n ) \right) \right)
	 <\frac{\e}{2 \cdot \norm{DH_{n-1}}_0}$$
	for any $i,j \in \Z$. Exploiting the $1/q_n$-equivariance of $h_n$, this yields for any $t\in \N$ that
	$$\operatorname{diam} \left( T^t_n \left(H_n(B_{i_1,i_2,i_3})\right) \right) = \text{diam} \left(H_{n-1} \circ h_n \circ R^t_{\a_{n+1}} (B_{i_1,i_2,i_3})\right) <\e$$
	by the definition of $d_n$ in (\ref{eq:dH0}). Hence, the points in a set $H_n(B_{i_1,i_2,i_3})$ lie in one $\e$-Bowen ball. Thus, for any $m\in \N$ we obtain
	$$N_{d^{T_n}_m} (\e)\leq \frac{q_n}{d_n} \cdot \frac{1}{q_nr^{(x)}_n} \cdot \frac{1}{r^{(y)}_n} \leq \frac{1}{\varepsilon^3}\cdot \left(8\cdot \norm{DH_{n-1}}_0\right)^3 \cdot \norm{D\varphi_{\e_n}}_0^2 \cdot q_n,$$
	where we used (\ref{eq:d0}) in the last step.
\end{proof}

\begin{lemma}\label{lem:FSWSpanning}
For any given $\e>0$ and any $n\in\N$, we have
\begin{align*}
N_{d^{T}_m}(\e)\leq C_n\frac{q_n}{\e^2}
\end{align*}
where $0\leq m \leq l_{n+1}'q_{n+1}$ and $C_n>0$ is a constant that dependents on $\e_n$ and $H_{n-1}$ but is independent of $m$ and $\e$.
\end{lemma}

\begin{proof}

Follows immediately from Lemma \ref{lem:FSWSpanningAtn-thStage} and Proposition \ref{prop:upgradeTntoTForSpanningSets}.
\end{proof}

\subsection{Upper topological slow entropy}

\begin{theorem}\label{thm:untwistedTopInter}
There exists an untwisted ergodic $C^\infty$ Anosov-Katok diffeomorphism $T$ isomorphic to an irrational translation of a circle constructed using parameters specified in \eqref{eq:abcparameters} with $m_n=n-1$, $l_n'$ as in \eqref{eq:ln'} with $r=4$, $l_n$ satisfying \eqref{eq:equation ln condition}, and conjugacies specified by \eqref{eq:hnUntwisted} and \eqref{eq:anosov-katokUntwisted}, such that the upper topological slow entropy is as follows,
\begin{equation}
\begin{aligned}
& \uent^{\top}_{a_m^{\operatorname{int1,4}}(t)}(T)=1  \\
\end{aligned}
\end{equation}
\end{theorem}
\begin{remark}\label{rem:compareDifferentScales}
Since Proposition \ref{prop:differentScales} implies that intermediate scale $a_m^{\operatorname{int1,4}}$ is faster than logarithmic scale but slower than polynomial scale, the fact that the AbC diffeomorphism's upper topological slow entropy is a finite positive number in the scale $a_m^{\operatorname{int1,4}}$ guarantees that its polynomial upper topological slow entropy is zero and logarithmic upper topological slow entropy is infinity.
\end{remark}
\begin{proof}
The proof of the theorem essentially follows by using Lemmas \ref{lem:FSWSeperation} and \ref{lem:FSWSpanning} to get estimates for the cardinality of $(m,\e)$ minimal covering and maximal separated sets. 

In fact for any $t\leq 1$, Lemma \ref{lem:FSWSeperation} and \eqref{eq:intermediateScaleAtQ(n+1)} give that
$$\limsup_{m\to\infty}\frac{S_{d_m^T}(\e)}{a_m^{\operatorname{int1,4}}(t)}\geq \limsup_{n\to\infty}\frac{S_{d_{q_{n+1}}^T}(\e)}{a_{q_{n+1}}^{\text{int1,4}}(t)} \geq  \limsup_{n\to\infty}\frac{Cq_{n}}{\e q_n^t}>0.$$
Thus it is clear that $\uent^{\top}_{a_m^{\operatorname{int1,4}}(t)}(T)\geq 1$ in this case.

On the other hand, for any $m$ and $q_n < m\leq l_n'q_{n}=q_{n}^{(n-1)^4}$, Lemma \ref{lem:FSWSpanning} and \eqref{eq:intermediateScaleAtQ(n+1)} guarantee that
$$\frac{N_{d_{m}^T}(\e)}{a_{m}^{\operatorname{int1,4}}(t)}\leq \frac{N_{d_{m}^T}(\e)}{a_{q_n}^{\operatorname{int1,4}}(t)}\leq \frac{C_{n-1}q_{n-1}}{\e^2q_{n-1}^t}.$$
While if $q_{n}^{(n-1)^4}=l_n'q_n<m\leq q_{n+1}$, Lemma \ref{lem:FSWSpanning} and \eqref{eq:intermediateScaleAtQ(n)} give that
$$\frac{N_{d_{m}^T}(\e)}{a_{m}^{\operatorname{int1,4}}(t)}\leq \frac{N_{d_{m}^T}(\e)}{a_{l_n'q_n}^{\operatorname{int1,4}}(t)}\leq \frac{C_{n}q_{n}}{\e^2\big[q_n\big]^{\frac{(n-1)^4t}{n^4}}}.$$
Hence it is clear that $\uent^{\top}_{a_m^{\operatorname{int1,4}}(t)}(T)\leq 1$ in this case.

Altogether, we conclude $\uent^{\top}_{a_m^{\operatorname{int1,4}}(t)}(T)= 1$.
\end{proof}

Recall in Lemma \ref{lem:estimateQ(n+1)vsQn}, with $m_n=n-1$, we can make any choice for $q_{n+1}$ as long as it is greater than $q_n^{n^4}$, and all the estimates still remains valid. The proofs of the following two theorems are almost identical with above one and thus we omit them.

\begin{theorem}\label{thm:untwistedTopIn}
There exists an untwisted ergodic $C^\infty$ Anosov-Katok diffeomorphism $T$ isomorphic to an irrational translation of a circle constructed using parameters specified in \eqref{eq:abcparameters} with $q_{n+1}=q_n^{q_n}$, $m_n=n-1$, $l_n'=q_n^{q_n-n}$, $l_n=q_n^{q_n-2}$, and conjugacies specified by \eqref{eq:hnUntwisted} and \eqref{eq:anosov-katokUntwisted}, such that the upper topological slow entropy with respect to the log scale is as follows,
$$\uent^{\top}_{a^{\ln}_m(t)}(T)=1.$$
\end{theorem}

On the other hand, it is possible to get nonzero finite upper topological slow entropy with respect to the polynomial scale by slowing down the speed of convergence. Unfortunately this also results in a lower regularity AbC diffeomorphism.
\begin{theorem}\label{thm:untwistedTopPol}
For any integer $K>1$, there exists an untwisted ergodic $C^{K-1}$ Anosov-Katok diffeomorphism $T$ isomorphic to an irrational translation of a circle constructed using parameters specified in \eqref{eq:abcparameters} with $q_{n+1}=q_n^{K^4}$, $m_n=K-1$, $l_n'=q_n^{K^4-K-3}$, $l_n=q_n^{K^4-2}$ satisfying \eqref{eq:equation ln condition}, and conjugacies specified by \eqref{eq:hnUntwisted} and \eqref{eq:anosov-katokUntwisted}, such that the upper topological slow entropy with respect to the polynomial scale is as follows,
$$\frac{1}{K^4}\leq \uent^{\top}_{a_m^{\pol}(t)}(T)<\frac{1}{K^4-K-2}.$$
\end{theorem}

\begin{remark}
We conclude this section by observing that the separation mechanism described in this section can be modified to obtain uniquely ergodic and weakly mixing examples. However, in the next section we describe a different separation mechanism for unique ergodicity and in Section \ref{sec:weakMixingAbC} we describe a third mechanism for separation. It is worth to notice that the separation mechanism described in the third section allows computation of measure-theoretic slow entropy in addition to upper topological slow entropy.
\end{remark}

\section{Topological slow entropy for some uniquely ergodic AbC diffeomorphisms}\label{sec:uniquelyErgodicAbC}

\subsection{The AbC construction}
In the uniquely ergodic version of the Fayad-Saprykina-Windsor \cite{FSW} AbC construction on $\T^2$, the conjugation map $h_n$ at the $n$-th stage is given by
$$h_{n}(x,y)=\left(\frac{1}{q_n}[\varphi_{\e_n}]_1(q_nx, y+q_nx), [\varphi_{\e_n}]_2(q_nx,y+q_nx)\right).$$
The general idea here is to apply a `shearing' in the vertical direction to enable the control over all orbits for ``most'' of the time. In our explicit construction, some adjustments have been made to the `shearing' component of the conjugating diffeomorphism, which will destroy some equivariance, that may arise in the original construction, in order to simplify the computations of slow entropy.

More precisely, our conjugation diffeomorphism is defined as
\begin{equation}\label{eq:h_n-ue}
h_{n}(x,y)=\phi_{q_n,\e_n} \circ D_{\psi_n} (x,y),
\end{equation}
where $\phi_{q_n,\e_n}$ is the quasi-rotation from Section \ref{subsec:Anosov-Katok} and the diffeomorphism $D_{\psi_n}:\T^2 \to \T^2$ is defined by $D_{\psi_n}(x,y)=(x,y+\psi_n(x))$ with a $C^{\infty}$ function $\psi_n:\T^1\to [0,1]$. Altogether this gives
$$h_{n}(x,y)= \left(\frac{1}{q_n}[\varphi_{\e_n}]_1(q_nx, y+\psi_{n}(x)), [\varphi_{\e_n}]_2(q_nx,y+\psi_n(x))\right).$$

While $\psi_n(x)=q_nx$ in \cite{FSW}, in our case the function $\psi_{n}(x)$ is a smooth approximation of a suitably chosen step function. We start the construction of $\psi_{n}(x)$ with a step function $\tilde{\psi}_{s,\e}(x)$ for any given $\e>0$ and $s\in \N$ of the form:
\begin{equation}
\begin{aligned}
\tilde{\psi}_{s,\e}(x) & = -3\e \cdot \left(\sum_{i=0}^{\lfloor\frac{1}{3\e}\rfloor-1}i \cdot\chi_{[is,(i+1)s)} + \sum_{i=\lfloor\frac{1}{3\e}\rfloor}^{2\lfloor\frac{1}{3\e}\rfloor-2} \left(2\lfloor\frac{1}{3\e}\rfloor-2-i\right) \cdot \chi_{[is,(i+1)s)} \right) \\
& = -3\e \cdot \left(\sum_{i=1}^{\lfloor\frac{1}{3\e}\rfloor-1}\chi_{[is,\infty)} - \sum_{i=\lfloor\frac{1}{3\e}\rfloor}^{2\lfloor\frac{1}{3\e}\rfloor-2} \chi_{[is,\infty)} \right),
\end{aligned}
\end{equation}
i.e. $s$ is the ``step length'' with $\tilde{\psi}_{s,\e}$ attaining a constant value. It is worth to point out that $\tilde{\psi}_{s,\e}$ can be considered to be a map from $[0,as]$ to $[0,1]$,
where
$$a=2\lfloor \frac{1}{3\e}\rfloor -1.$$

In order to approximate $\tilde{\psi}_{s,\e}$, we define $\rho: \mathbb{R}\to\mathbb{R}$ be an increasing smooth function that equals $0$ for $x\leq -1$ and $1$ for $x\geq 1$. Then we define the map  $\overline{\psi}_{s,\varepsilon}: \left[0,as\right]\rightarrow [0,1]$ by
\begin{equation}
\overline{\psi}_{s,\varepsilon}\left(x\right)= -3\e \cdot \left(\sum^{\lfloor\frac{1}{3\e}\rfloor-1}_{i=1} \rho \left(\frac{x}{\varepsilon}-\frac{is}{\varepsilon}\right) - \sum_{i=\lfloor\frac{1}{3\e}\rfloor}^{2\lfloor\frac{1}{3\e}\rfloor-2}\rho \left(\frac{x}{\varepsilon}-\frac{is}{\varepsilon}\right) \right).
\end{equation}
We note that $$\overline{\psi}_{s,\varepsilon}|_{\left[ks+\e,(k+1)s-\e \right]} = \tilde{\psi}_{s,\varepsilon}|_{\left[ks+\e,(k+1)s-\e \right]} $$ for every $0 \leq k \leq a-1$. Furthermore, we have the estimates that
\begin{equation}\label{eq:stepDeriEst}
\left\|D^l\overline{\psi}_{s,\varepsilon}\right\|_0 \leq \frac{3}{\varepsilon^{l-1}} \cdot \left\|D^l\rho\right\|_0,
\end{equation}
for any $l\in \N$. Here $\|\cdot\|_0$ refers to the standard supremum norm.
\begin{figure}[H]
\begin{center}
\includegraphics[trim=0cm 0cm 0cm 0cm, clip, width=14cm,height=7cm]{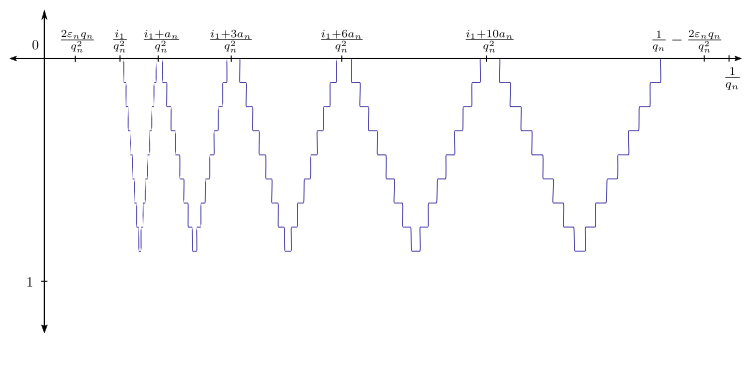}
\caption{A schematics diagram showing the function $\psi_n$.}
\label{figure psi}
\end{center}
\end{figure}

In our specific situation we consider $\overline{\psi}_{s,\varepsilon_n}$ defined on $[0,a_ns]$, where $a_n =2\lfloor \frac{1}{3\e_n}\rfloor -1$. Then we define $\psi_n:[0,\frac{1}{q_n}] \to [0,1]$ by
\[
\psi_n(x) = \begin{cases}
\tau_{n,s}(x), & \text{ if } x\in \left[ \frac{a_n \cdot s \cdot (s-1)+2i_1}{2q^2_n}, \frac{a_n \cdot s \cdot (s+1)+2i_1}{2q^2_n} \right] \text{ for some } s\in \{1,2,\dots ,s_1\}, \\
0 & \text{ otherwise},
\end{cases}
\]
where $\tau_{n,s}(x)=\overline{\psi}_{s,\e_n}\left(q^2_nx-i_1 - a_n \frac{s\cdot (s-1)}{2}\right)$ and the numbers $i_1,s_1\in \N$ are chosen such that $i_1\geq \lceil 2\varepsilon_nq_n \rceil$ and $i_1 + a_n \frac{s_1 (s_1+1)}{2} \leq q_n-\lceil 2\varepsilon_nq_n \rceil$ (see figure \ref{figure psi}). In particular, we have that $\psi_n=\text{id}$ on $\left[0,\frac{2\varepsilon_n}{q_n}\right]\cup \left[\frac{1-2\varepsilon_n}{q_n},\frac{1}{q_n}\right]$. Since $\psi_n$ coincides with the identity in a neighborhood of the boundary, we can extend it to a $C^{\infty}$ map $\psi_n : \T \to [0,1]$ with period $\frac{1}{q_n}$.

With the above conjugating diffeomorphisms, we define the AbC conjugacies as
\begin{equation}\label{eq:anosov-katok-uniquelyErgodic}
    T_n=H_n\circ R_{\a_{n+1}}\circ H_n^{-1},\qquad\text{where}\qquad H_n=h_1\circ\ldots\circ h_{n}.
\end{equation}

Similar to \cite{FSW} we obtain the following theorem using lemma \ref{lem:convergenceOfAbC}.
\begin{theorem}\label{the:ue}
The sequence of diffeomorphisms $\{T_n\}_{n\in \N}$ described in \eqref{eq:anosov-katok-uniquelyErgodic} converges to a uniquely ergodic diffeomorphisms $T$ of the torus that is measure-theoretical isomorphic to a circle rotation.
\end{theorem}

\subsubsection{Norm estimates and parameter growth}

\begin{lemma}\label{lem:boundOnDerivativeUniquelyErgodicCase}
The conjugating diffeomorphisms satisfy the following norm estimates:
\begin{align*}
\max\{\vertiii{h_n}_k, \vertiii{h_n^{-1}}_k\}\leq C q_n^{4k^2+4k},
\end{align*}
where the constant $C$ is dependent on $k$ and $\e_n$  but not on $q_n$.
\end{lemma}

\begin{proof}
Recall that the conjugation map $h_n=\phi_{q_n,\e_n} \circ D_{\psi_n}$. Then the proof follows from Lemma \ref{lem:submultiplicative} and Lemma \ref{lem:derivativeNormEstimateForQuasiRotations} and following estimate, which is a direct consequence of \eqref{eq:stepDeriEst} for any $k\in \N$:
\begin{equation}\label{eq:normPsi}
\norm{D^k\psi_n}_0 \leq q^{2k}_n \cdot \norm{D^k \overline{\psi}_{\e_n}}_0 \leq q^{2k}_n \cdot \frac{3}{\e^{k-1}_n} \cdot \norm{D^k \rho}_0.
\end{equation}
\end{proof}

Following methods similar to those described in Section \ref{sec:untwistedAbC}, we obtain the following lemma.
\begin{lemma}\label{lem:estimateQ(n+1)vsQnUniquelyErgodic}
The expression $\lceil\vertiii{H_{n}}_1\rceil\max\{\vertiii{H_{n}}^{m_n+1}_{m_n+1}, \vertiii{H_{n}^{-1}}^{m_n+1}_{m_n+1}\}$ is bounded above by $q_n^{(m_n+1)^5-2}$ and hence any choice of $l_n\geq q_{n}^{(m_n+1)^5-2}$ satisfies the requirement imposed by \ref{eq:equation ln condition}. Additionally if $k_n=1$, we get $$q_{n+1}\geq q_n^{(m_n+1)^5}.$$ In particular we can choose
\begin{itemize}
    \item $m_n=n-1$ and $l_n=q_n^{n^5-2}$ yielding $q_{n+1}= q_n^{n^5}=q_1^{\prod_{m=1}^n(m^5)},$ or,
    \item $m_n=n-1$ and $l_n=q_n^{q_n-2}$ yielding $q_{n+1}= q_n^{q_n},$ or,
    \item $m_n=K-1$ for some integer $K>1$ and $l_n=q_n^{K^5-2}$ yielding $q_{n+1}= q_n^{K^5}=q_1^{K^{5n}}.$
\end{itemize}
Moreover, if $m_n\geq 1$, then we have the freedom for the choice of a sequence $l_n'$ to satisfy \eqref{eq:equation ln' convergence} and \eqref{eq:equation ln' are summable}.
\end{lemma}
\begin{proof}
Proceeding as in proof of Lemma \ref{lem:estimateQ(n+1)vsQn}, denote $\gamma_n=m_n+1$, then the proof follows from the following estimates with the aid of Lemma \ref{lem:submultiplicative} and Lemma \ref{lem:boundOnDerivativeUniquelyErgodicCase}:
\begin{equation}
\begin{aligned}
& \vertiii{H_{n}}_1 \max\{\vertiii{H_{n}}^{\gamma_n}_{\gamma_n}, \vertiii{H_{n}^{-1}}^{\gamma_n}_{\gamma_n}\} \\
\leq &\Big[\vertiii{H_{n-1}^{-1}}_1 \max\{\vertiii{H_{n-1}}^{\gamma_n}_{\gamma_n}, \vertiii{H_{n-1}^{-1}}^{\gamma_n^2}_{\gamma_n}\}\Big] \cdot \vertiii{h_{n}}_1\cdot \max\{\vertiii{h_{n}}^{\gamma_n^2}_{\gamma_n}, \vertiii{h_{n}^{-1}}^{\gamma_n}_{\gamma_n}\} \\
\leq &\Big[\vertiii{H_{n-1}^{-1}}_1 \max\{\vertiii{H_{n-1}}^{\gamma_n}_{\gamma_n}, \vertiii{H_{n-1}^{-1}}^{\gamma_n^2}_{\gamma_n}\}\Big]\cdot  q_n^{4\gamma_n^4+5\gamma_n^3+\gamma_n^2+5}\\
< &q^{\gamma_n^5-2}_n,
\end{aligned}
\end{equation}
where the last inequality follows from $q_n>>q_{n-1}$. As a result of this estimate, a choice of $l_n\geq q_n^{\gamma_n^5-2}$ will satisfy \eqref{eq:equation ln condition}. This completes the proof.
\end{proof}

\subsection{Lower bounds for cardinality of maximal separated sets}

To exclude the regions coming from smoothing the approximative step function $\rho$, we introduce
$$B_n \coloneqq \Big(\bigcup_{i\in \Z} \left[\frac{i-\varepsilon_n}{q^2_n} , \frac{i+\varepsilon_n}{q^2_n} \right]\Big)\bigcap\T.$$
Hereby, we define the \emph{good domain} of the conjugation map $h_n =\phi_{q_n,\e_n} \circ D_{\psi_n}$ as
$$L_n \coloneqq (\T \setminus B_n ) \times \T \cap D^{-1}_{\psi_n} \left(K_n \right),$$
where
$$K_n=\bigcup_{i=0}^{q_n-1}\Big[\frac{i+2\e_n}{q_n}, \frac{i+1-2\e_n}{q_n}\Big]\times \Big[2\e_n, 1-2\e_n\Big]$$
is the $n$-th rotation zone, i.e. in this portion of the torus, $\phi_{q_n,\e_n}$ and $\phi_{q_n,\e_n}^{-1}$ act as a pure rotation.
Accordingly, $h_n(L_n)$ is the \emph{good domain} of $h^{-1}_n$.

\begin{lemma}\label{lem:uniquelyErgodicApproLower}
Given any $\e>0$ there is $N\in\N$ such that we have for all $n> N$ that
	$$S_{d^{T_{n}}_{q_{n+1}}}(\e)\geq  0.5 q_{n}.$$
\end{lemma}
\begin{proof}
	For a given $\varepsilon>0$ we let $b_0=\varepsilon=c_0$. Recursively, we define $c_{i+1}=q_{i+1}b_i$ and $b_{i+1}=\frac{c_i}{q_{i+1}}$. Let $N=N(\varepsilon)\coloneqq \max \{i:c_i<1\}$.   We consider
	$$\Theta_N \coloneqq \bigcap^{N}_{n=1}L_n,$$
	i.e. the ``good domain'' of $H_N$. Then we take an interval $\hat{I}_u=\left[ \frac{u}{q_{N+1}}, \frac{u+1}{q_{N+1}} \right]$ lying in $\pi_1(\Theta_N)$  such that $\psi_N=0$ on $\hat{I}_u$.

In the next step we introduce subsets of $S^{(u)}\coloneqq \hat{I}_u \times [2\varepsilon_N, 1-2\varepsilon_N]$ as follows:
\[
S^{(u)}_{v,w} = \Big[ \frac{u+(3v+2)\varepsilon_{N+1}}{q_{N+1}},\frac{u+(3v+4)\varepsilon_{N+1}}{q_{N+1}}\Big) \times \Big[4\e_N+ \frac{w+2\varepsilon_{N+1}}{q_{N+1}}, 4\e_N+ \frac{w+1-2\varepsilon_{N+1}}{q_{N+1}} \Big),
\]
where $0\leq v \leq \lfloor \frac{1}{3\e_{N+1}}\rfloor -2$ and $0\leq w< q_{N+1}-6\e_{N}q_{N+1}$ (see figure \ref{figure Suvw}).

We claim that points from different sets $S^{(u)}_{v,w}$ are $\varepsilon$-separated under $\{H_{N+1} \circ R^t_{\alpha_{N+2}} \circ h^{-1}_{N+1} \}_{0\leq t \leq q_{N+2}}$. The sets $S^{(u)}_{v,w}$ are positioned in such a way that
\begin{equation}
\begin{aligned}
\phi^{-1}_{q_{N+1}} \left( S^{(u)}_{v,w} \right) = & \Big[\frac{u+1-4\e_N}{q_{N+1}}- \frac{w+1-2\varepsilon_{N+1}}{q^2_{N+1}}, \frac{u+1-4\e_N}{q_{N+1}}- \frac{w+2\varepsilon_{N+1}}{q^2_{N+1}} \Big) \\ &\times \Big[ (3v+2)\e_{N+1}, (3v+4)\e_{N+1} \Big),
\end{aligned}
\end{equation}
which lies in the ``good domain'' $(\mathbb{T}\setminus B_{N+1}) \times \mathbb{T}$ of the map $D^{-1}_{\psi_{N+1}}$ . Hence, $-\Psi_{N+1}$ attains a constant value on $\pi_1(\phi^{-1}_{q_{N+1}}(S^{(u)}_{v,w}))$ and we denote this value by $f_w \cdot 3\e_{N+1}$ for some $f_w \in \{0,1,\dots , \lfloor \frac{1}{3\e_{N+1}} \rfloor -1\}$.

In the next step, we choose $t_1=t_1(u)$, $0\leq t_1<q_{N+2}$, such that
\begin{equation}
\begin{aligned}
 R^{t_1}_{\alpha_{N+2}} \circ \phi^{-1}_{q_{N+1}} \left( S^{(u)}_{v,w} \right) = & \Big[\frac{i_c+1-4\e_N}{q_{N+1}}- \frac{w+1-2\varepsilon_{N+1}}{q^2_{N+1}}, \frac{i_c +1 - 4\e_N}{q_{N+1}}- \frac{w+2\varepsilon_{N+1}}{q^2_{N+1}} \Big) \\& \times \Big[ (3v+2)\e_{N+1}, (3v+4)\e_{N+1} \Big),
\end{aligned}
\end{equation}
where $i_c$ is some central index for $H_N$ (i.e. $\pi_2(H_N(\Delta^{i_c}_{q_{N+1}} \cap K_{N}))\subset \left[\frac{1}{4},\frac{3}{4}\right]$ similar to the definition in equation \eqref{eq:centralIndex}). Let $t$ be of the form $t=t_1+\lambda \frac{q_{N+2}}{q_{N+1}}$ with some $0\leq \lambda <2\e_Nq_{N+1}$. Suppose that $D_{\psi_{N+1}} \circ R^t_{\alpha_{N+2}} \circ D^{-1}_{\psi_{N+1}}$ causes a net vertical translation by $g_{w,t}\cdot 3\e_{N+1}$ for some $g_{w,t} \in \{f_w, f_w -1, \dots, f_w -\lfloor \frac{1}{3\e_{N+1}} \rfloor +1 \}$. Then
\begin{align*}
 &D_{\psi_{N+1}} \circ R^t_{\alpha_{N+2}} \circ h^{-1}_{N+1} \left( S^{(u)}_{v,w} \right) \\
= & \Big[\frac{i_c+1-4\e_N}{q_{N+1}}- \frac{w+1-2\varepsilon_{N+1}}{q^2_{N+1}}, \frac{i_c+1-4\e_N}{q_{N+1}}- \frac{w+2\varepsilon_{N+1}}{q^2_{N+1}} \Big) \\
& \times \Big[ 2\e_{N+1}+(g_{w,t}+v)\cdot 3\e_{N+1}, 4\e_{N+1}+(g_{w,t}+v)\cdot 3\e_{N+1} \Big).
\end{align*}
For $g_{w,t}+v=-1$ or $g_{w,t}+v=\lfloor \frac{1}{3\e_{N+1}} \rfloor -1$, i.e. the $y$-coordinate is in $[-\e_{N+1},\e_{N+1}]$, this image lies in the identity zone of $\phi_{q_{N+1}}$ (notice that for each $g_{w,t}$ at most one of these situations can occur). Otherwise, $\phi_{q_{N+1}}$ maps it to
\begin{align*}
& \Big[ \frac{i_c+2\e_{N+1}+(g_{w,t}+v)\cdot 3\e_{N+1}}{q_{N+1}},\frac{i_c+4\e_{N+1}+(g_{w,t}+v)\cdot 3\e_{N+1}}{q_{N+1}}\Big) \\
\times & \Big[4\e_N+ \frac{w+2\varepsilon_{N+1}}{q_{N+1}}, 4\e_N+ \frac{w+1-2\varepsilon_{N+1}}{q_{N+1}} \Big).
\end{align*}
Hence, we see by our choice of $i_c$ that $H_{N+1} \circ R^t_{\alpha_{N+2}} \circ h^{-1}_{N+1}\left( S^{(u)}_{v,w} \right)$ with $g_{w,t}+v=-1$ or $g_{w,t}+v=\lfloor \frac{1}{3\e_{N+1}} \rfloor -1$ is $\e$-separated from the other images $H_{N+1} \circ R^t_{\alpha_{N+2}} \circ h^{-1}_{N+1}\left( S^{(u)}_{v,w} \right)$. Since $D_{\psi_{N+1}} \circ R^t_{\alpha_{N+2}} \circ D^{-1}_{\psi_{N+1}}$ attains net vertical translations $g_{w,t}\cdot 3\e_{N+1}$ with possible values $g_{w,t} \in \{f_w, f_w -1, \dots, f_w -\lfloor \frac{1}{3\e_{N+1}} \rfloor +1 \}$, all $S^{(u)}_{v,w}$, $0\leq v \leq \lfloor \frac{1}{3\e_{N+1}}\rfloor -2$, get $\e$-separated from each other.

In the same way, we explore that for any different $0 \leq w_1,w_2<q_{N+1}-6\e_{N}q_{N+1}$ all $S^{(u)}_{v_1,w_1}$, $0\leq v_1 \leq \lfloor \frac{1}{3\e_{N+1}}\rfloor -2$, get $\e$-separated from all $S^{(u)}_{v_2,w_2}$, $0\leq v_2 \leq \lfloor \frac{1}{3\e_{N+1}}\rfloor -2$: As above, we see that $H_{N+1} \circ R^t_{\alpha_{N+2}} \circ h^{-1}_{N+1}\left( S^{(u)}_{v_1,w_1} \right)$ with $g_{w_1}+v_1=-1$ or $g_{w_1}+v_1=\lfloor \frac{1}{3\e_{N+1}} \rfloor -1$ is $\e$-separated from all the images $H_{N+1} \circ R^t_{\alpha_{N+2}} \circ h^{-1}_{N+1}\left( S^{(u)}_{v_2,w_2} \right)$ apart from those with $g_{w_2}+v_2=-1$ or $g_{w_2}+v_2=\lfloor \frac{1}{3\e_{N+1}} \rfloor -1$. By the differences of sequences $\{g_{w_1,t}\}$ and $\{g_{w_2,t}\}$ of vertical translations caused by $D_{\psi_{N+1}} \circ R^t_{\alpha_{N+2}} \circ D^{-1}_{\psi_{N+1}}$ (due to the varying step lengths in the construction of the map $\Psi_{N+1}$, the sequences $\{g_{w_1,t}\}$ and $\{g_{w_2,t}\}$ get shifted with respect to each other after some time), we conclude the claim.

Counting the number of different sets $S^{(u)}_{v,w}$ we obtain
\[
S_{d^{T_{N+1}}_{q_{N+2}}}(\e)\geq \left(\lfloor \frac{1}{3\e_{N+1}} \rfloor -1\right) \cdot \left(1-6\e_N\right)\cdot q_{N+1} >0.5q_{N+1}.
\]
By the same methods we continue for any $n> N+1$.
\end{proof}

\begin{figure}[H]
\begin{center}
\includegraphics[trim=0cm 0cm 0cm 0cm, clip, width=5cm,height=8.5cm]{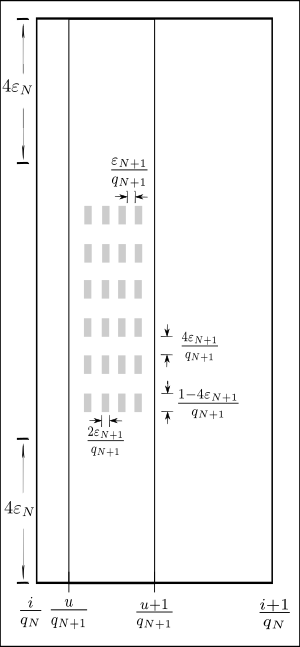}
\caption{A schematic diagram showing the sets $S_{v,w}^{(u)}$.}
\label{figure Suvw}
\end{center}
\end{figure}

Combining Proposition \ref{prop:upgradeTntoTForSeparatedSets} and Lemma \ref{lem:uniquelyErgodicApproLower}, we have:
\begin{lemma}\label{lem:LowerSmooth2}
Given any $\e>0$ there is $N\in\N$ such that we have for all $n> N$ that
$$S_{d^T_{q_{n+1}}}(\e)\geq  0.5 q_n.$$
\end{lemma}

\subsection{Upper bounds for cardinality of minimal covering sets}

\begin{lemma}\label{lem:upperSmooth2}
	For any given $\e>0$ and any $n\in \N$ sufficiently large we have for $m\in\N$ that
	\[
	N_{d^{T_n}_m} (\e)\leq C_n \cdot\frac{q_n}{\e^3},
	\]
	where the constant $C_n$ depends on $\e_n$ and $H_{n-1}$ but is independent of $q_n$, $m$, and $\e$.
\end{lemma}

\begin{proof}
The proof is identical to the proof of Lemma \ref{lem:FSWSeperationAtN-thStage} with equation \ref{eq:rnxrny} replaced by,
\begin{equation}
\begin{aligned}
	&r^{(x)}_n \coloneqq \frac{\varepsilon}{8q^2_n \cdot \norm{D\varphi_{\e_n}}_0 \cdot \norm{D\overline{\psi}_{\e_n}}_0 \cdot \norm{DH_{n-1}}_0}, \\  &r^{(y)}_n \coloneqq \frac{\varepsilon}{8\cdot \norm{D\varphi_{\e_n}}_0 \cdot \norm{D\overline{\psi}_{\e_n}}_0 \cdot \norm{DH_{n-1}}_0}.
\end{aligned}
\end{equation}
\end{proof}

\begin{lemma}\label{lem:FSWSpanning2}
For any given $\e>0$ and any $n\in\N$, we have
$$N_{d^{T}_m}(\e)\leq C_n\frac{q_n}{\e^3}$$
where $0\leq m \leq l_{n+1}'q_{n+1}$ and $C_n>0$ is some constant dependent on $\e_n$ and $H_{n-1}$ but independent of $m$ and $\e$.
\end{lemma}

\subsection{Upper topological slow entropy}

Proceeding similar as in the proof of Theorem \ref{thm:untwistedTopInter}, we obtain the following results by replacing Lemmas \ref{lem:FSWSeperation} and \ref{lem:FSWSpanning} by Lemmas \ref{lem:LowerSmooth2} and \ref{lem:FSWSpanning2}, respectively.
\begin{theorem}\label{thm:uniquelyErgodicTopInter}
There exists an untwisted uniquely ergodic $C^\infty$ Anosov-Katok diffeomorphism $T$ isomorphic to an irrational translation of a circle constructed using parameters specified in \eqref{eq:abcparameters} with $m_n=n-1$, $l_n'$ as in \eqref{eq:ln'} with $r=5$, $l_n$ satisfying \eqref{eq:equation ln condition}, and conjugacies specified by \eqref{eq:h_n-ue} and \eqref{eq:anosov-katok-uniquelyErgodic}, such that the upper topological slow entropy is
\begin{equation}
 \uent^{\top}_{a_m^{\operatorname{int1,5}}(t)}(T)=1.
\end{equation}
\end{theorem}
As observed in Remark \ref{rem:compareDifferentScales}, the above theorem also implies that the given AbC diffeomorphism's logarithmic upper topological slow entropy is infinity and polynomial upper topological slow entropy is zero.

\begin{theorem}\label{thm:uniquelyErgodicTopIn}
There exists an untwisted uniquely ergodic $C^\infty$ Anosov-Katok diffeomorphism $T$ isomorphic to an irrational translation of a circle constructed using parameters specified in \eqref{eq:abcparameters} with $q_{n+1}=q_n^{q_n}$, $m_n=n-1$, $l_n'=q_n^{q_n-n}$, $l_n$ satisfying \eqref{eq:equation ln condition}, and conjugacies specified by \eqref{eq:h_n-ue} and \eqref{eq:anosov-katok-uniquelyErgodic}, such that the upper topological slow entropy with respect to the log scale is
\begin{align*}
        \uent^{\top}_{a^{\ln}_m(t)}(T)=1
\end{align*}
\end{theorem}

\begin{theorem}\label{thm:uniquelyErgodicTopPol}
For any integer $K>1$, there exists an untwisted uniquely ergodic $C^{K-1}$ Anosov-Katok diffeomorphism $T$ isomorphic to an irrational translation of a circle constructed using parameters specified in \eqref{eq:abcparameters} with $q_{n+1}=q_n^{K^5}$, $m_n=K-1$, $l_n'=q_n^{K^5-K-9}$, $l_n=q_n^{K^5-2}$ satisfying \eqref{eq:equation ln condition}, and conjugacies specified by \eqref{eq:h_n-ue} and \eqref{eq:anosov-katok-uniquelyErgodic}, such that the upper topological slow entropy with respect to the polynomial scale is as follows,
$$\frac{1}{K^5}\leq \uent^{\top}_{a_m^{\pol}(t)}(T)<\frac{1}{K^5-K-8}.$$
\end{theorem}

\section{Slow entropy for some weakly mixing AbC diffeomorphisms}\label{sec:weakMixingAbC}

\subsection{The AbC construction}
Using the quantitative version of the AbC method Fayad and Saprykina \cite{FS} constructed weakly mixing diffeomorphisms on $\T^2$, $\mathbb{S}^1\times [0,1]$ and $\mathbb{D}$ with arbitrary Liouville rotation number. Using the shear $g_b(x,y)=(x+by,y)$ into the horizontal direction, their conjugation map is
\begin{align}\label{eq:hn-wm}
h_n(x,y)= g_{[nq^{\sigma}_n]} \circ \phi_{n} (x,y) = \left([\phi_{n}]_1(x,y) + [nq^{\sigma}_n] \cdot [\phi_{n}]_2(x,y), [\phi_{n}]_2(x,y) \right)
\end{align}
for some $0<\sigma<1$. Here, the $1/q_n$-equivariant map $\phi_n$ is built as $\phi_n=\phi_{2q_n,\e_n}$ on $[0,1/(2q_n)]\times [0,1]$ and $\phi_n=\text{id}$ on $[1/(2q_n), 1/q_n]\times [0,1]$.

As in the previous section we modify their construction in order to simplify the slow entropy estimates. For the computation of the topological slow entropy along the lines of the previous section we could have worked with the conjugation map
\begin{equation}
\phi_n= \begin{cases}
\phi_{2q^2_n,\e_n} & \text{ on } \left[\frac{i}{2q^2_n}, \frac{i+1}{2q^2_n}\right] \times [0,1] \text{ for } i=0,1,\dots,q_n-1,\\
\text{id} & \text{ on } \left[\frac{1}{2q_n}, \frac{1}{q_n} \right] \times [0,1].
\end{cases}
\end{equation}
To carry out an exact computation of the upper measure-theoretical slow entropy it proved convenient to modify $\phi_n$ even further. We assign different mapping behavior of $\phi_n$ on distinct domains by some probabilistic procedure. This allows us to show in Lemma \ref{lem:lowerSmooth4} that orbits starting in different domains are Hamming apart from each other which will give us a lower bound on the upper measure-theoretical slow entropy. To get an upper bound we provide in Lemma \ref{lem:upperSmooth4} some $(1-\e)$-cover with $(\e, q_{n+1})$-Hamming balls with respect to a given partition.

\subsubsection{A probabilistic Lemma}

A key ingredient to control the measure-theoretical upper slow entropy in our construction is a probabilistic method similar to the so-called ``Substitution Lemma'' in \cite{FRW}. More precisely, we present a method to select good choices of coding words so that these constructed sequences satisfy strong uniformity and that all pairs of building blocks occur with about the same frequency when comparing two sequences with each other, even after some sliding along the sequence. To state this precisely, we introduce some notation.

\begin{definition}
	Let $\Sigma$ be an alphabet. For a word $w\in\Sigma^{k}$ and $x\in\Sigma$
	we write $r(x,w)$ for the number of times that $x$ occurs in $w$
	and $\freq(x,w)=\frac{r(x,w)}{k}$ for the frequency of occurrences
	of $x$ in $w$. Similarly, for $(w,w')\in\Sigma^{k}\times\Sigma^{k}$
	and $(x,y)\in\Sigma\times\Sigma$ we write $r(x,y,w,w')$ for the
	number of $i<k$ such that $x$ is the $i$-th member of $w$ and
	$y$ is the $i$-th member of $w'$. We also introduce $\freq(x,y,w,w')=\frac{r(x,y,w,w')}{k}$.
\end{definition}
We also state the Law of Large Numbers with its large deviations using Chernoff bounds:
\begin{lemma}[Law of Large Numbers]
	Let $\left(X_{i}\right)_{i\in\mathbb{N}}$ be a sequence of independent
	identically distributed random variables taking value $1$ with probability
	$p$ and taking value $0$ with probability $1-p$. Then for any $\delta>0$
	we have
	\[
	\mathbb{P}\left(\left|\frac{1}{n}\sum_{i=1}^{n}X_{i}-p\right|\geq\delta\right)<\exp\left(-\frac{n\delta^{2}}{4}\right).
	\]
\end{lemma}
Inspired by the proof of the Substitution Lemma in \cite[Proposition 44]{FRW} we apply the Law of Large Numbers to guarantee the existence of selections with the desired properties mentioned above. An even stronger probabilistic lemma was proven in \cite{BKW1} which allowed us to also control the lower measure-theoretical slow entropy of some combinatorial constructions. For the sake of completeness and the reader's convenience we include statement and proof of a probabilistic lemma sufficient for our purposes.

\begin{lemma}\label{lem:prob3}
Let $\epsilon>0$ and $\Sigma$ be a finite alphabet. For any sequence $\{b_n\}_{n\in \N}$ with $\lim_{n\to \infty} \frac{\log b_n}{n} = 0$ there exists $K_0 \in\mathbb{N}$ such that for all $k\geq K_0$, that are multiples of $|\Sigma|$, and all $N\leq b_k$ there is a collection of sequences $\Theta\subset\Sigma^k$ with cardinality $|\Theta|=N$ satisfying the following properties:
\begin{enumerate}[(1)]
  \item(Exact uniformity) For every $x\in\Sigma$ and every $w\in\Theta$, we have
      $$\operatorname{freq}(x,w)=\frac{1}{|\Sigma|};$$
  \item(Hamming separation) Let $0\leq t<(1-\epsilon)k$, $w,w'\in\Theta$ and $I\subset[0,k-1]\cap\mathbb{Z}$ be the indices in the overlap of $w$ and $\sh^t(w')$, where $\sh^t(w')$ moves $w'$'s digits to the left by $t$ units. If $w,w'$ are different from each other, then we have
      \begin{equation}
      d^H_{k-t}(w\upharpoonright I,\sh^t(w')\upharpoonright I)\geq 1-\frac{1}{|\Sigma|}-\epsilon|\Sigma|;
      \end{equation}
      if $1\leq t\leq(1-\epsilon)k$, then we have
      \begin{equation}\label{eq:selfSliding}
      d^H_{k-t}(w\upharpoonright I,\sh^t(w)\upharpoonright I)\geq 1-\frac{1}{|\Sigma|}-\epsilon|\Sigma|,
      \end{equation}
  where $w\upharpoonright I$ denotes the restriction of $w$ on the index set $I$, i.e. if $I=\{i_1,i_2,\ldots\}$, then $(w\upharpoonright I)_p=w_{i_p}$.
\end{enumerate}

\end{lemma}

\begin{proof}

We will use the Law of Large Numbers to show that for sufficiently large $k\in \N$ most choices in $\Sigma^{k}$	satisfy the aimed properties.
	
Let $\delta<\frac{\epsilon^2}{5}$. We consider $\varOmega_{k}\coloneqq\left(\Sigma^{k} \right)^{N}=\Sigma^{k}\times\dots\times\Sigma^{k}$ equipped with the counting measure as our probability space. For each $x\in\Sigma$ and every $i\in\left\{ 0,1,\dots,k-1\right\} $ let $X_{i}$ be the random variable that takes the value $1$ if $x$ occurs in the $i$-th place of an element $w\in\Sigma^{k}$ and $0$ otherwise. Then the $X_{i}$ are independent and identically distributed. Hence, the Law of Large Numbers gives $k_{x}=k_{x}(\delta)$ such that for all $k\geq k_{x}$ a proportion $\left(1-(\exp(-\delta^2/4))^k\right)$ of sequences in $\Sigma^{k}$ satisfy
\begin{equation}\label{eq:prop13}
\abs{\frac{1}{k}\sum_{i=0}^{k-1}X_{i}-\frac{1}{|\Sigma|}}<\delta.
\end{equation}
Moreover, we define for each $0\leq t <(1-\epsilon)k$ and pair $(x,y)\in\Sigma\times\Sigma$ the random variable $Y^t_{i}$ that takes the value $1$ if $x$ occurs in the $i$-th place of $w$ for an element $w\in\Sigma^{k}$ and $y$ is the $i$-th entry of $\sh^t(w')$ for some $w'\in\Sigma^{k}$, $w'\neq w$. Otherwise, $Y^t_{i}$ takes the value $0$. Since the $Y^t_{i}$ are independent and identically distributed, the Law of Large Numbers gives $k_{x,y}=k_{x,y}(\delta)$ such that for all $k\geq k_{x,y}$, all $0\leq t <(1-\epsilon)k$ and all but a $(\exp(-\delta^2/4))^{k-t}$ proportion of sequences $(w,w')\in\Sigma^{k}\times\Sigma^{k}$ satisfy
\begin{equation}\label{eq:prop23}
\abs{\frac{1}{k-t}\sum_{i=0}^{k-t-1}Y^t_{i}-\frac{1}{|\Sigma|^{2}}}<\delta.
\end{equation}
Finally, we introduce for each $1\leq t <(1-\epsilon)k$ the random variable $Z^t_i$ that takes the value $1$ if the $i$-th symbol of $w$ agrees with the $i$-th symbol of $\sh^t(w)$ for some $w\in \Sigma^k$. Otherwise, $Z^t_i$ takes the value $0$. Since the $Z^t_{i}$ are independent and identically distributed, the Law of Large Numbers gives $k(\delta)$ such that for all $k\geq k(\delta)$ and all $1\leq t <(1-\epsilon)k$  a $\left(1-(\exp(-\delta^2/4))^{k-t}\right)$ proportion of sequences $w\in\Sigma^{k}$ satisfy
\begin{equation}\label{eq:prop33}
\abs{\frac{1}{k-t}\sum_{i=0}^{k-t-1}Z^t_{i}-\frac{1}{|\Sigma|}}<\delta.
\end{equation}
	
We point out that the number of requirements is less than $N |\Sigma|+2kN^2|\Sigma|^{2} \leq 3kN^2|\Sigma|^{2}$. Since
\[
3kN^2|\Sigma|^{2} \cdot (\exp(-\delta^2/4))^{\epsilon k} \leq 3kb_k^2|\Sigma|^{2} \cdot (\exp(-\delta^2/4))^{\epsilon k} \to 0 \text{ as } k\to \infty,
\]
we conclude by Bernoulli inequality that for sufficiently large $k$ the vast majority of elements in $\varOmega_{k}$ satisfies all the conditions \eqref{eq:prop13}, \eqref{eq:prop23}, and \eqref{eq:prop33}. We pick $k$ large enough such that there is such a collection of sequences $\Theta'\subset\Sigma^{k}$ with cardinality $|\Theta'|=N$. Then by equation \eqref{eq:prop13} in any $w_{in}\in\Theta^{\prime}$, we can remove symbols at at most $2\delta k$ places to obtain a word $w_{red}$ in which each element of $\Sigma$ occurs the same number of times. Afterwards, each element of $\Sigma$ can be filled into the empty slots exactly the same number of times. Clearly, the constructed word $w$ satisfies uniformity. The sequences built this way constitute our collection $\Theta\subset\Sigma^{k}$.

To check the second property we denote for $w,w^{\prime}\in\Theta$ their original strings in $\Theta^{\prime}$ by $w_{in}$ and $w_{in}^{\prime}$, respectively. From equation \eqref{eq:prop23} we obtain for every $x,y\in\Sigma$ that
	\[
	\abs{\frac{r(x,y,w_{in}\upharpoonright I,\sh^t(w'_{in})\upharpoonright I)}{k-t}-\frac{1}{|\Sigma|^{2}}}<\delta.
	\]
Since $w_{in}$ and $\sh^t(w'_{in})$ were changed at most $2\delta k$ places, at most $4\delta k$ positions in the alignment of $w$ and $\sh^t(w')$ can be affected. Hereby, we conclude that
\begin{equation*}
\begin{aligned}
	& \abs{\freq\left(x,y,w\upharpoonright I,\sh^t(w')\upharpoonright I\right)-\frac{1}{|\Sigma|^{2}}}\\
	\leq & \abs{ \frac{r(x,y,w\upharpoonright I,\sh^t(w')\upharpoonright I)-r(x,y,w_{in}\upharpoonright I,\sh^t(w'_{in})\upharpoonright I)}{k-t}} \\&+ \abs{\frac{r(x,y,w_{in}\upharpoonright I,\sh^t(w'_{in})\upharpoonright I)}{ k-t} -\frac{1}{|\Sigma|^{2}}}\\
	\leq & \frac{4\delta  k}{ \epsilon k}+\delta	< \epsilon.
\end{aligned}
\end{equation*}

In particular, this implies
\[
d^H_{k-t-1}\left(w\upharpoonright I,\sh^t(w')\upharpoonright I\right)\geq 1-\frac{1}{|\Sigma|}-\epsilon |\Sigma|,
\]
which yields the first part of property (2). Similarly, we check its second part with the aid of (\ref{eq:prop33}).

\end{proof}

To fix some notation we make the following immediate observation from Lemma \ref{lem:prob3}.

\begin{remark} \label{rem:kForCLT2}
	Let $s\in \Z$, $s\geq 2$, and $t\in \R_{\geq 0}$. Then there is $K_1=K_1(t,s,\varepsilon)$ such that for all $k\geq K_1$ and any finite alphabet $\Sigma$ of cardinality $|\Sigma|=s$ there is a collection $\Theta \subset \Sigma^{k}$ of cardinality $|\Theta|=\lfloor k^t\rfloor $ such that the words in $\Theta$ satisfy the properties from Lemma \ref{lem:prob3}.
\end{remark}

\subsubsection{Construction of the conjugation maps}
In our case, let $q_n$ be large enough such that there is a collection $\Theta$ of $q_n$ many words of length $q_n$ in the alphabet $\{0,\dots ,n^2-1\}$ satisfying the properties in Lemma \ref{lem:prob3} with $\epsilon=\frac{1}{\abs{\Sigma}^2}=\frac{1}{n^4}$. This corresponds to $q_n \geq K_1(1,n^2,\frac{1}{n^4})$ in the notation from Remark \ref{rem:kForCLT2}. Then we concatenate these words from $\Theta$ to form a word $\overline{w}$ of length $q^2_n$. With it we introduce the word $\tilde{w}$ of length $q^2_n$ by $\tilde{w}_i = \overline{w}_i + 1 \mod n^2$. Then we concatenate these two words and obtain a word $W=\overline{w}\tilde{w}=W_0\dots W_{2q^2_n-1}$. We use this word to define the conjugation map $\phi_n$ on $\left[\frac{i}{2q^3_n}, \frac{i+1}{2q^3_n}\right] \times [0,1]$ for $i=0,1,\dots,2q^2_n-1$  via the following method:
\begin{enumerate}[(i)]
\item If $W_i=0$, then $\phi_n=id$;
\item If $W_i=j\in\{1,\ldots, n^2-1\}$, then $\phi_n=\phi_{2q^{j+2}_n,\e_n}$ on each $\left[\frac{i}{2q^3_n}+\frac{s}{2q^{j+2}_n}, \frac{i}{2q^3_n}+\frac{s+1}{2q^{j+2}_n}\right] \times [0,1]$ for $s=0,1,\ldots, q_n^{j-1}$.
\end{enumerate}
Thereby, $\phi_n$ is defined on the fundamental domain $\left[0,\frac{1}{q_n}\right] \times [0,1]$. Finally, we extend it $1/q_n$-equivariantly. By Lemma \ref{lem:derivativeNormEstimateForQuasiRotations} we have
\begin{equation}\label{eq:normPhi3}
|||\phi_n|||_k \leq C_{n,k} \cdot q^{(n^2+1)\cdot k}_n \text{ for every } k\in \N,
\end{equation}
where $C_{n,k}$ is a constant depending on $k$ and $\e_n$ but is independent of $q_n$.

If $\phi_n =\phi_{2q^{j+2}_n,\e_n} $ on $\left[\frac{i}{2q^3_n}, \frac{i+1}{2q^3_n}\right] \times [0,1]$, then we set
\begin{align*}
\tilde{K}^{(i)}_n = \bigcup^{q^{j-1}_n -1}_{s=0} \left[\frac{i}{2q^3_n} + \frac{s+2\e_n}{2q^{j+2}_n}, \frac{i}{2q_n^3} + \frac{s+1-2\e_n}{2q^{j+2}_n} \right] \times \left[2\e_n,1-2\e_n\right],
\end{align*}
while if $\phi_n=\text{id}$, then we set
\begin{align*}
\tilde{K}^{(i)}_n = \left[\frac{i}{2q^3_n}, \frac{i+1}{2q^3_n}\right] \times \left[2\e_n,1-2\e_n\right].
\end{align*}
Then we use this to introduce
\begin{align*}
K^{(0)}_n=\bigcup_{i:W_i=0} \tilde{K}^{(i)}_n \ \ \ \text{ and } \ \ \ K^{(j)}_n = \bigcup_{i:W_i=j} \tilde{K}^{(i)}_n
\end{align*}
for $j\in \{1,\dots,n^2-1 \}$. Altogether, we define
\begin{equation} \label{eq:Ksmooth3}
K_n = \bigcup^{2q^3_n-1}_{i=0} \tilde{K}^{(i)}_n = \bigcup^{n^2-1}_{j=0} K^{(j)}_n
\end{equation}
as the \emph{good domain} of $\phi_n$ and $\phi^{-1}_n$.

Instead of the shear map $g_b$ we use as the function $g_n$ a smooth approximation of a suitably chosen step function. Moreover, we will choose a variable $\sigma_n \searrow 0$ instead of some fixed $0<\sigma<1$. 
\begin{lemma} \label{lem:G}
	Let $\sigma_n>0$ and $a_n = [nq^{\sigma_n}_n]\cdot 2q^3_n $. There is a smooth measure-preserving diffeomorphism $g_n: \mathbb{T}^1 \times \left[0,1\right]\rightarrow \mathbb{T}^1 \times \left[0,1\right]$ such that
	\begin{itemize}
		\item $g_n$ acts as the translation by $\left[nq^{\sigma_n}_n \right] \cdot \frac{i}{a_n}$ in the $x$-direction on $\mathbb{T}^1 \times \left[ \frac{i+ \varepsilon_n}{a_n}, \frac{i+1- \varepsilon_n}{a_n} \right] \subset \mathbb{T}^1 \times \left[2\e_n, 1- 2\e_n\right]$,
		\item $g_n$ coincides with the identity on $\mathbb{T}^1 \times \left[0, \e_n \right] \cup \mathbb{T}^1 \times \left[1- \e_n,1 \right]$,
		\item $g_n \circ R_{\frac{1}{q_n}} = R_{\frac{1}{q_n}} \circ g_n$,
		\item $\vertiii{g_n}_k\leq C_{n,k} \cdot \left[nq^{\sigma_n}_n\right]^k\cdot q^{3\cdot (k-1)}_n$, where the constant $C_{n,k}$ depends on $k$, $n$, and $\e_n$ but is independent of $q_n$.
	\end{itemize}
\end{lemma}

\begin{proof}
Let $a,b \in \mathbb{N}$, $\varepsilon>0$ satisfying $\frac{1}{\varepsilon} \in \mathbb{N}$ and $\rho: \mathbb{R} \rightarrow \mathbb{R}$ be a smooth increasing function that equals $0$ for $x\leq -1$ and $1$ for $x\geq 1$. Moreover, we let $j_0\in \N$ be the minimum $j\in \N$, $j\geq a\varepsilon$, such that $b\cdot \frac{j}{a}\equiv 0 \mod 1$. Then we define the map $\tilde{\psi}_{a,b,\varepsilon}: \left[0,1\right]\rightarrow\mathbb{R}$ by
\begin{equation*}
\tilde{\psi}_{a,b,\varepsilon}\left(x\right)=\frac{b}{a} \cdot \sum^{a-j_0}_{i=j_0+1} \rho \left(\frac{a \cdot x}{\varepsilon}-\frac{i}{\varepsilon} \right).
\end{equation*}
Note that $\tilde{\psi}_{a,b,\varepsilon}|_{\left[0,\e \right] \cup \left[1-\e, 1\right]} \equiv 0 \mod 1$ and for every $j_0 \leq i \leq a-j_0$ we have $$\tilde{\psi}_{a,b,\varepsilon}|_{\left[\frac{i+\e}{a},\frac{i+1-\varepsilon}{a}\right]} \equiv b \cdot \frac{i}{a} \mod 1.$$ Furthermore, we can estimate
\begin{equation} \label{eq:g1}
\left\|D^k\tilde{\psi}_{a,b,\varepsilon}\right\|_0 \leq \frac{b\cdot a^{k-1}}{\varepsilon^k} \cdot \left\|D^k\rho\right\|_0.
\end{equation}
Then we define the measure-preserving diffeomorphism $g_{a,b,\varepsilon}: \mathbb{S}^1 \times \left[0,1\right]\rightarrow \mathbb{S}^1 \times \left[0,1\right]$ by
\begin{equation*}
g_{a,b,\varepsilon}\left(x,y\right)= \left(x+\tilde{\psi}_{a,b,\varepsilon}\left(y\right) ,y\right).
\end{equation*}
In our concrete constructions we will use
\begin{equation*}
g_n = g_{\left[nq^{\sigma_n}_n\right]\cdot 2q^3_n, \left[nq^{\sigma_n}_n\right],\e_n}.
\end{equation*}
We observe $g_n \circ R_{\frac{1}{q_n}}=R_{\frac{1}{q_n}} \circ g_n$ and
\begin{equation} \label{eq:normG}
\left\|D^k\psi_{n}\right\|_0 \leq \frac{\left[nq^{\sigma_n}_n\right]^k\cdot q^{3\cdot (k-1)}_n}{\varepsilon^k_n} \cdot \left\|D^k\rho\right\|_0 =: C_{\psi_n,k} \cdot  \left[nq^{\sigma_n}_n\right]^k\cdot q^{3\cdot (k-1)}_n
\end{equation}
from (\ref{eq:g1}) with $\psi_n= \tilde{\psi}_{\left[nq^{\sigma_n}_n\right]\cdot 2q^3_n, \left[nq^{\sigma_n}_n\right],\e_n}$. This immediately yields $\vertiii{g_n}_k\leq C_{n,k} \cdot \left[nq^{\sigma_n}_n\right]^k\cdot q^{3\cdot (k-1)}_n$, where the constant $C_{n,k}$ is independent of $q_n$.
\end{proof}

To exclude the regions related to smoothing coming from $\rho$ we introduce
\[
C_n \coloneqq \bigcup_{i\in \Z} \left[\frac{i-\varepsilon_n}{\left[nq^{\sigma_n}_n\right]\cdot 2q^3_n} , \frac{i+\varepsilon_n}{\left[nq^{\sigma_n}_n\right]\cdot 2q^3_n} \right]\bigcap\T.
\]
Using the different regions $\tilde{K}^{(i)}_n$ for $\phi_n$ we set
\begin{align*}
\tilde{L}^{(i)}_n \coloneqq \tilde{K}^{(i)}_n \cap \phi^{-1}_n\left(\T \times (\T \setminus C_n ) \right),
\end{align*}
which gives (with the notation $a_n=\left[nq^{\sigma_n}_n\right]\cdot 2q^3_n$)
\begin{align*}
\tilde{L}^{(i)}_n = &\bigcup^{q^{j-1}_n -1}_{s_1=0} \bigcup^{(1-2\e_n)a_n-1}_{s_2=2\e_n a_n} \left[\frac{i}{2q^3_n} + \frac{s_1}{2q^{j+2}_n}+ \frac{s_2+\e_n}{a_n \cdot 2q^{j+2}_n}, \frac{i}{2q_n^3} + \frac{s_1}{2q^{j+2}_n} + \frac{s_2+1-\e_n}{a_n \cdot 2q^{j+2}_n} \right] \\ &\times \left[2\e_n,1-2\e_n\right]
\end{align*}
in case of $\phi_n =\phi_{2q^{j+2}_n,\e_n} $, while if $\phi_n=\text{id}$, then we have
\begin{align*}
\tilde{L}^{(i)}_n = \bigcup^{(1-4\e_n)\cdot a_n-1}_{s=2\lceil \e_n \cdot a_n\rceil } \left[\frac{i}{2q^3_n}, \frac{i+1}{2q^3_n}\right] \times \left[\frac{s+\varepsilon_n}{a_n} , \frac{s+1-\varepsilon_n}{a_n} \right].
\end{align*}
Hereby, we define the \emph{good domain} of the conjugation map $h_n =g_n \circ \phi_{n}$ as
\begin{equation} \label{eq:Lsmooth3}
L_n = \bigcup^{2q^3_n-1}_{i=0} \tilde{L}^{(i)}_n.
\end{equation}

\subsubsection{Norm estimates and parameter growth}
Proceeding in a way similar to Section \ref{sec:uniquelyErgodicAbC} we obtain the following set of estimates.

\begin{lemma}\label{lemma bound on derivative weak-mixing case}
The conjugating diffeomorphisms satisfy the following norm estimates:
\begin{align*}
\max\{\vertiii{h_n}_k, \vertiii{h_n^{-1}}_k\}\leq C q_n^{3k^2-3k+\sigma_nk^2+(n^2+1)k^2},
\end{align*}
where the constant $C$ is dependent on $k$ and $\e_n$  but not on $q_n$.
\end{lemma}


\begin{lemma}\label{lemma estimate q_(n+1) vs q_n weak mixing}
The expression $\lceil\vertiii{H_{n}}_1\rceil\max\{\vertiii{H_{n}}^{m_n+1}_{m_n+1}, \vertiii{H_{n}^{-1}}^{m_n+1}_{m_n+1}\}$ is bounded above by $q_{n}^{n^2(m_n+1)^5-2}$ and hence any choice of $l_n\geq q_{n}^{n^2(m_n+1)^5-2}$ satisfies the requirement imposed by \eqref{eq:equation ln condition}. Additionally if $k_n=1$, we get $$q_{n+1}\geq q_n^{n^2(m_n+1)^5}.$$ In particular we can choose
\begin{itemize}
    \item $m_n=n-1$ and $l_n=q_n^{n^8-2}$ yielding $q_{n+1}= q_n^{n^8}=q_1^{\prod_{m=1}^n(m^8)},$ or,
    \item $m_n=n-1$ and $l_n=q_n^{q_n-2}$ yielding $q_{n+1}= q_n^{q_n},$ or,
    \item $m_n=K-1$ for some integer $K>1$ and $l_n=q_n^{n^2K^5-2}$ yielding $q_{n+1}= q_n^{n^2K^5}.$
\end{itemize}
Moreover, if $m_n\geq 1$, then we have the freedom for the choice of a sequence $l_n'$ to satisfy \eqref{eq:equation ln' convergence} and \eqref{eq:equation ln' are summable}.
\end{lemma}

\begin{proof}
Proceeding as in proof of Lemma \ref{lem:estimateQ(n+1)vsQnUniquelyErgodic}, and using the estimates from Lemma \ref{lemma bound on derivative weak-mixing case} we get using the notation $\gamma_n=m_n+1$,
\begin{align*}
& \vertiii{H_{n}}_1 \max\{\vertiii{H_{n}}^{\gamma_n}_{\gamma_n}, \vertiii{H_{n}^{-1}}^{\gamma_n}_{\gamma_n}\} \\
 \leq& \Big[\vertiii{H_{n-1}^{-1}}_1 \max\{\vertiii{H_{n-1}}^{\gamma_n}_{\gamma_n}, \vertiii{H_{n-1}^{-1}}^{\gamma_n^2}_{\gamma_n}\}\Big] \cdot \vertiii{h_{n}}_1\cdot \max\{\vertiii{h_{n}}^{\gamma_n^2}_{\gamma_n}, \vertiii{h_{n}^{-1}}^{\gamma_n}_{\gamma_n}\} \\
 < &q^{n^2\gamma_n^5-2}_n.
\end{align*}
So, in conclusion  $l_n=q_n^{n^2\gamma_n^5-2}$ will satisfy all criteria required.
\end{proof}

\subsubsection{Weak mixing and safe domains}
\begin{theorem}
The AbC construction defined above converges to a weakly mixing AbC diffeomorphism
\end{theorem}

\begin{proof}
We sketch the proof which follows along the strategy in \cite{FS}. To apply the criterion for weak mixing in \cite[Proposition 3.9]{FS} we have to show that the diffeomorphism $\Phi_n \coloneqq \phi_n \circ R^{m_n}_{\alpha_{n+1}} \circ \phi^{-1}_n$ is uniformly distributing (see \cite[Definition 3.6]{FS} for the definition of this notion), where we take $m_n$ such that $m_n \cdot (\a_{n+1}-\a_n) =\frac{1}{2q_n}$ as in \cite[section 5.4.1]{FS}. For this purpose, one checks by direct computation that the maps $\phi_{2q^{3}_n,\varepsilon_n} \circ R^{m_n}_{\alpha_{n+1}} \circ \text{id}$, $\text{id} \circ R^{m_n}_{\alpha_{n+1}} \circ \phi^{-1}_{2q^{n^2+1}_n,\varepsilon_n}$, and  $\phi_{2q^{j+1}_n,\varepsilon_n} \circ R^{m_n}_{\alpha_{n+1}} \circ \phi^{-1}_{2q^{j}_n,\varepsilon_n}$ are uniformly distributing (see e.g. \cite[Lemma 4.3]{Kdisjointness}). Moreover, we note that the second word $\tilde{w}$ describes the combinatorics of the conjugation map $\phi_n$ on the second half $\left[\frac{1}{2q_n}, \frac{1}{q_n}\right] \times \T$ of the fundamental domain and that it was defined in such a way that $\phi_n \circ R^{m_n}_{\a_{n+1}} \circ \phi^{-1}_n$ gives one of these three cases. Hence, $\Phi_n$ is uniformly distributing and $T$ is weakly mixing. Convergence follows from Lemma \ref{lem:convergenceOfAbC}.
\end{proof}

\begin{remark} \label{rem:safeSmooth}
	To compare the orbits of $T^t_n$ and $T^t_{n-1}$ for small numbers of iterates $t\leq l_n' q_n$ in Lemmas \ref{lem:upperSmooth4} and \ref{lem:lowerSmooth4} we introduce the sets $\overline{L}^{(n)}_{i_1,i_2,\dots, i_{n^2+1},j}$ as
	{\footnotesize\begin{align*}
	&\left[ \frac{i_1}{q_n}+ \frac{i_2}{2q^3_n} + \dots + \frac{i_{n^2}}{2q^{n^2+1}_n} + \frac{i_{n^2+1} + \e_n}{a_n \cdot 2q^{n^2+1}_n}, \frac{i_1}{q_n}+ \frac{i_2}{2q^3_n} + \dots + \frac{i_{n^2}}{2q^{n^2+1}_n} + \frac{i_{n^2+1} + 1- \e_n}{a_n \cdot 2q^{n^2+1}_n}-\frac{2l_n'}{k_nl_nq_n}\right] \\
	& \times \left[\frac{j+\varepsilon_n}{a_n} , \frac{j+1-\varepsilon_n}{a_n} \right]
	\end{align*}}
	(with the notation $a_n=\left[nq^{\sigma_n}_n\right]\cdot 2q^3_n$) and use them to define the \emph{safe domain} of $T_n$ as
	\begin{align*}
	\overline{L}_n = \bigcup^{q_n -1}_{i_1=0} \bigcup^{2q^2_n-1}_{i_2=0} \bigcup^{(1-2\e_n)q_n}_{i_3=2\e_nq_n} \dots \bigcup^{(1-2\e_n)q_n}_{i_{n^2}=2\e_nq_n} \bigcup^{(1-2\e_n)\cdot a_n-1}_{i_{n^2+1}=2\e_n \cdot a_n }\bigcup^{(1-2\e_n)\cdot a_n-1}_{j=2\e_n \cdot a_n }  h_n\left(\overline{L}^{(n)}_{i_1,i_2,\dots, i_{n^2+1},j}\right).
	\end{align*}
	To even compare iterates $T^t$ and $T^t_{n-1}$ for $t\leq l_n' q_n$ we define for $m\geq n$ the sets $\tilde{L}^{(m)}_{i_1,i_2,\dots, i_{m^2+1},j}$ as
	{\footnotesize\begin{align*}
	& \left[ \frac{i_1}{q_m}+ \frac{i_2}{2q^3_m} + \dots + \frac{i_{m^2}}{2q^{m^2+1}_m} + \frac{i_{m^2+1} + \e_m}{a_m \cdot 2q^{m^2+1}_m}, \frac{i_1}{q_m}+ \frac{i_2}{2q^3_m} + \dots + \frac{i_{m^2}}{2q^{m^2+1}_m} + \frac{i_{m^2+1} + 1- \e_m}{a_m \cdot 2q^{m^2+1}_m}-\frac{2 l_n' q_n}{k_ml_mq^2_m}\right] \\
	& \times \left[\frac{j+\varepsilon_m}{a_m} , \frac{j+1-\varepsilon_m}{a_m} \right]
	\end{align*}}
	(note that it coincides with the previous definition if $n=m$) and its union
	{\footnotesize\begin{align*}
	\Xi_{n,m} = \bigcup^{q_m -1}_{i_1=0} \bigcup^{2q^2_m-1}_{i_2=0} \bigcup^{(1-2\e_m)q_m}_{i_3=2\e_mq_m} \dots \bigcup^{(1-2\e_m)q_m}_{i_{m^2}=2\e_mq_m} \bigcup^{(1-2\e_m)\cdot a_m-1}_{i_{m^2+1}=2\e_m \cdot a_m }\bigcup^{(1-2\e_m)\cdot a_m-1}_{j=2\e_m \cdot a_m }  h_n \circ \dots h_m \left(\tilde{L}^{(m)}_{i_1,i_2,\dots, i_{m^2+1},j}\right).
	\end{align*}}
Hereby, we set
	\[
	\Xi_{n} \coloneqq \overline{L}_n \cap \bigcap_{m> n} \Xi_{n,m}.
	\]
		
	We note that
	\[
	\mu(\Xi_{n,m}) \geq  1-4(m^2+1)\e_m -\frac{a_m \cdot 2q^{m^2+1}_m \cdot 2l_n' q_n}{k_ml_mq^2_m}.
	\]
	Under our condition $l_m \geq q^{2m^2-2}$ from Lemma \ref{lemma estimate q_(n+1) vs q_n weak mixing} and our assumption $\e_m \leq \frac{1}{m^4}$ from \eqref{eq:en4}, this yields for any given $\delta>0$ that $\mu(\Xi_n)>1-\delta$ for $n$ sufficiently large.
\end{remark}

\subsection{Lower bounds for cardinality of maximal separated sets}

\begin{lemma}
	Given any $\e>0$ there is $N\in\N$ such that we have for all $n\geq N$ that
	\begin{align*}
	S_{d^{T_n}_{q_{n+1}}}(\e)\geq  q^{n^2}_n.
	\end{align*}
\end{lemma}
\begin{proof}
	As in the proof of Lemma \ref{lem:LowerSmooth2} we start by defining for any given $\varepsilon>0$ the numbers $b_0=\varepsilon=c_0$ and $c_{i}=2q^{i^2+1}_{i}b_{i-1}$ as well as $b_{i}=\frac{c_{i-1}}{2q^{i^2+1}_{i}}$ by recursion. Let $N=N(\varepsilon)\coloneqq \max \{i:c_{i}<1\}$. We consider the ``good domain''
	\[
	\Theta_N \coloneqq \bigcap^{N}_{m=1}L_m
	\]
	of $H_N$ as well as its subset
	\[
	\Theta_{N,1} \coloneqq \Theta_N \cap K^{(N^2-1)}_N \cap \bigcap^{N-1}_{m=1} h^{-1}_N \circ \dots \circ h^{-1}_{m+1} \left(K^{(m^2-1)}_m\right),
	\]
	i.e. $h_{m+1}\circ \dots \circ h_N (\Theta_{N,1})$ lies in the domain $K^{(m^2-1)}_m$ and $\phi_m$ acts as $\phi_{2q^{m^2+1}_m,\e_m}$ on it.
	
	Then we take an interval $\hat{I}_u=\left[ \frac{u}{q_{N+1}}, \frac{u+1}{q_{N+1}} \right]$ lying in $\pi_1(\Theta_{N,1})$. In the next step $n=N+1$ we introduce subsets $S^{(u)}_{i_1,i_2,i_3,j_1,j_2}$ of each set $S^{(u)}\coloneqq \hat{I}_u \times [2\varepsilon_N, 1-2\varepsilon_N]$ as follows: Let $\Sigma_{n^2-1} \subset \{0,\dots,2q^2_n-1\}$ be the set of indices $i$ with $W_i=n^2-1$. By uniformity in Lemma \ref{lem:prob3} we have $\abs{\Sigma_{n^2-1}}= \frac{2q^2_n}{n^2}$. Then we define subsets $S^{(u)}_{i_1,i_2,i_3,j_1,j_2}$ as
	{\small\begin{align*}
	 & \Big[ \frac{u}{q_{n}} + \frac{i_1}{2q^3_n}+ \frac{i_2}{2q^{n^2+1}_n}+ \frac{d_n}{2q^{n^2+1}_n \cdot \left[nq^{\sigma_n}_n\right]}+\frac{i_3\cdot (e_n+1)}{2q^{n^2+2}_n \cdot \left[nq^{\sigma_n}_n\right]}+\frac{\e_n}{2q^{n^2+1}_n \cdot 2q^3_n \cdot \left[nq^{\sigma_n}_n\right]}, \\
	 & \ \frac{u}{q_{n}} + \frac{i_1}{2q^3_n}+ \frac{i_2}{2q^{n^2+1}_n}+ \frac{d_n}{2q^{n^2+1}_n \cdot\left[nq^{\sigma_n}_n\right]}+\frac{i_3\cdot (e_n+1)}{2q^{n^2+2}_n \cdot \left[nq^{\sigma_n}_n\right]}+\frac{1-\e_n}{2q^{n^2+1}_n \cdot 2q^3_n \cdot \left[nq^{\sigma_n}_n\right]}\Big) \\
	\times & \Big[\frac{j_1 \cdot (d_n+1)}{\left[nq^{\sigma_n}_n\right]}+\frac{j_2 \cdot (e_n+1)}{q_n \cdot \left[nq^{\sigma_n}_n\right]}+\frac{\e_n}{2q^3_n \cdot \left[nq^{\sigma_n}_n\right]}, \frac{j_1 \cdot (d_n+1)}{\left[nq^{\sigma_n}_n\right]}+\frac{j_2 \cdot (e_n+1)}{q_n \cdot \left[nq^{\sigma_n}_n\right]}+\frac{1-\e_n}{2q^3_n \cdot \left[nq^{\sigma_n}_n\right]} \Big),
	\end{align*}}
	where $d_n \coloneqq \lceil 2\e_n nq^{\sigma_n}_n \rceil$, $e_n \coloneqq \lceil b_n q_n \rceil$, $i_1\in \Sigma_{n^2-1}$, $0\leq i_2<2q^{n^2-2}_n$, $0\leq i_3 < \frac{q_n}{e_n+1}$, $0 \leq j_1 < \frac{\left[nq^{\sigma_n}_n\right]}{d_n+1}$, and $0\leq j_2 < \frac{q_n}{e_n +1}$.
	
	We claim that points from different $S^{(u)}_{i_1,i_2,i_3,j_1,j_2}$ with $i_3 -j_2 \neq 0 \mod \frac{q_n}{e_n+1}$ are $\e$-separated under $\{H_{n} \circ R^t_{\a_{n+1}} \circ h^{-1}_n\}_{0\leq t \leq q_{n+1}}$. For this purpose, we start by calculating the image of $S^{(u)}_{i_1,i_2,i_3,j_1,j_2}$ under $g^{-1}_n$ to
	{\footnotesize\begin{align*}
	 & \Big[ \frac{u-j_2 \cdot (e_n+1)}{q_{n}} + \frac{i_1}{2q^3_n}+ \frac{i_2}{2q^{n^2+1}_n}+ \frac{d_n}{2q^{n^2+1}_n \cdot \left[nq^{\sigma_n}_n\right]}+\frac{i_3\cdot (e_n+1)}{2q^{n^2+2}_n \cdot \left[nq^{\sigma_n}_n\right]}+\frac{\e_n}{2q^{n^2+1}_n \cdot 2q^3_n \cdot \left[nq^{\sigma_n}_n\right]}, \\
	& \ \frac{u-j_2 \cdot (e_n+1)}{q_{n}} + \frac{i_1}{2q^3_n}+ \frac{i_2}{2q^{n^2+1}_n}+ \frac{d_n}{2q^{n^2+1}_n \cdot\left[nq^{\sigma_n}_n\right]}+\frac{i_3\cdot (e_n+1)}{2q^{n^2+2}_n \cdot \left[nq^{\sigma_n}_n\right]}+\frac{1-\e_n}{2q^{n^2+1}_n \cdot 2q^3_n \cdot \left[nq^{\sigma_n}_n\right]}\Big) \\
	\times & \Big[\frac{j_1 \cdot (d_n+1)}{\left[nq^{\sigma_n}_n\right]}+\frac{j_2 \cdot (e_n+1)}{q_n \cdot \left[nq^{\sigma_n}_n\right]}+\frac{\e_n}{2q^3_n \cdot \left[nq^{\sigma_n}_n\right]}, \frac{j_1 \cdot (d_n+1)}{\left[nq^{\sigma_n}_n\right]}+\frac{j_2 \cdot (e_n+1)}{q_n \cdot \left[nq^{\sigma_n}_n\right]}+\frac{1-\e_n}{2q^3_n \cdot \left[nq^{\sigma_n}_n\right]} \Big).
	\end{align*}}
	This image is positioned in the domain where $\phi^{-1}_n=\phi^{-1}_{2q^{n^2+1}_n,\e_n}$ since $i_1\in \Sigma_{n^2-1}$. Hence, we get that $\phi^{-1}_n \circ g^{-1}_n (S^{(u)}_{i_1,i_2,i_3,j_1,j_2})$ is equal to
	{\footnotesize\begin{align*}
	& \Big[ \frac{u-j_2 \cdot (e_n+1)}{q_{n}} + \frac{i_1}{2q^3_n} + \frac{i_2+1}{2q^{n^2+1}_n}- \frac{j_1 \cdot (d_n+1)}{2q^{n^2+1}_n \cdot \left[nq^{\sigma_n}_n\right]}-\frac{j_2 \cdot (e_n+1)}{2q^{n^2+2}_n \cdot \left[nq^{\sigma_n}_n\right]}-\frac{1-\e_n}{2q^{n^2+1}_n \cdot 2q^3_n \cdot \left[nq^{\sigma_n}_n\right]}, \\
	& \ \frac{u-j_2 \cdot (e_n+1)}{q_{n}} + \frac{i_1}{2q^3_n} + \frac{i_2+1}{2q^{n^2+1}_n}- \frac{j_1 \cdot (d_n+1)}{2q^{n^2+1}_n \cdot \left[nq^{\sigma_n}_n\right]}-\frac{j_2 \cdot (e_n+1)}{2q^{n^2+2}_n \cdot \left[nq^{\sigma_n}_n\right]}-\frac{\e_n}{2q^{n^2+1}_n \cdot 2q^3_n \cdot \left[nq^{\sigma_n}_n\right]}\Big) \\
	\times & \Big[\frac{d_n}{\left[nq^{\sigma_n}_n\right]}+\frac{i_3\cdot (e_n+1)}{q_n \cdot \left[nq^{\sigma_n}_n\right]}+\frac{\e_n}{2q^3_n \cdot \left[nq^{\sigma_n}_n\right]}, \frac{d_n}{\left[nq^{\sigma_n}_n\right]}+\frac{i_3\cdot (e_n+1)}{q_n \cdot \left[nq^{\sigma_n}_n\right]}+\frac{1-\e_n}{2q^3_n \cdot \left[nq^{\sigma_n}_n\right]} \Big).
	\end{align*}}
		
	Suppose that $R^t_{\a_{n+1}}$ causes a translation as follows.
	{\scriptsize\begin{align*}
	& \Big[ \frac{u-j_2 \cdot (e_n+1)+t}{q_{n}} + \frac{i_1+t_1}{2q^3_n} + \frac{i_2+1+t_2}{2q^{n^2+1}_n}- \frac{j_1 \cdot (d_n+1)-t_3}{2q^{n^2+1}_n \cdot \left[nq^{\sigma_n}_n\right]}-\frac{j_2 \cdot (e_n+1)}{2q^{n^2+2}_n \cdot \left[nq^{\sigma_n}_n\right]}-\frac{1-\e_n}{2q^{n^2+1}_n \cdot 2q^3_n \cdot \left[nq^{\sigma_n}_n\right]}, \\
	& \ \frac{u-j_2 \cdot (e_n+1)+t}{q_{n}} + \frac{i_1+t_1}{2q^3_n} + \frac{i_2+1+t_2}{2q^{n^2+1}_n}- \frac{j_1 \cdot (d_n+1)-t_3}{2q^{n^2+1}_n \cdot \left[nq^{\sigma_n}_n\right]}-\frac{j_2 \cdot (e_n+1)}{2q^{n^2+2}_n \cdot \left[nq^{\sigma_n}_n\right]}-\frac{\e_n}{2q^{n^2+1}_n \cdot 2q^3_n \cdot \left[nq^{\sigma_n}_n\right]}\Big) \\
	\times & \Big[\frac{d_n}{\left[nq^{\sigma_n}_n\right]}+\frac{i_3\cdot (e_n+1)}{q_n \cdot \left[nq^{\sigma_n}_n\right]}+\frac{\e_n}{2q^3_n \cdot \left[nq^{\sigma_n}_n\right]}, \frac{d_n}{\left[nq^{\sigma_n}_n\right]}+\frac{i_3\cdot (e_n+1)}{q_n \cdot \left[nq^{\sigma_n}_n\right]}+\frac{1-\e_n}{2q^3_n \cdot \left[nq^{\sigma_n}_n\right]} \Big),
	\end{align*}}
	
	Notice that in dependence of $i_1$ and $i_2$, the iterates $R^{t}_{\a_{n+1}} \circ h^{-1}_n(S^{(u)}_{i_1,i_2,i_3,j_1,j_2})$ lie in the distinct domains of $\phi_n$ at different times. In case that $R^{t}_{\a_{n+1}} \circ h^{-1}_n(S^{(u)}_{i_1,i_2,i_3,j_1,j_2})$ lies in the domain where $\phi_n=\phi_{2q^{n^2+1}_n,\e_n}$ (case 1), then we obtain that $h_n \circ R^t_{\a_{n+1}} \circ h^{-1}_n(S^{(u)}_{i_1,i_2,i_3,j_1,j_2})$ is equal to
	{\footnotesize\begin{align*}
	& \Big[ \frac{u+t}{q_{n}} + \frac{i_1+t_1}{2q^3_n}+ \frac{i_2+t_2}{2q^{n^2+1}_n}+ \frac{d_n}{2q^{n^2+1}_n \cdot \left[nq^{\sigma_n}_n\right]}+\frac{i_3\cdot (e_n+1)}{2q^{n^2+2}_n \cdot \left[nq^{\sigma_n}_n\right]}+\frac{\e_n}{2q^{n^2+1}_n \cdot 2q^3_n \cdot \left[nq^{\sigma_n}_n\right]}, \\
	& \ \frac{u+t}{q_{n}} + \frac{i_1+t_1}{2q^3_n}+ \frac{i_2+t_2}{2q^{n^2+1}_n}+ \frac{d_n}{2q^{n^2+1}_n \cdot\left[nq^{\sigma_n}_n\right]}+\frac{i_3\cdot (e_n+1)}{2q^{n^2+2}_n \cdot \left[nq^{\sigma_n}_n\right]}+\frac{1-\e_n}{2q^{n^2+1}_n \cdot 2q^3_n \cdot \left[nq^{\sigma_n}_n\right]}\Big) \\
	\times & \Big[\frac{j_1 \cdot (d_n+1)-t_3}{\left[nq^{\sigma_n}_n\right]}+\frac{j_2 \cdot (e_n+1)}{q_n \cdot \left[nq^{\sigma_n}_n\right]}+\frac{\e_n}{2q^3_n \cdot \left[nq^{\sigma_n}_n\right]}, \frac{j_1 \cdot (d_n+1)-t_3}{\left[nq^{\sigma_n}_n\right]}+\frac{j_2 \cdot (e_n+1)}{q_n \cdot \left[nq^{\sigma_n}_n\right]}+\frac{1-\e_n}{2q^3_n \cdot \left[nq^{\sigma_n}_n\right]} \Big),
	\end{align*}}
	
	On the other hand, if $R^{t}_{\a_{n+1}} \circ h^{-1}_n(S^{(u)}_{i_1,i_2,i_3,j_1,j_2})$ lies in the domain with $\phi_n=\text{id}$ (case 2), then we calculate that $h_n \circ R^t_{\a_{n+1}} \circ h^{-1}_n(S^{(u)}_{i_1,i_2,i_3,j_1,j_2})$ is equal to
	{\tiny\begin{align*}
	& \Big[ \frac{u+(i_3-j_2) \cdot (e_n+1)+t}{q_{n}} + \frac{i_1+t_1}{2q^3_n} + \frac{i_2+1+t_2}{2q^{n^2+1}_n}- \frac{j_1 \cdot (d_n+1)-t_3}{2q^{n^2+1}_n \cdot \left[nq^{\sigma_n}_n\right]}-\frac{j_2 \cdot (e_n+1)}{2q^{n^2+2}_n \cdot \left[nq^{\sigma_n}_n\right]}-\frac{1-\e_n}{2q^{n^2+1}_n \cdot 2q^3_n \cdot \left[nq^{\sigma_n}_n\right]}, \\
	& \ \frac{u+(i_3-j_2) \cdot (e_n+1)+t}{q_{n}} + \frac{i_1+t_1}{2q^3_n} + \frac{i_2+1+t_2}{2q^{n^2+1}_n}- \frac{j_1 \cdot (d_n+1)-t_3}{2q^{n^2+1}_n \cdot \left[nq^{\sigma_n}_n\right]}-\frac{j_2 \cdot (e_n+1)}{2q^{n^2+2}_n \cdot \left[nq^{\sigma_n}_n\right]}-\frac{\e_n}{2q^{n^2+1}_n \cdot 2q^3_n \cdot \left[nq^{\sigma_n}_n\right]}\Big) \\
	\times & \Big[\frac{d_n}{\left[nq^{\sigma_n}_n\right]}+\frac{i_3\cdot (e_n+1)}{q_n \cdot \left[nq^{\sigma_n}_n\right]}+\frac{\e_n}{2q^3_n \cdot \left[nq^{\sigma_n}_n\right]}, \frac{d_n}{\left[nq^{\sigma_n}_n\right]}+\frac{i_3\cdot (e_n+1)}{q_n \cdot \left[nq^{\sigma_n}_n\right]}+\frac{1-\e_n}{2q^3_n \cdot \left[nq^{\sigma_n}_n\right]} \Big),
	\end{align*}}
		
	By definition of $b_n$ and $e_n$ we get separation between blocks in case 1 and those in case 2 from the horizontal distance if $i_3 -j_2 \neq 0 \mod \frac{q_n}{e_n+1}$. Since there are adjacent domains with the mapping behaviors $\phi_n=\phi_{2q^{n^2+1}_n,\e_n}$ and $\phi_n=\text{id}$ by property (2) of Lemma \ref{lem:prob3}, we get separation for those $S^{(u)}_{i_1,i_2,i_3,j_1,j_2}$.
	
	Counting the number of different sets $S^{(u)}_{i_1,i_2,i_3,j_1,j_2}$ we obtain
	\[
	S_{d^{T_{n}}_{q_{n+1}}}(\e)\geq q^{n^2}_n.
	\]
	By the same methods we continue for any $n> N+1$.
\end{proof}

Using methods from Section \ref{sec:untwistedAbC} we obtain

\begin{lemma} \label{lem:LowerSmooth3}
Given any $\e>0$ there is $N\in\N$ such that we have for all $n\geq N$,
$$	S_{d^T_{q_{n+1}}}(\e)\geq  q^{n^2}_n.$$
\end{lemma}

\subsection{Upper bounds for cardinality of minimal covering sets}

\begin{lemma} 
	For any given $\e>0$ and any $n\in \N$ sufficiently large we have
	\[
	N_{d^{T_n}_m} (\e)\leq C_n \cdot\frac{1}{\e^3} \cdot q^{n^2+2\sigma_n}_n,
	\]
	for any integer $m\geq 0$, where the constant $C_n$ depends on $\e_n$ and $H_{n-1}$ but is independent of $q_n$, $m$, and $\e$.
\end{lemma}
\begin{proof}
The proof is identical to the proof of Lemma \ref{lem:FSWSpanningAtn-thStage} with equation \eqref{eq:rnxrny} replaced by,
\begin{equation}
\begin{aligned}
	r^{(x)}_n & \coloneqq \frac{\varepsilon}{16q^{n^2+1}_n \cdot [nq^{\sigma_n}_n] \cdot  \norm{D\varphi_{\e_n}}_0 \cdot C_{\psi_n,1} \cdot \norm{DH_{n-1}}_0},  \\
	r^{(y)}_n & \coloneqq \frac{\varepsilon}{8\cdot [nq^{\sigma_n}_n]\cdot \norm{D\varphi_{\e_n}}_0 \cdot C_{\psi_n,1} \cdot \norm{DH_{n-1}}_0},
\end{aligned}
\end{equation}
where $C_{\psi_n,1}$ is the constant from (\ref{eq:normG}).
\end{proof}
Using methods from Section \ref{sec:untwistedAbC}, we obtain
\begin{lemma}\label{lem:upperSmooth3}
For any given $\e>0$ and any $n\in\N$, we have
$$N_{d^T_m} (\e)\leq C_n \cdot\frac{1}{\e^3} \cdot q^{n^2+2\sigma_n}_n,$$
where $0\leq m \leq l_{n+1}'q_{n+1}$ and $C_n>0$ is some constant dependent on $\e_n$ and $H_{n-1}$ but independent of $m$ and $\e$.
\end{lemma}

\subsection{Upper topological slow entropy}
Similar to the proof of Theorem \ref{thm:untwistedTopInter}, we obtain the following results by replacing Lemmas \ref{lem:FSWSeperation} and \ref{lem:FSWSpanning} by Lemmas \ref{lem:LowerSmooth3} and \ref{lem:upperSmooth3}, respectively.
\begin{theorem}\label{thm:weakMixingTopInter}
There exists a weakly mixing $C^\infty$ Anosov-Katok diffeomorphism $T$ constructed using parameters specified in \eqref{eq:abcparameters} with $m_n=n-1$, $l_n'$ as in \eqref{eq:ln'} with $r=8$, $l_n$ satisfying \eqref{eq:equation ln condition}, and conjugacies specified by \eqref{eq:hn-wm}, such that the upper topological slow entropy is
\begin{equation}
\uent^{\top}_{a_m^{\operatorname{int2,8}}(t)}(T)=1.
\end{equation}
\end{theorem}
\begin{proof}
We consider $a_m^{\text{int2,8}}(t)$. If for a given $k$, $q_n < m_k\leq l_n'q_{n}=q_{n}^{(n-1)^8}$, then we have the estimate
\begin{align*}
\frac{N_{d_{m_k}^T}(\e)}{a_{m_k}^{\text{int2,8}}(t)}\leq \frac{N_{d_{m_k}^T}(\e)}{a_{q_n}^{\text{int2,8}}(t)}\leq \frac{C_{n-1}q_{n-1}^{(n-1)^2+2\sigma_{n-1}}}{\e^3q_{n-1}^{(n-1)^2t}}.
\end{align*}
If $q_{n}^{(n-1)^8}=l_n'q_n<m_k\leq q_{n+1}$, then we have the estimate
\begin{align*}
\frac{N_{d_{m_k}^T}(\e)}{a_{m_k}^{\text{int2,8}}(t)}\leq \frac{N_{d_{m_k}^T}(\e)}{a_{l_n'q_n}^{\text{int2,8}}(t)}\leq \frac{C_{n}q_{n}^{n^2+2\sigma_{n}}}{\e^3\big[q_n\big]^{\frac{(n-1)^8t}{n^6}}}.
\end{align*}
Hence, it is clear that $\overline{\text{ent}}^{top}_{a_m^{\text{int2,8}}(t)}(T)\leq 1$. On the other hand, for any $t\leq 1$
\begin{align*}
\limsup_{m\to\infty}\frac{S_{d_m^T}(\e)}{a_m^{\text{int2,8}}(t)}\geq \limsup_{n\to\infty}\frac{S_{d_{q_{n+1}}^T}(\e)}{a_{q_{n+1}}^{\text{int2,8}}(t)} \geq  \limsup_{n\to\infty}\frac{q_{n}^{n^2}}{q_n^{n^2t}}>0.
\end{align*}
Hence, $\overline{\text{ent}}^{top}_{a_m^{\text{int2}}(t)}(T)= 1$ in this case.
\end{proof}
We also get results for the polynomial and log scale.

\begin{theorem}\label{thm:weakMixingTopIn}
There exists a weakly mixing $C^\infty$ Anosov-Katok diffeomorphism $T$ constructed using parameters specified in \eqref{eq:abcparameters} with $q_{n+1}=q_n^{q_n^{n^2}}$, $m_n=n-1$, $l_n'=q_n^{q_n^{n^2}-2n^2}$, $l_n$ satisfying \eqref{eq:equation ln condition}, and conjugacies specified by \eqref{eq:hn-wm}, such that the upper topological slow entropy with respect to the log scale is
\begin{align*}
        \uent^{\top}_{a^{\ln}_m(t)}(T)=1.
\end{align*}
\end{theorem}

\begin{theorem}\label{thm:weakMixingTopPol}
For any integer $K>1$, there exists a weakly mixing $C^K$ Anosov-Katok diffeomorphism $T$ constructed using parameters specified in \eqref{eq:abcparameters} with $q_{n+1}=q_n^{n^2K^5}$, $m_n=K-1$, $l_n'=q_n^{n^2(K^5-2)-2}$, $l_n=q_n^{n^2K^5-2}$ satisfying \eqref{eq:equation ln condition}, and conjugacies specified by \eqref{eq:hn-wm}, such that the upper topological slow entropy with respect to the polynomial scale is
\begin{align*}
        \frac{1}{K^5}\leq \uent^{\top}_{a_m^{\pol}(t)}(T)<\frac{1}{K^5-2}.
\end{align*}
\end{theorem}

\subsection{Upper measure-theoretic slow entropy} \label{subsec:UpperMeasure}

We now turn to the measure-theoretic slow entropy. On the one hand, we provide some $(1-\e)$-cover with $(\e, q_{n+1})$-Hamming balls with respect to a given partition $\xi$. The cardinality of this cover gives an upper bound for $S^H_{\xi}(T,q_{n+1},\e)$. On the other hand, we give examples of regions that are $(\e, q_{n+1})$-Hamming apart from each other. Then we use this to deduce a lower bound on the cardinality of a $(1-\e)$-cover. Note that we still need good norm estimates for the conjugation maps to get precise growth rates of $q_{n+1}$ expressed in $q_n$ which allows us to find an appropriate scaling function in the slow entropy invariant.

We define the partial partition
\[
\xi_n = \Meng{R^{(n)}_{i,j}}{0\leq i <q_n , \ 2\e_nq_n\leq j < q_n-2\e_nq_n},
\]
where (recall the definition of the ``good domain'' $K_n$ of $\phi^{-1}_n$ from (\ref{eq:Ksmooth3}))
\begin{align*}
R^{(n)}_{i,j} & \coloneqq \Big[\frac{i}{q_n}, \frac{i+1}{q_n}\Big) \times \Big[\frac{j}{q_n}, \frac{j+1}{q_n}\Big) \cap K_n .
\end{align*}
Then we notice that the image partitions $\eta_n \coloneqq H_{n-1} \circ g_n (\xi_n)$ converge to the decomposition into points since
\[
\lim_{n\to \infty} \text{diam} \left(H_{n-1} \circ g_n \left(R^{(n)}_{i,j}\right) \right) \leq \lim_{n\to \infty} \norm{DH_{n-1}}_0 \cdot 2 \cdot C_{\psi_n,1} \cdot \left[nq^{\sigma_n}_n\right] \cdot \frac{\sqrt{2}}{q_n} =0,
\]
where $C_{\psi_n,1}$ is the constant from (\ref{eq:normG}). Hence, we can calculate the upper measure-theoretical slow entropy along the sequence of partitions $\{\eta_n\}_{n\in \N}$ using the generator theorem from Proposition \ref{prop:generator}.

We start with an upper bound on the number of covering sets.

\begin{lemma}\label{lem:upperSmooth4}
	Let $0<\e<\frac{1}{2}$ and $m\in \N$ sufficiently large. For $n>m$ we have for all $0\leq L \leq l_{n+1}'q_{n+1}$ that
	{\[
	S^H_{\eta_m}(T,L,\e) \leq C_n \cdot q^2_n,
	\]}
where the constant $C_n$ is independent of $q_n$.
\end{lemma}

\begin{proof}
	By Remark \ref{rem:safeSmooth} we can choose $m$ sufficiently large such that $\mu(\Xi_m)> 1-\frac{\e}{2}$. In addition to that, by condition \eqref{eq:en4} we can choose $m$ sufficiently large such that
	\begin{equation} \label{eq:deltaSmooth2}
	\sum^{\infty}_{n=m} 10n^2\e_n < \frac{\e}{2}.
	\end{equation}
	Then we define the sets $B^{(m)}_{\vec{j}}$ as
	{\footnotesize\begin{align*}
	 & \Big[ \frac{j_1}{q_m}+\frac{j_2}{2q^{3}_m}+\frac{j_3}{2q^{4}_m}+ \dots +\frac{j_{m^2}}{2q^{m^2+1}_m}+\frac{j_{m^2+1}}{2q^{m^2+2}_m}+\frac{j_{m^2+2} \e}{4q^{m^2+2}_m} , \frac{j_1}{q_m}+\frac{j_2}{2q^{3}_m}+ \dots +\frac{j_{m^2+1}}{2q^{m^2+2}_m}+\frac{(j_{m^2+2}+1) \e}{4q^{m^2+2}_m}\Big) \\
	 &\times \Big[\frac{j_{m^2+3}}{q_m}, \frac{j_{m^2+3}+1}{q_m} \Big)
	\end{align*}}
	for $\vec{j}=(j_1,\dots , j_{m^2+3})$ with $0\leq j_1<q_m$, $0\leq j_2<2q^{2}_m$, $4\e_mq_m\leq j_i < q_m-4\e_mq_m$ for $i=3,\dots , m^2+1$, $0\leq j_{m^2+2} < \lfloor \frac{2}{\e} \rfloor$, and $2\e_mq_m\leq j_{m^2+3} < q_m-2\e_mq_m$. We note that for any point in such a set $B^{(m)}_{\vec{j}}$ at most $4m^2\e_mq_{m+1}$ iterates under $\{R^t_{\a_{m+1}}\}_{0\leq t \leq q_{m+1}}$ do not lie in $\bigcup_{R^{(m)}_{i,j} \in \xi_m}\phi^{-1}_m \left(R^{(m)}_{i,j}\right)$. Altogether we see that each $H_m(B^{(m)}_{\vec{j}})$ is contained in one $\left(\frac{\e}{2}+8m^2\e_m,q_{m+1}\right)$-Hamming ball for $T_m$ with respect to the partition $\eta_m$. Since for any point $P\in H_m(B^{(m)}_{\vec{j}}\cap \Xi_{m+1})$ the images $T^t(P)$ and $T^t_m(P)$ for any $t\leq l_{m+1}' q_{m+1}$ lie in the same element of $\eta_m$ by the definition of the safe domain in Remark \ref{rem:safeSmooth}, we also obtain that each $H_m(B^{(m)}_{\vec{j}}\cap \Xi_{m+1})$ is contained in one $(\frac{\e}{2}+8m^2\e_m,q_{m+1})$-Hamming ball for $T$ with respect to the partition $\eta_m$.
	
	In the next step, we let $n=m+1$ and treat the collection $\tilde{\eta}_m \coloneqq \{B^{(m)}_{\vec{j}}\}$ as our ``target partition''. We denote the width of such sets $B^{(m)}_{\vec{j}}$ by $s_m \coloneqq \frac{\e}{4q^{m^2+2}_m}$. Within $h^{-1}_{n}\left(\bigcup B^{(m)}_{\vec{j}}\right)$ we consider sets of the following form (recalling $a_{n} \coloneqq 2q^3_{n}\cdot \left[nq^{\sigma_{n}}_{n} \right]$ and $b_{n} \coloneqq \left[nq^{\sigma_{n}}_{n} \right]$ from the construction of the conjugation map $g_{n}$):
	\begin{align*}
	& B^{(n)}_{\vec{k}}\coloneqq L_{n}  \cap \bigcup^{c_n -1}_{i_1=0} \bigcup^{u_{n}-1}_{i_2=0} \bigcup^{t_{n}-1}_{i_3=0} \dots \bigcup^{t_{n}-1}_{i_{n^2+1}=0} B^{(n)}_{\vec{k},\vec{i}} ,
	\end{align*}
	where $B^{(n)}_{\vec{k},\vec{i}}$ is the set
	\begin{align*}
	& \Big[ \frac{k_1 +i_1}{q_{n}}+\frac{k_2}{2q^3_{n}}+\frac{k_3  d_{n}+i_3 e_n +k_4 f_n} {2q^4_n}+\dots +\frac{k_{2n^2+1}  d_{n}+i_{n^2+1} e_n +k_{2n^2+2} f_n} {2q^{n^2+2}_n}, \\
	& \ \frac{k_1 +i_1}{q_{n}}+\frac{k_2}{2q^3_{n}}+\frac{k_3  d_{n}+i_3  e_n +k_4 f_n} {2q^4_n}+\dots +\frac{k_{2n^2+1}  d_{n}+i_{n^2+1} e_n +(k_{2n^2+2}+1) \cdot f_n} {2q^{n^2+2}_n}\Big) & \\
	& \times \Big[\frac{k_{2n^2+3}}{q_m}+\frac{i_2+k_{2n^2+4} \cdot \e_n s_m}{b_{n}}, \frac{k_{2n^2+3}}{q_m}+\frac{i_2+(k_{2n^2+4}+1) \cdot \e_n s_m}{b_{n}} \Big),
	\end{align*}
	and $c_n\coloneqq \lfloor \e_{n}s_mq_{n}\rfloor$, $d_{n}\coloneqq \lfloor \frac{\e_{n}q_{n}}{q_m} \rfloor$, $e_n \coloneqq \lfloor \frac{q_n}{b_n} \rfloor$, $f_n \coloneqq \lfloor \frac{\e_n s_m q_n}{b_n} \rfloor$, $t_{n} \coloneqq b_nd_n $, and $u_n=\lfloor \frac{b_n}{q_m}\rfloor$. Descriptively speaking, such a set $B^{(n)}_{\vec{k},\vec{i}}$ is a union of sets of width $\frac{ \e_n s_m }{b_n \lambda_{\max}}$ (where $\lambda_{\max}=2q^{n^2+1}_n$ is the largest expansion factor appearing in the definition of $\phi_n$) and height $\frac{ \e_n s_m}{b_n}$, whose images have small diameter under $h_n$ compared to the size of elements in $\tilde{\eta}_m$. When building $B^{(n)}_{\vec{k}}$ we take unions to reflect equivariance of $\phi_n$ and periodicity of $g_n$. Overall, such a $B^{(n)}_{\vec{k}}$ is a  holey subset of a set of width $\e_n s_m$ (by the union over $i_1$) and height $1/q_m$ (by the union over $i_2$). Moreover, we note that there are less than
	\[
	\frac{1}{\e_n s_m} \cdot 2q^2_n \cdot \left(\frac{q_m}{\e_n} \cdot \frac{1}{\e_ns_m}\right)^{n^2-1} \cdot q_m \cdot \frac{1}{\e_n s_m}
	\]
	many sets $B^{(n)}_{\vec{k}}$. We note that for any point in such a set $B^{(n)}_{\vec{k}}$ at most $(4m^2\e_m+4n^2\e_n)q_{n+1}$ iterates under $\{R^t_{\a_{n+1}}\}_{0\leq t \leq q_{n+1}}$ do not lie in $h^{-1}_{n}\left(\bigcup B^{(m)}_{\vec{j}}\right)$. On the remaining iterates these sets are chosen in this way such that each $h_n\left(B^{(n)}_{\vec{k}}\right)$ is contained in one $2\e_n$-Hamming ball for $h_n \circ R_{\a_{n+1}}\circ h^{-1}_n$ with respect to the partition $\tilde{\eta}_m$. Since each partition element of $H_m(\tilde{\eta}_m)$ lies within one $\left(\frac{\e}{2}+8m^2\e_m,q_{m+1}\right)$-Hamming ball for $T_m$ with respect to $\eta_m$ as seen above, we obtain that each $H_n\left(B^{(n)}_{\vec{k}}\right)$ is contained in one $\left(\frac{\e}{2}+9m^2 \e_m + 9n^2 \e_n ,q_{n+1}\right)$-Hamming ball for $T_n$ with respect to $\eta_m$. By definition of the safe domain $\Xi_n$, for any point $P\in H_n(B^{(n)}_{\vec{k}}\cap \Xi_n)$ the images $T^t(P)$ and $T^t_n(P)$ for any $t\leq l_{n+1}' q_{n+1}$ lie in the same element of $\eta_m$. Altogether we conclude
	{\[
	S^H_{\eta_m}\left(T,q_{n+1},\frac{\e}{2}+9m^2 \e_m + 9n^2 \e_n\right) \leq C_n \cdot q^2_n.
	\]}
	
	By induction we continue for $n>m+1$. To conclude the statement we make use of $\frac{\e}{2}+\sum^{\infty}_{n=m}10n^2\e_n <\e$ by condition (\ref{eq:deltaSmooth2}).
\end{proof}

In the converse direction, we also find a lower bound of the same order $q^2_n$ on the number of separated points.

\begin{lemma}\label{lem:lowerSmooth4}
	Let $0<\e<\frac{1}{2}$ and $m\in \N$. For $n>m$ we have
	\[
	S^H_{\eta_m}(T,q_{n+1},\e) \geq q^2_n.
	\]
\end{lemma}

\begin{proof}
	As in the proof of Lemma \ref{lem:upperSmooth4} we choose $m$ sufficiently large such that
	\begin{equation} \label{eq:condSafe4}
	    \mu(\Xi_m)> 1-\e.
	\end{equation}
	This time we additionally require on $m$ that
	\begin{equation} \label{eq:deltaSmooth4}
	\sum^{\infty}_{k=m}\left(\frac{3}{k^2}+4k^2\e_k\right) <\e
	\end{equation}
	Within the images $\bigcup_{R^{(m)}_{i,j} \in \xi_m}\phi^{-1}_m(R^{(m)}_{i,j})$ we define for $0\leq j_1 <q_m$, $0\leq j_2 < 2q^2_m$, $2\e_mq_m \leq j_i < (1-2\e_m)q_m$, where $3\leq i \leq m^2+2$, and $0\leq a < \lfloor \frac{q_{m+1}}{2q^{m^2+2}_m} \rfloor $ the subsets $B^{(m)}_{j_1,\dots , j_{m^2+2},a}$ as
	\begin{align*}
	& \Big[ \frac{j_1}{q_m}+\frac{j_2}{2q^{3}_m}+\frac{j_3}{2q^{4}_m}+ \dots +\frac{j_{m^2+1}}{2q^{m^2+2}_m} + \frac{a}{q_{m+1}} , \frac{j_1}{q_m}+\frac{j_2}{2q^{3}_m}+ \dots +\frac{j_{m^2+1}}{2q^{m^2+2}_m}+ \frac{a+1}{q_{m+1}}\Big) \\
	\times & \Big[\frac{j_{m^2+2}}{q_m}, \frac{j_{m^2+2}+1}{q_m} \Big).
	\end{align*}
	
	\begin{claim}\label{cl:starting} Let $j_2 \neq j^{\prime}_2 \mod q^2_m$, then we have that any two points $$P_1 \in H_m\left(B^{(m)}_{j_1,j_2,j_3,\dots ,j_{m^2+1}, j_{m^2+2},a} \cap \Xi_{m+1}\right),$$ $$P_2 \in H_m\left(B^{(m)}_{j^{\prime}_1,j^{\prime}_2,j_3,\dots ,j_{m^2+1}, j^{\prime}_{m^2+2},a} \cap \Xi_{m+1}\right)$$ are $(1-\frac{3}{m^2}-\frac{1}{q_m}-4m^2\e_m-\e,q_{m+1})$-Hamming apart from each other under the map $T$ with respect to the partition $\eta_m$.
	\end{claim}
	\begin{proof} We compute that under the map $\phi_{2q^s_m, \e_m}$, $3\leq s \leq m^2+1$, a set of the form $B^{(m)}_{j_1,\dots , j_{m^2+2},a}$ is mapped to
	\begin{align*}
	& \Big[ \frac{j_1}{q_m}+\frac{j_2}{2q^{3}_m}+ \dots +\frac{j_{s-1}}{2q^{s}_m}+\frac{j_{m^2+2}}{2q^{s+1}_m}, \frac{j_1}{q_m}+\frac{j_2}{2q^{3}_m}+ \dots +\frac{j_{s-1}}{2q^{s}_m}+\frac{j_{m^2+2}+1}{2q^{s+1}_m}\Big) \\
	&\times \Big[1-\frac{j_{s}}{q_m}- \dots -\frac{j_{m^2+1}}{q^{m^2+2-s}_m}-\frac{(a+1) \cdot q^s_m}{q_{m+1}}, 1-\frac{j_{s}}{q_m}- \dots -\frac{j_{m^2+1}}{q^{m^2+2-s}_m}-\frac{a \cdot q^s_m}{q_{m+1}} \Big).
	\end{align*}
	Hence, we see that if the action of $\phi_m$ on $\left[\frac{j_1}{q_m}+\frac{j_2}{2q^{3}_m},\frac{j_1}{q_m}+\frac{j_2+1}{2q^{3}_m}\right) \times [0,1]$ is different from $\phi_m$ on $\left[\frac{j^{\prime}_1}{q_m}+\frac{j^{\prime}_2}{2q^{3}_m},\frac{j^{\prime}_1}{q_m}+\frac{j^{\prime}_2+1}{2q^{3}_m}\right) \times [0,1]$, then a proportion of at most $\frac{1}{q_m}$ of the sets $B^{(m)}_{j_1,j_2,j_3,\dots ,j_{m^2+1}, j_{m^2+2},a}$ and $B^{(m)}_{j^{\prime}_1,j^{\prime}_2,j_3,\dots ,j_{m^2+1}, j^{\prime}_{m^2+2},a}$ are mapped into the same partition element of $\xi_m$ under $\phi_m$. By the combinatorics of $\phi_m$ coming from part (2) of Lemma \ref{lem:prob3}, on a proportion of at most $\frac{3}{m^2}$ of domains the actions of $\phi_m$ on $R^t_{\a_{m+1}}\circ H^{-1}_m(P_1)$ and $R^t_{\a_{m+1}}\circ H^{-1}_m(P_2)$, $0\leq t < q_{m+1}$, coincide. Thus, we obtain that under $T_m$ the points $P_1$ and $P_2$ are $(1-\frac{3}{m^2}-\frac{1}{q_m}-4m^2\e_m,q_{m+1})$-Hamming apart from each other with respect to the partition $\eta_m$. Since for any point $P\in H_m(B^{(m)}_{j_1,j_2,j_3,\dots ,j_{m^2+1}, j_{m^2+2},a}\cap \Xi_{m+1})$ the images $T^t(P)$ and $T^t_m(P)$ for any $t\leq q_{m+1}$ lie in the same element of $\eta_m$ by the definition of the safe domain in Remark \ref{rem:safeSmooth}, we conclude the claim for the map $T$ with the aid of condition (\ref{eq:condSafe4}).
\end{proof}
	
	This also motivates to consider the unions
	\begin{align*}
	\overline{B}^{(m)}_{k_1,k_2,k_3} = \bigcup^{(1-3\e)q_m}_{j_3=2\e_mq_m} \bigcup^{(1-2\e_m)q_m}_{j_4=2\e_mq_m} \dots \bigcup^{(1-2\e_m)q_m}_{j_{m^2+1}=2\e_mq_m} \bigcup^{\lfloor \frac{q_{m+1}}{2q^{m^2+2}_m} \rfloor -1}_{a=0} B^{(m)}_{k_1,k_2, j_3,j_4, \dots,j_{m^2+1},k_3,a}
	\end{align*}
	for $0\leq k_1<q_m$, $0\leq k_2 <2q^2_m$, and $2\e_mq_m \leq k_3 < (1-2\e_m)q_m$.
	
	In the next step, we let $n=m+1$ and $\tilde{\eta}_m \coloneqq \{\overline{B}^{(m)}_{k_1,k_2,k_3}\}$. As in the proof of Lemma \ref{lem:upperSmooth4} $\tilde{\eta}_m$ will serve as our ``target partition''. This time the width of such sets $\overline{B}^{(m)}_{k_1,k_2,k_3}$ is approximately $s_m \coloneqq \frac{1-3\e}{2q^3_m}$. Within $h^{-1}_{n}\left(\bigcup \overline{B}^{(m)}_{k_1,k_2,k_3}\right)$ we consider sets $B^{(n)}_{j_1, j_2,\vec{i},a}$ of the following form
	\begin{align*}
	L_n \cap & \Big[ \frac{j_1}{q_{n}}+\frac{j_2}{2q^3_{n}}+\frac{i_1  d_{n}+i_2 e_n +i_3 f_n} {2q^4_n}+\dots +\frac{i_{3n^2-5}  d_{n}+i_{3n^2-4} e_n +i_{3n^2-3} f_n} {2q^{n^2+2}_n}+\frac{a}{q_{n+1}}, \\
	& \ \frac{j_1}{q_{n}}+\frac{j_2}{2q^3_{n}}+\frac{i_1  d_{n}+i_2  e_n +i_3 f_n} {2q^4_n}+\dots +\frac{i_{3n^2-5}  d_{n}+i_{3n^2-4} e_n +i_{3n^2-3} f_n} {2q^{n^2+2}_n}+\frac{a+1}{q_{n+1}}\Big) & \\
	& \times \Big[\frac{i_{3n^2-2}}{q_m}+\frac{i_{3n^2-1}+i_{3n^2} \cdot  s_m}{b_{n}}, \frac{i_{3n^2-2}}{q_m}+\frac{i_{3n^2-1}+(i_{3n^2}+1) \cdot  s_m}{b_{n}} \Big),
	\end{align*}
	where $d_{n}\coloneqq \lfloor \frac{q_{n}}{q_m} \rfloor$, $e_n \coloneqq \lfloor \frac{q_n}{b_n} \rfloor$, and $f_n \coloneqq \lfloor \frac{s_m q_n}{b_n} \rfloor$.
	
\begin{claim}\label{cl:induction}
Suppose $P_1 \in H_n \left(B^{(n)}_{j_1, j_2,\vec{i},a}\cap \Xi_{n+1} \right)$ and $P_2 \in H_n \left(B^{(n)}_{j_1, j^{\prime}_2,\vec{i},a} \cap \Xi_{n+1} \right)$ with $j_2 \neq j^{\prime}_2 \mod q^2_n$, then we have that the Hamming distance with length $q_{n+1}$ between $P_1$ and $P_2$ under map $T$ with respect to $\eta_m$ is larger than $1-3\e$
\end{claim}
\begin{proof}
As in the proof of the Claim \ref{cl:starting} we make the following observation by direct computation: If $\phi_n$ on $\left[\frac{j_1}{q_n}+\frac{j_2}{2q^{3}_n},\frac{j_1}{q_n}+\frac{j_2+1}{2q^{3}_n}\right) \times [0,1]$ is different from $\phi_n$ on $\left[\frac{j_1}{q_n}+\frac{j^{\prime}_2}{2q^{3}_n},\frac{j_1}{q_n}+\frac{j^{\prime}_2+1}{2q^{3}_n}\right) \times [0,1]$, then a proportion of $\frac{1}{s_m}$ of the sets $B^{(n)}_{j_1, j_2,\vec{i},a}$ and $B^{(n)}_{j_1, j^{\prime}_2,\vec{i},a}$ are mapped into the same partition element of $\tilde{\eta}_m$ under $h_n=g_n \circ \phi_n$. For points $P_1 \in H_n \left(B^{(n)}_{j_1, j_2,\vec{i},a} \right)$ and $P_2 \in H_n \left(B^{(n)}_{j_1, j^{\prime}_2,\vec{i},a} \right)$ with $j_2 \neq j^{\prime}_2 \mod q^2_n$ we use part (2) of Lemma \ref{lem:prob3} again to see that on a proportion of at most $\frac{3}{n^2}$ of domains the actions of $\phi_n$ on $R^t_{\a_{n+1}}\circ H^{-1}_n(P_1)$ and $R^t_{\a_{m+1}}\circ H^{-1}_n(P_2)$, $0\leq t < q_{n+1}$, coincide. Combining both observations yields that under $h_n \circ R_{\alpha_{n+1}}\circ H^{-1}_n$ the points $P_1$ and $P_2$ are $(1-\frac{3}{n^2}-\frac{1}{s_m}-4n^2\e_n,q_{n+1})$-Hamming apart from each other with respect to the partition $\tilde{\eta}_m$. Since for any point $P\in H_n(B^{(n)}_{j_1, j_2,\vec{i},a}\cap \Xi_{n+1})$ the images $T^t(P)$ and $T^t_n(P)$ for any $t\leq q_{n+1}$ lie in the same element of $\eta_m$ by the definition of the safe domain in Remark \ref{rem:safeSmooth}, we conclude with the aid of the Claim \ref{cl:starting} and condition (\ref{eq:condSafe4}) that $P_1$ and $P_2$ are $\left(1-\e-\frac{1}{q_m}-\frac{1}{s_m}-\sum^{m+1}_{k=m}(\frac{3}{k^2}+4k^2\e_k),q_{n+1}\right)$-Hamming apart from each other under the map $T$ with respect to the partition $\eta_m$.
	
By induction we continue for $n>m+1$. Hereby, we complete the proof of the claim since $\sum^{\infty}_{k=m}(\frac{3}{k^2}+4k^2\e_k) <\e$ by condition (\ref{eq:deltaSmooth4}).
\end{proof}

In fact, Claim \ref{cl:induction} implies for every fixed $j_3$, $\vec{i}$ and $a$ that if $P_1 \in H_n \left(B^{(n)}_{j_1, j_2,\vec{i},a}\cap \Xi_{n+1} \right)$, $P_2 \in H_n \left(B^{(n)}_{j_3, j_4,\vec{i},a} \cap \Xi_{n+1} \right)$, then there exists at most one $j_4$ such that $P_1$ and $P_2$ are $\epsilon$-Hamming close with respect to $\eta_m$ with length $q_{n+1}$ under map $T$: Suppose that there are $j'_4\neq j_4$ such that for $P'_2 \in H_n \left(B^{(n)}_{j_3, j'_4,\vec{i},a} \cap \Xi_{n+1} \right)$ we also have $P_1$ and $P'_2$ are $\epsilon$-Hamming close with respect to $\eta_m$ with length $q_{n+1}$ under map $T$. This would imply that $P_2$ and $P'_2$ are $2\epsilon$-Hamming close with respect to $\eta_m$ with length $q_{n+1}$ under map $T$, which contradicts Claim \ref{cl:induction}.

As a result, we obtain the estimate of the measure of the $\epsilon$-Hamming ball which contains $P \in H_n \left(\Xi_{n+1} \right)$:
\begin{equation}
\mu(B_{\eta_m,q_{n+1}}(P,\epsilon))\leq\frac{1}{q_n^2}.
\end{equation}
Combining this with \eqref{eq:condSafe4}, we complete the proof of the lemma.

\end{proof}

\begin{theorem}\label{thm:weakMixingMeaInter}
There exists a weakly mixing $C^{\infty}$ AbC diffeomorphism such that
\begin{equation}
\uent^{\mu}_{a_m^{\operatorname{int1,8}}(t)}(T)=2.
\end{equation}
\end{theorem}

\begin{proof}
Since $\eta_l$ is a generating sequence of partitions, Proposition \ref{prop:generator} allows us to obtain the measure-theoretic slow entropy of $T$ by computing its slow entropy along $\eta_l$.


For any $m$ satisfying $q_n < m\leq l_n'q_{n}=q_{n}^{(n-1)^8}$, we get from Lemma \ref{lem:upperSmooth4} and \eqref{eq:intermediateScaleAtQ(n+1)} that
$$\frac{S_{\eta_l}^H(T,m,\e)}{a_{m}^{\text{int1,8}}(t)}\leq \frac{S_{\eta_l}^H(T,m,\e)}{a_{q_n}^{\operatorname{int1,8}}(t)}\leq \frac{C_{n-1}q_{n-1}^{2}}{q_{n-1}^{t}}.$$
For any $m$ with $q_{n}^{(n-1)^8}=l_n'q_n<m\leq q_{n+1}$, we have by Lemma \ref{lem:upperSmooth4} and \eqref{eq:intermediateScaleAtQ(n)} that
$$\frac{S_{\eta_l}^H(T,m,\e)}{a_{m}^{\text{int1,8}}(t)}\leq \frac{S_{\eta_l}^H(T,m,\e)}{a_{l_n'q_n}^{\operatorname{int1,8}}(t)}\leq \frac{C_{n}q_{n}^{2}}{\big[q_n\big]^{\frac{(n-1)^8t}{n^8}}}.$$
Hence, it is clear that $\uent^{\mu}_{a_m^{\operatorname{int1,8}}(t)}(T)\leq 2$.

On the other hand, for any $t\leq 2$, Lemma \ref{lem:lowerSmooth4} and \eqref{eq:intermediateScaleAtQ(n+1)} yield
$$\limsup_{m\to\infty}\frac{S_{\eta_l}^H(T,m,\e)}{a_m^{\text{int1,8}}(t)}\geq \limsup_{n\to\infty}\frac{S_{\eta_l}^H(T,q_{n+1},\e)}{a_{q_{n+1}}^{\text{int1,8}}(t)} \geq  \limsup_{n\to\infty}\frac{C_nq_{n}^{2}}{q_n^{t}}>0.$$
As a result, we have $\uent^{\mu}_{a_m^{\operatorname{int1,8}}(t)}(T)\geq 2$. Combining all these steps, we complete the proof.

\end{proof}

\section{Regularity of AbC constructions and slow entropy}

It appears from the above examples that there is a connection between the speed of convergence of the AbC method and slow entropy of the limit diffeomorphism. Speed of convergence of the AbC method is in its turn related to the regularity of the AbC diffeomorphism, with higher the regularity, higher is the requirement of speed of convergence. Here we formulate some results and questions in an attempt to further understand this connection.

\subsection{Measure-theoretical slow entropy and regularity of AbC constructions}\label{sec:vanishSuperSmooth}

We begin with a proof of Theorem \ref{thm:vanishSuperSmooth} to show that for $C^\infty$ AbC diffeomorphisms the upper measure-theoretical slow entropy is always zero at polynomial scale.

By \cite[Lemma 4.5]{BKW1} we have for $\frac{\varepsilon}{2} l_nq_n\leq L<\frac{\varepsilon}{2} l_{n+1}q_{n+1}$ that
$$N\left(\eta_m, L,\varepsilon\right)\leq \frac{2}{\varepsilon^2}k_n q_n s_n.$$
The conjugation map $h_n$ approximating the permutation of $\frac{1}{k_nq_n} \times \frac{1}{s_n}$ rectangles smoothly has to satisfy $\norm{h_n}_u \geq (\max(k_nq_n,s_n))^u$. Along a sequence $u_n \in \N$ growing to infinity the numbers $l_n$ have to satisfy $l_n \geq n^2 \norm{h_n}_{u_n}$. Hence:
\begin{align*}
\uent^{\mu}_{n^t}(T) \leq \frac{\frac{2}{\varepsilon^2}k_n q_n s_n}{(\frac{\varepsilon}{2} l_nq_n)^t} \leq \frac{\frac{2}{\varepsilon^2}  (\max\left(k_nq_n,s_n \right))^2 }{(\frac{\varepsilon}{2} q_n \cdot n^2 \cdot (\max(k_nq_n,s_n))^{u_n} )^t},
\end{align*}
whose limit is zero for all $t>0$.

\subsection{Topological slow entropy and regularity of AbC constructions}

The computations we made in this article prompt the following questions:

\noindent\textbf{Question 1:} Is it possible for a $C^\infty$ AbC diffeomorphism to have finite non-zero upper topological slow entropy in the polynomial scale?

\noindent\textbf{Question 2:} What is an appropriate family of scaling functions for the slow entropy of $C^\omega$ (real-analytic) AbC diffeomorphisms? In particular, is it possible for a $C^\omega$ AbC diffeomorphism to have non-zero upper slow entropy in the log scale?

We point out that estimates from \cite{Ku-rigidity} indicate that the upper topological slow entropy for analytic AbC diffeomorphisms can be non-zero at $\left(\ln(\ln(n))\right)^t$ scale, while the ones in \cite{Ba-Ku} can be non-zero at the far slower $\left(\ln(\ln(\ln(n)))\right)^t$ scale.

\end{document}